\documentclass[10pt,a4paper,final]{amsart}

\usepackage{ifdraft}

\ifdraft{\usepackage[draft]{showkeys}}{\usepackage[final]{showkeys}}

\usepackage{fontenc}
\usepackage[utf8]{inputenc}

\usepackage{amsmath,amsfonts,mathscinet,eucal,amsthm,amssymb,bm,extarrows,upgreek,tensor,mathrsfs,textcomp,comment,mathtools,bbm,braket}

\mathtoolsset{mathic}
\usepackage[shortlabels]{enumitem}

\usepackage[usenames,dvipsnames]{xcolor} \usepackage{tikz}
\usetikzlibrary{patterns,matrix,through,arrows,decorations.pathreplacing,decorations.markings,decorations.pathmorphing,shadows,shapes.geometric,positioning,calc,backgrounds,fit}

\def \myweightstyle {}
\def \myweightx {0.09}
\def \myweighty {0.2}
\newcommand{\mydrawdown}[1]%
  {\draw [\myweightstyle] #1 -- ++(\myweightx,\myweighty);%
    \draw [\myweightstyle] #1 -- ++(-\myweightx,\myweighty);}
\newcommand{\mydrawup}[1]%
  {\draw [\myweightstyle] #1 -- ++(-\myweightx,-\myweighty);%
    \draw [\myweightstyle] #1 -- ++(\myweightx,-\myweighty);}

\tikzset{anchorbase/.style={baseline={([yshift=-0.5ex]current bounding box.center)}},
  int/.style={thick},
  cross line/.style={preaction={draw=white,line width=6pt,-}},
  wall/.style={thin,double,blue},
  middlearrow/.style={postaction=decorate,decoration={markings,mark=at
    position .55 with {\arrow{stealth};}}},
  middlearrowrev/.style={postaction=decorate,decoration={markings,mark=at
    position .55 with {\arrowreversed{stealth};}}},
  smallnodes/.style={every node/.style={font=\normalsize}},
  ev/.style={shape=rectangle, draw}
}

\usepackage{todonotes} 

\usepackage[bbgreekl]{mathbbol}

\DeclareSymbolFontAlphabet{\mathbb}{AMSb}
\DeclareSymbolFontAlphabet{\mathbbol}{bbold}
\DeclareMathAlphabet{\mathpzc}{OT1}{pzc}{m}{it}

\DeclareSymbolFont{usualmathcal}{OMS}{cmsy}{m}{n}
\DeclareSymbolFontAlphabet{\mathucal}{usualmathcal}

\numberwithin{equation}{section}

\newtheoremstyle{myplain} {6pt plus 6pt minus 2pt}
{6pt plus 6pt minus 2pt}
{\itshape}
{}
{\bfseries}
{.}
{.5em}
{}

\theoremstyle{myplain}
\newtheorem{theorem}{Theorem}[section]
\newtheorem*{theorem*}{Theorem}
\newtheorem{lemma}[theorem]{Lemma}
\newtheorem{prop}[theorem]{Proposition}
\newtheorem{corollary}[theorem]{Corollary}
\newtheorem{conjecture}[theorem]{Conjecture}

\newtheoremstyle{mydefinition} {6pt plus 6pt minus 2pt}
{6pt plus 6pt minus 2pt}
{\itshape}
{}
{\bfseries}
{.}
{.5em}
{}

\theoremstyle{mydefinition}
\newtheorem{definition}[theorem]{Definition}

\newtheoremstyle{myexample} {6pt plus 6pt minus 2pt}
{6pt plus 6pt minus 2pt}
{}
{}
{\scshape}
{.}
{.5em}
{}

\theoremstyle{myexample}
\newtheorem{example}[theorem]{Example}

\newtheoremstyle{myremark} {6pt plus 6pt minus 2pt}
{6pt plus 6pt minus 2pt}
{}
{}
{\scshape}
{.}
{.5em}
{}

\theoremstyle{myremark}
\newtheorem{remark}[theorem]{Remark}

\newcommand*{\MyDef}{\mathrm{def}}
\newcommand*{\eqdefU}{\ensuremath{\mathop{\overset{\MyDef}{=}}}}
\newcommand*{\eqdef}{\mathop{\overset{\MyDef}{\resizebox{\widthof{\eqdefU}}{\heightof{=}}{=}}}}
\newcommand{\abs}[1]{\left|#1\right|}
 
\newcommand{\canonicalHecke}{{\underline H}}

\newcommand{\frakg}{\mathfrak{g}}

\newcommand{\adjunction}{\dashv}
\newcommand{\bbB}{{\mathbb{B}}}
\newcommand{\bbS}{{\mathbb{S}}}

\newcommand{\bolda}{{\boldsymbol{a}}}
\newcommand{\boldb}{{\boldsymbol{b}}}
\newcommand{\C}{\mathbb{C}}
\newcommand{\R}{\mathbb{R}}
\newcommand{\calA}{{\mathcal{A}}}
\newcommand{\calC}{{\mathcal{C}}}
\newcommand{\calD}{{\mathcal{D}}}
\newcommand{\calE}{{\mathcal{E}}}
\newcommand{\calF}{{\mathcal{F}}}
\newcommand{\calL}{{\mathcal{L}}}
\newcommand{\calP}{{\mathcal{P}}}
\newcommand{\calQ}{{\mathcal{Q}}}
\newcommand{\calS}{{\mathcal{S}}}
\newcommand{\calT}{{\mathcal{T}}}
\newcommand{\catO}{\mathcal{O}}
\DeclareMathOperator{\coker}{coker}
\newcommand{\down}{{\mathord\vee}}
\renewcommand{\emptyset}{\varnothing}
\renewcommand{\epsilon}{\varepsilon}
\DeclareMathOperator{\End}{End}
\DeclareMathOperator{\Ext}{Ext}
\newcommand{\fraka}{{\mathfrak{a}}}
\newcommand{\frakb}{{\mathfrak{b}}}
\newcommand{\frakf}{{\mathfrak{f}}}
\newcommand{\frakh}{{\mathfrak{h}}}
\newcommand{\fraki}{{\mathfrak{i}}}
\newcommand{\frakj}{{\mathfrak{j}}}
\newcommand{\frakl}{{\mathfrak{l}}}
\newcommand{\frakn}{{\mathfrak{n}}}
\newcommand{\frakp}{{\mathfrak{p}}}
\newcommand{\frakq}{{\mathfrak{q}}}
\newcommand{\frakQ}{{\mathfrak{Q}}}
\newcommand{\fraks}{{\mathfrak{s}}}
\newcommand{\frakz}{{\mathfrak{z}}}
\newcommand{\frakE}{{\mathfrak{E}}}

\newcommand{\gl}{\mathfrak{gl}}
\newcommand{\gmod}[1]{#1\mathrm{-gmod}}
\newcommand{\id}{\mathrm{id}}
\DeclareMathOperator{\image}{Im} 
\DeclareMathOperator{\Ind}{Ind}
\newcommand{\into}{\hookrightarrow}
\newcommand{\K}{\mathbb{K}}
\newcommand{\len}{\ell}
\newcommand{\lmod}[1]{#1\mathrm{-mod}}
\newcommand{\mapto}{\rightarrow}
\newcommand{\N}{\mathbb{N}}
\renewcommand{\phi}{\varphi}
\DeclareMathOperator{\rad}{rad}
\newcommand{\rgmod}[1]{\mathrm{gmod-}#1}
\newcommand{\rmod}[1]{\mathrm{mod-}#1}

\newcommand{\scrP}{{\mathscr{P}}}
\newcommand{\sfb}{{\mathsf{b}}}

\newcommand{\sfi}{{\mathsf{i}}}
\newcommand{\sfj}{{\mathsf{j}}}
\newcommand{\sfP}{{\mathsf{P}}}
\newcommand{\sfQ}{{\mathsf{Q}}}
\newcommand{\sfz}{{\mathsf{z}}}
\newcommand{\suchthat}{\,|\,} 
\newcommand{\surto}{\twoheadrightarrow}
\newcommand{\up}{{\mathord\wedge}}
\DeclareMathOperator{\Hom}{Hom} \DeclareMathOperator{\vectorspan}{span}
\newcommand{\Z}{\mathbb{Z}}

\newcommand{\acts}{\mathop{\,\;\raisebox{1.7ex}{\rotatebox{-90}{\(\circlearrowright\)}}\;\,}}
\newcommand{\actsb}{\mathop{\,\;\raisebox{0ex}{\rotatebox{90}{\(\circlearrowleft\)}}\;\,}}

\newcommand{\ucalH}{{\mathucal H}}
\newcommand{\ucalM}{{\mathucal M}}
\newcommand{\ucalN}{{\mathucal N}}

\newcommand{\ucalD}{{\mathucal D}}

\newcommand{\trasT}{{\mathbb{T}}}

\newcommand{\Uqgl}{U_q}
\newcommand{\Uqgle}{U_q(\mathfrak{gl}(1|1))}

\newcommand{\compm}{{\mathbbol{m}}}
\newcommand{\compn}{{\mathbbol{n}}}

\newcommand{\blank}{{\mathord{\bullet}}}

\newcommand{\webjoin}{\tikz[baseline,scale=0.9]{\draw (0,0) -- (1ex,1ex) -- (2ex,0); \draw (1ex,1ex) -- (1ex,2ex);}}
\newcommand{\webjoinmultiple}{\tikz[baseline,scale=0.9]{\draw (0,0) -- (1ex,1ex); \draw (2ex,0) -- (1ex,1ex); \draw (0.5ex,0) -- (1ex,1ex); \draw (1ex,0) -- (1ex,1ex);\draw (1.5ex,0) -- (1ex,1ex);\draw[thick] (1ex,1ex) -- (1ex,2ex);}}
\newcommand{\websplit}{\tikz[baseline,scale=0.9]{\draw (0,2ex) -- (1ex,1ex) -- (2ex,2ex); \draw (1ex,1ex) -- (1ex,0);}}
\newcommand{\websplitmultiple}{\tikz[baseline,scale=0.9]{\draw (0,2ex) -- (1ex,1ex); \draw (0.5ex,2ex) -- (1ex,1ex); \draw (1ex,2ex) -- (1ex,1ex);\draw (1.5ex,2ex) -- (1ex,1ex);\draw (2ex,2ex) -- (1ex,1ex); \draw[thick] (1ex,1ex) -- (1ex,0);}}

\newcommand{\short}{{\mathrm{short}}}
\newcommand{\longrep}{{\mathrm{long}}}

\renewcommand{\epsilon}{\varepsilon}

\DeclareMathOperator{\Trace}{Tr}
\DeclareMathOperator{\Add}{Add}
\newcommand{\der}{\scr D}
\newcommand{\derL}{\mathbb L}

\newcommand{\canon}{{\mathbin{\diamondsuit}}}
\newcommand{\dualcanon}{{\mathbin{\heartsuit}}}
\newcommand{\dualstan}{{\mathbin{\clubsuit}}}

\newcommand{\STL}{{\mathrm{STL}}}

\newcommand{\quantumq}{{\boldsymbol{\mathrm{q}}}}
\newcommand{\qbin}[2]{\genfrac{[}{]}{0pt}{}{#1}{#2}}

\newcommand{\catRep}{{\mathsf{Rep}}}
\newcommand{\catWeb}{{\mathsf{Web}}}
\newcommand{\Cat}{{\catO\mathsf{Cat}}}
\newcommand{\funcT}{{\mathscr{T}}}
\newcommand{\funcF}{{\mathscr{F}}}
\newcommand{\bK}{{\boldsymbol{\mathrm{K}}}}

\newcommand{\symm}{\odot}
\newcommand{\pres}{\text{-}\mathrm{pres}}
\newcommand{\oDelta}{{\overline{\Delta}}}
\newcommand{\catOZ}{\prescript{\Z}{}{\catO}}

\newcommand{\boldeta}{{\boldsymbol{\eta}}}
\newcommand{\boldgamma}{{\boldsymbol{\gamma}}}

\newcommand{\sgnmodule}{{\mathrm{sgn}}}
\newcommand{\trvmodule}{{\mathrm{trv}}}

\newcommand{\counit}{\boldsymbol{\mathrm{u}}}

\newcommand{\twoheadlongrightarrow}{\relbar\joinrel\twoheadrightarrow}

\usepackage{wrapfig} \usepackage{subcaption}

\usepackage{hyperref,bookmark}

\hypersetup{
  colorlinks = false,
  urlcolor = false,
  pdfauthor = {Antonio Sartori},
  pdfkeywords = {Categorification, Lie superalgebras, gl(1|1), Category O, induced Hecke modules, properly stratified algebras},
  pdftitle = {Categorification of tensor powers of the vector representation of Uq(gl(1|1))},
  pdfsubject = {},
  pdfpagemode = UseNone,
  bookmarksopen = true,
  bookmarksopenlevel = 3,
  pdfdisplaydoctitle = true }

\title[Categorification of representations of $U_q(\gl(1|1))$]{Categorification of tensor powers of the vector representation of $U_q(\gl(1|1))$}

\author{Antonio Sartori}
\address{Mathematisches Institut\\%
Albert-Ludwigs-Universität Freiburg\\%
Eckerstraße 1\\%
79104 Freiburg im Breisgau\\%
Germany}
\email{antonio.sartori@math.uni-freiburg.de}
\urladdr{http://home.mathematik.uni-freiburg.de/asartori}
\date{\today}
\thanks{This work has been supported by the Graduiertenkolleg 1150, funded by the Deutsche Forschungsgemeinschaft.}

\begin{document}

\begin{abstract}
  We consider the monoidal subcategory of finite dimensional
  representations of \(\Uqgle\) generated by the vector representation,
  and we provide a diagram calculus for the intertwining operators,
  which allows to compute explicitly the canonical basis. We construct
  then a categorification of these representations and of the action
  of both \(\Uqgle\) and the intertwining operators using subquotient
  categories of the BGG category \(\catO(\gl_n)\).
\end{abstract}

\maketitle

\setcounter{tocdepth}{1}
\tableofcontents

\section{Introduction}
\label{sec:introduction}

The \emph{Jones polynomial} is a classical invariant of
links in \(\R^3\) defined using the vector representation of
the Lie algebra \(\fraks\frakl_2\) (or more precisely of the
quantum algebra \(U_q(\fraks\frakl_2)\)). In his fundamental
paper
\cite{MR1740682}, Khovanov constructed a graded
homology theory for links whose graded Euler characteristic
is the Jones polynomial. Khovanov homology has two main
advantages over the Jones polynomial: first, it has been
proven to be a finer invariant and second, it has values in
a category of complexes and it also assigns to cobordisms
between links chain maps between chain complexes.  This
categorical approach to classical invariants is often called
\emph{categorification}. Khovanov's work raised great
interest in categorification, and since then a
categorification program for  representations of more
general semisimple Lie algebras and even Kac-Moody algebras
has been developed by several authors and motivated various
generalizations (see for example \cite{MR2305608},
\cite{MR2567504}, \cite{2013arXiv1309.3796W},
\cite{MR2525917}, \cite{MR2763732},
\cite{2008arXiv0812.5023R}). The main
tools in all these works come from \emph{representation
  theory} and geometry related to it.

Another very important invariant of knots is the \emph{Alexander
  polynomial}  \cite{MR1501429}, which is much older than the Jones
polynomial. Originally defined using the topology of the knot
complement, the Alexander polynomial is not the quantum invariant
corresponding to some complex semisimple Lie algebra, like the Jones
polynomial. Instead, it can be defined using the representation theory
of the general Lie \emph{super}algebra \(\gl(1|1)\) (or, more precisely,
its quantum enveloping superalgebra \(U_q(\gl(1|1))\), see
\cite{MR1881401}, \cite{MR2255851}, \cite{2013arXiv1308.2047S}; alternatively, one can use the
quantum enveloping algebra \(U_q(\mathfrak{sl}_2)\) where \(q\) is a root
of unity, but we will not consider this approach). A categorification
of the Alexander polynomial exists, but comes from a very different
area of mathematics: a homology theory, known as Heegard-Floer
homology, whose Euler characteristic gives the Alexander polynomial,
has been developed using symplectic geometry
\cite{2005math.....12286O}, \cite{MR2372850}.  This homology
theory, however, does not have an interpretation or a counterpart in
representation theory yet.

The present work is motivated by the attempt to
construct/understand categorifications of Lie superalgebras
(and hopefully a categorification of the Alexander
polynomial) using tools from representation theory. In fact,
there are only a few other recent works studying
representation theoretical categorifications of Lie
superalgebras and related structures
\cite{2010arXiv1007.3517K}, \cite{MR3280041},
\cite{2014arXiv1406.1676K}, \cite{MR3255459}. We hope that this paper can be a starting point for
a categorification program for \(\gl(1|1)\), beginning with a
categorification of tensor powers of the vector
representation and of their subrepresentations. We point out
that a counterpart of our construction in the setting of symplectic and contact 
geometry has been developed by Tian
\cite{2012arXiv1210.5680T}, \cite{2013arXiv1301.3986T}.

Our main result can be summarized as follows:

\begin{theorem*}[See Theorems~\ref{thm:1} and \ref{thm:2}]
  Let \(V\) be the vector representation of \(\Uqgle\), fix \(n>0\) and consider the commuting actions of \(\Uqgle\) and of the Hecke algebra \(\ucalH_n = \ucalH(\bbS_n)\) on \(V^{\otimes n}\):
  \begin{equation}\tag{\maltese}\label{eq:172}
    \Uqgle \acts V^{\otimes n} \actsb \ucalH_n.
  \end{equation}
  For each \(n > 0\) there exists a triangulated category \(\calD^\nabla \calQ(\compn)\) whose Grothendieck group is isomorphic to \(V^{\otimes n}\) and two families of endofunctors \(\{\calE, \calF\}\) and \(\{\calC_i \suchthat i=1,\ldots,n-1\}\) which commute with each other and which on the Grothendieck group level give the actions \eqref{eq:172} of \(\Uqgle\) and of the Hecke algebra \(\ucalH_n\) on \(V^{\otimes n}\) respectively:
  \begin{equation*}
    [\calE], [\calF] \acts \bK^{\C(q)} (\calD^\nabla \calQ(\compn)) \actsb [\calC_i].
  \end{equation*}
\end{theorem*}

A remarkable property (and also a complication) of the
finite-dimensional representations of \(\gl(1|1)\) (and more
generally of \(\gl(m|n)\)) is that they need not be
semisimple. For example, if \(V\) is the vector
representation of \(\gl(1|1)\), then \(V \otimes V^*\) is a
four-dimensional indecomposable non-irreducible
representation. It is not clear how the lack of
semisimplicity should affect the categorification, but it is
plausible that this provides additional difficulties. What
we can categorify in the present work is indeed only a
semisimple monoidal subcategory of the representations of
\(\gl(1|1)\), that contains the vector representation \(V\),
but not its dual \(V^*\). We remark that we will develop all
the details for the quantum version, but in order to keep
this introduction technically clean we avoid to introduce the quantum
enveloping algebra now.

Our categorification relies on a very careful analysis of the representation theory of \(\gl(1|1)\) and its \emph{canonical basis} (based on \cite{MR2491760}). In the categorification, indecomposable projective modules correspond to canonical basis elements, that we can compute explicitly via a diagram calculus, analogous to the diagram calculus developed in \cite{MR1446615} for \(\fraks\frakl_2\). The key-tool for our construction is the so-called \emph{super Schur-Weyl duality}  \eqref{eq:172} (originally studied in \cite{MR884183} and \cite{MR735715}): the symmetric group algebra \(\C[\bbS_n]\) acts on the tensor power \(V^{\otimes n}\), and this action commutes with the action of \(\gl(1|1)\).
The weight spaces of \(V^{\otimes n}\) are modules for \(\C[\bbS_n]\), and explicitly they are isomorphic to mixed induced modules of the form
\begin{equation}\tag{\dag}\label{eq:177}
(  \trvmodule_{\bbS_k} \boxtimes \sgnmodule_{\bbS_{n-k}}) \otimes_{\C[\bbS_k \times \bbS_{n-k}]} \C[\bbS_n].
\end{equation}
In particular, they can be equipped with a canonical basis coming from the action of symmetric group algebra.
A crucial point is the following observation:
\begin{theorem*}[See Proposition \ref{prop:19}]
  Lusztig's canonical basis of \(V^{\otimes n}\), defined using the action of \(\gl(1|1)\), agrees with the canonical basis defined in term of the symmetric group action.
\end{theorem*}

This Schur-Weyl duality is strictly related to a version of super skew Howe
duality that connects representations of \(\gl(1|1)\), or more generally
\(\gl(m|n)\), with representations of \(\gl_N\) \cite{MR1847665}.
In fact, the whole categorification process we develop works more generally for tensor powers of the vector representation of \(\gl(m|n)\). We will sketch the main ideas for the general case in Section~\ref{sec:howe-duality-glm}, while in the rest of the paper we will restrict to \(\gl(1|1)\) and work out the details in this case; this turns out to be already very laborious. To develop the \(\gl(1|1)\)--categorification theory we will use super Schur-Weyl duality instead of Howe duality, and hence reduce the problem to symmetric group categorification. The two approaches are equivalent, but we personally prefer to work out the detail based on the first one.

The fundamental tool used in our construction is the BGG category \(\catO\) \cite{MR2428237}, which plays already an important role in many other representation theoretical categorifications. In particular, we will construct a categorification of tensor powers of \(V\) and of their subrepresentations using some subquotient categories of \(\catO(\gl_n)\). These categories are built in two steps: first one takes a parabolic subcategory and then a ``\(\frakq\)--presentable'' quotient; the two steps can be reversed, and one gets the same result. The process is sketched by the following picture, which is also helpful to remember how we index our categories:

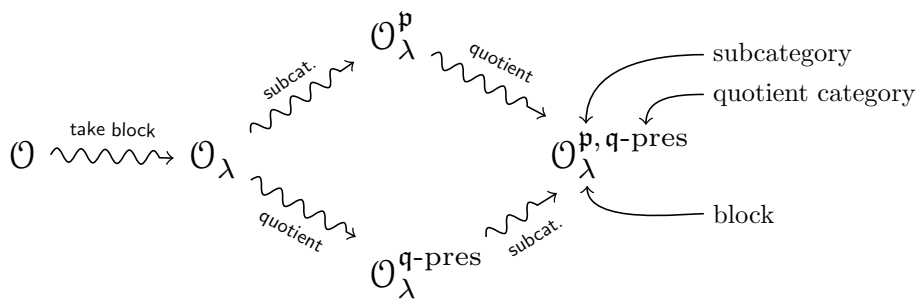
\begin{equation*}
\scalebox{1.5}{\begin{tikzpicture}[quoziente/.style={->,decorate,decoration={coil,aspect=0,segment length=1.4ex,amplitude=0.4ex}},processo/.style={sloped,scale=0.5, font=\sffamily}]
  \node (usO) at (-13.5em,0) {\(\catO\)};

  \node (cO) [inner sep=0pt]  at (-9em,0) {\(\catO\)};
  \node (clambda) [base right= -0.1ex of cO, inner sep=0pt, yshift=-0.6ex] {\(\scriptstyle \lambda\)};

  \node (cc) [fit= (cO) (clambda)] {};

  \node (aO) [inner sep=0pt]  at (-4.5em,3em) {\(\catO\)};
  \node (ap) [ base right = -0.1ex of aO, inner sep=0pt, yshift=1.3ex] {\(\scriptstyle\frakp\)}; 
  \node (alambda) [base right= -0.1ex of aO, inner sep=0pt, yshift=-0.6ex] {\(\scriptstyle \lambda\)};
  \node (aa) [fit =(aO) (ap) (alambda)] {};

  \node (bO) [inner sep=0pt]  at (-4.5em,-3em) {\(\catO\)};
  \node (bq) [ base right= -0.1ex of bO, inner sep=0pt,yshift=1.3ex] {\(\scriptstyle \frakq\pres\)};
  \node (blambda) [base right= -0.1ex of bO, inner sep=0pt, yshift=-0.6ex] {\(\scriptstyle \lambda\)};
  \node (bb) [fit = (bO) (bq) (blambda)] {};

  \node [scale=0.66,anchor=west] (pt) at (3.5em,2.5 em) {subcategory};
  \node [scale=0.66,anchor=west] (qt) at (3.5em,1.5em)  {quotient category};
  \node [scale=0.66,anchor=west] (lambdat) at (3.5em,-1.5em) {block};
  \node (O) [inner sep=0pt]  at (0,0) {\(\catO\)};
  \node (p) [ base right = -0.1ex of O, inner sep=0pt, yshift=1.3ex] {\(\scriptstyle\frakp\)}; 
  \node (p1) [ base right= -0.2ex of p, inner sep=0pt] {\(\scriptstyle,\)};
  \node (q) [ base right= 0 of p1, inner sep=0pt] {\(\scriptstyle \frakq\pres\)};
  \node (lambda) [base right= -0.1ex of O, inner sep=0pt, yshift=-0.6ex] {\(\scriptstyle \lambda\)};
  \node (OO) [fit = (O) (p) (p1) (q) (lambda)] {};
  \draw[->] (pt.west) .. controls +(left:1.4em) and +(up:1.4em) .. ([yshift=0.4ex]p.north) ;
  \draw[->] (qt.west) .. controls +(left:0.7em) and +(up:0.7em) ..  ([yshift=0.4ex]q.north) ;
  \draw[->] (lambdat.west) .. controls +(left:0.5em) and +(down:1em) ..  ([yshift=-0.4ex]lambda.south) ;

  \draw[quoziente] (usO) to node[above=0.8ex,processo] {take block} (cc);
  \draw[quoziente] (cc) to node[above=0.6ex,processo] {subcat.} (aa);
  \draw[quoziente] (aa) to node[above=0.6ex,processo] {quotient} (OO);
  \draw[quoziente] (cc) to node[below=0.6ex,processo] {quotient} (bb);
  \draw[quoziente] (bb) to node[below=0.6ex,processo] {subcat.} (OO);
\end{tikzpicture}}
\end{equation*}
We will give the precise setup and definitions and discuss the technical Lie-theoretical details in Section~\ref{sec:category-cato}.

The construction of these subquotient categories is
motivated by the following. Usually a semisimple module
\(M\) is categorified via some abelian category
\(\calC\). Now, \(M\) decomposes as direct sum of simple
modules, but the category \(\calC\) is not supposed to
decompose into blocks according to the decomposition of
\(M\). This is indeed one of the main points of the
categorification: we want \(\calC\) to have more structure
than \(M\). When \(M\) is equipped with some \emph{canonical
  basis}, the submodules generated by canonical basis
elements in \(M\) give a filtration of \(M\) (but not a
decomposition!); this corresponds to the filtration of
\(\calC\) with subcategories. This principle has been
applied in \cite{MR2450613} to categorify induced modules
for the symmetric group: the category \(\catO_0(\gl_n)\) is
well-known to be a categorification of the regular
representation of the symmetric group \(\bbS_n\); these
induced modules for the symmetric group are direct summands
of the regular representation of \(\C[\bbS_n]\); hence they
can be categorified via subquotient categories of
\(\catO_0(\gl_n)\).

In particular, \cite{MR2450613} provide some categories,
which we denote by \(\calQ_k(\compn)\), categorifying the
induced modules \eqref{eq:177}, and define on them a
categorical action of \(\C[\bbS_n]\) using translation
functors. To categorify \(V^{\otimes n}\) we take the direct
sum of all these categories \(\calQ_k(\compn)\) for
\(k=0,\ldots,n\).  In addition, we consider also the
corresponding singular blocks \(\calQ_k(\bolda)\) of the
same subquotients categories. Note that singular blocks do
not appear in \cite{MR2450613} since they do not provide
categorifications of \(\C[\bbS_n]\)--modules; in our
picture, they categorify subrepresentations of \(V^{\otimes
  n}\). The translation functors of category
\(\catO(\gl_n)\) restrict to all these subcategories
\(\calQ_k(\bolda)\) and categorify the action of the
intertwining operators of the \(\gl(1|1)\)--action.

We remark that the categories \(\calQ_k(\bolda)\) have a natural grading (inherited from the Koszul grading on \(\catO(\gl_n)\)) and all the functors we consider are actually graded functors between these categories. In fact, this grading was used in \cite{MR2450613} to get an action of the Hecke algebra instead of the symmetric group algebra for the induced modules \eqref{eq:177}. As a result, the categorification lifts to a categorification of representations of the quantum enveloping superalgebra \(\Uqgle\). We will work out all the details in the graded setting.

What is left to complete the picture is to define functors
that categorify the action of \(\gl(1|1)\) itself. There is
a natural way to define adjoint functors \(\calE\) and
\(\calF\) between \(\calQ_k(\bolda)\) and
\(\calQ_{k+1}(\bolda)\), which portend to categorify the
action of the generators \(E\) and \(F\) of
\(U(\gl(1|1))\). Although \(\calE\) is exact, \(\calF\) is
only right exact in general, and we need to derive our
categories and functors in order to have an action on the
Grothendieck groups. However, the following problem
arises. The categories we consider are equivalent to
categories of modules over some finite-dimensional
algebras. Unfortunately, these algebras are not always
quasi-hereditary; in general they are only \emph{properly
  stratified} (the definition of standardly and properly
stratified algebras has been modeled to describe the
properties of some generalized parabolic subcategories of
\(\catO\), introduced by \cite{MR1921761}, that include as
particular cases the categories that we consider). A
properly stratified algebra does not have in general finite
global dimension (this happens if and only if the algebra is
quasi-hereditary). As a consequence, finite projective
resolutions do not always exist, and we are forced to
consider unbounded derived categories. But the Grothendieck
groups of these unbounded derived categories vanish by some
Eilenberg-swindle argument \cite{MR2223266}.  A workaround
to this problem has been developed in \cite{pre06137854},
using the additional structure of a \emph{mixed Hodge structure},
which in our case is given by the grading. Given a graded
abelian category, \cite{pre06137854} define a proper
subcategory of the left unbounded derived category of graded
modules; this subcategory is big enough to contain
projective resolutions, but small enough to prevent the
Grothendieck group to vanish. In particular, the
Grothendieck group of this triangulated subcategory is a
\(q\)--adic completion of the Grothendieck group of the
original graded abelian category. We will describe in detail
how the categories we consider and the functors \(\calE\)
and \(\calF\) can be derived using these techniques.

Of course at this point one would like to understand and describe the involved categories \(\calQ_k(\bolda)\) explicitly. Very surprisingly (at least for us), this is indeed possible. To give an idea, let us present the categorification of \(V^{\otimes 2}\). First, we notice that \(V^{\otimes 2}\) has a weight space decomposition as given by the following picture:
\begin{equation*}
\begin{tikzpicture}
  \node (Z1) at (0,-1) {\(\sgnmodule_{\bbS_2}\)};
  \node (Z2) at (4,-1) {\(\C[\bbS_2]\)};
  \node (Z3) at (8,-1) {\(\trvmodule_{\bbS_2}\)};
  \node[rotate=270] at (0,-0.5) {\(\cong\)};
  \node[rotate=270] at (4,-0.5) {\(\cong\)};
  \node[rotate=270] at (8,-0.5) {\(\cong\)};
  \node (A1) at (0,0) {\(\big(V^{\otimes 2}\big)_0\)};
  \node (A2) at (4,0) {\(\big(V^{\otimes 2}\big)_1\)};
  \node (A3) at (8,0) {\(\big(V^{\otimes 2}\big)_2\)};
  \draw[->] ([yshift=0.4ex]A1.east) [bend left=15] to node[above] {\(E\)} ([yshift=0.4ex]A2.west);
  \draw[->] ([yshift=0.4ex]A2.east) [bend left=15] to node[above] {\(E\)} ([yshift=0.4ex]A3.west);
  \draw[->] ([yshift=-0.4ex]A3.west) [bend left=15] to node[below] {\(F \)} ([yshift=-0.4ex]A2.east);
  \draw[->] ([yshift=-0.4ex]A2.west) [bend left=15] to node[below] {\(F\)} ([yshift=-0.4ex]A1.east);
\end{tikzpicture}
\end{equation*}
where \(E\) and \(F\) are generators of \(\gl(1|1)\) and the vertical isomorphisms are isomorphisms of \(\C[\bbS_2]\)--modules. We let \(R=\C[x]/(x^2)\) and \(A = \End_R(\C \oplus R)\). The algebra \(A\) can be identified with the path algebra of the quiver
\begin{equation*}
  \begin{tikzpicture}[baseline=(A.base)]
    \node (A)  at (0,0) {\(1\)};
    \node (B)  at (1,0) {\(2\)};
    \draw[->] (A) [bend left] to node[above,font=\small] {\(a\)} (B);
    \draw[->] (B) [bend left] to node[below,font=\small] {\(b\)} (A);
  \end{tikzpicture}
  \qquad \text{with the relation } ba=0.
\end{equation*}
We denote by \(e_1\) and \(e_2\) the two idempotents corresponding to the vertices of the quiver. Let us identify  \(\C\) with \(A/Ae_1A\) and notice that \(\C\) becomes then naturally an \((A,\C)\)--bimodule. Moreover, notice that \(R\) is naturally isomorphic to the endomorphism ring of the projective module \(Ae_2\), so that we can consider \(A e_2\) as an \((A,R)\)--bimodule. The categorification of \(V^{\otimes 2}\) is then given by the following picture:
\begin{equation*}
\begin{tikzpicture}
  \node (Z1) at (0,1) {\(\catO^\frakp_0(\gl_2)\)};
  \node (Z2) at (4,1) {\(\catO_0(\gl_2)\)};
  \node (Z3) at (8,1) {\(\catO^{\frakp\pres}_0(\gl_2)\)};
  \node[rotate=270] at (0,0.5) {\(\cong\)};
  \node[rotate=270] at (4,0.5) {\(\cong\)};
  \node[rotate=270] at (8,0.5) {\(\cong\)};
  \node (A1) at (0,0) {\(\lmod{\C}\)};
  \node (A2) at (4,0) {\(\lmod{A}\)};
  \node (A3) at (8,0) {\(\lmod{R}\)};
  \draw[->] ([yshift=0.4ex]A1.east) [bend left=15] to node[above,font=\small] {\(\C \otimes \blank\)} ([yshift=0.4ex]A2.west);
  \draw[->] ([yshift=0.4ex]A2.east) [bend left=15] to node[above,font=\small] {\(\Hom_A(Ae_2, \blank)\)} ([yshift=0.4ex]A3.west);
  \draw[->] ([yshift=-0.4ex]A3.west) [bend left=15] to node[below,font=\small] {\(Ae_2 \otimes_R \blank \)} ([yshift=-0.4ex]A2.east);
  \draw[->] ([yshift=-0.4ex]A2.west) [bend left=15] to node[below,font=\small] {\(\Hom_A(\C,\blank)\)} ([yshift=-0.4ex]A1.east);
\end{tikzpicture}
\end{equation*}
where \(\frakp = \gl_2\).
This should be compared with the standard categorification of \(W^{\otimes 2}\) (see \cite{MR2305608}), where \(W\) is the vector representation of \(\mathfrak{sl}_2\):
\begin{equation*}
\begin{tikzpicture}
  \node (Z1) at (0,1) {\(\catO^\frakp_0(\gl_2)\)};
  \node (Z2) at (4,1) {\(\catO_0(\gl_2)\)};
  \node (Z3) at (8,1) {\(\catO^{\frakp}_0(\gl_2)\)};
  \node[rotate=270] at (0,0.5) {\(\cong\)};
  \node[rotate=270] at (4,0.5) {\(\cong\)};
  \node[rotate=270] at (8,0.5) {\(\cong\)};
  \node (A1) at (0,0) {\(\lmod{\C}\)};
  \node (A2) at (4,0) {\(\lmod{A}\)};
  \node (A3) at (8,0) {\(\lmod{\C}\)};
  \draw[->] ([yshift=0.4ex]A1.east) [bend left=15] to node[above,font=\small] {\(\C \otimes \blank\)} ([yshift=0.4ex]A2.west);
  \draw[->] ([yshift=0.4ex]A2.east) [bend left=15] to node[above,font=\small] {\(\Hom_A(\C, \blank)\)} ([yshift=0.4ex]A3.west);
  \draw[->] ([yshift=-0.4ex]A3.west) [bend left=15] to node[below,font=\small] {\(\C \otimes \blank \)} ([yshift=-0.4ex]A2.east);
  \draw[->] ([yshift=-0.4ex]A2.west) [bend left=15] to node[below,font=\small] {\(\Hom_A(\C, \blank)\)} ([yshift=-0.4ex]A1.east);
\end{tikzpicture}
\end{equation*}
In particular, note that the first and the second leftmost
weight spaces are categorified in the same way for \(\gl(1|1)\) and for \(\mathfrak{sl}_2\). This will
hold for all tensor powers \(V^{\otimes n}\) and
\(W^{\otimes n}\) and is due to the fact that these weight
spaces for \(\gl(1|1)\) and for \(\mathfrak{sl}_2\) agree as
 modules for the symmetric group. The second leftmost
weight space, in particular, is categorified using the
well-known category of modules over the path algebra of the
Khovanov-Seidel quiver \cite{1035.53122}
\begin{equation*}
  \begin{tikzpicture}[baseline=(A.base)]
    \node (A)  at (0,0) {\(1\)};
    \node (B)  at (1,0) {\(2\)};
    \node (C) at (2,0) {\hphantom{1}};
    \node (D) at (2.5,0) {\hphantom{1}};
    \node at (2.25,0) {\(\cdots\)};
    \node (E) at (3.5,0) {\(n\)};
    \draw[->] (A) [bend left] to node[above,font=\small] {\(a_1\)} (B);
    \draw[->] (B) [bend left] to node[below,font=\small] {\(b_1\)} (A);
    \draw[->] (B) [bend left] to node[above,font=\small] {\(a_2\)} (C);
    \draw[->] (C) [bend left] to node[below,font=\small] {\(b_2\)} (B);
    \draw[->] (D) [bend left] to node[above,font=\small] {\(a_{n-1}\)} (E);
    \draw[->] (E) [bend left] to node[below,font=\small] {\(b_{n-1}\)} (D);
  \end{tikzpicture}
  \qquad \parbox{6cm}{with relations \(b_1a_1 = 0\) and\\
    \(b_ia_i  = a_{i-1} b_{i-1}\) for all \(i=2,\ldots,n-1\).}
\end{equation*}
One should however notice the remarkable difference in the rightmost weight space of our example. Here our categorification differs from the \(\mathfrak{sl}_2\) picture and leaves the world of highest weight categories. This is evident, since \(R\) has infinite global dimension.

In general, the description of our categories is slightly
more involved, but still explicit. A graded diagram algebra
\(A_{n,k}\) which is graded isomorphic to the endomorphism
ring of a projective generator of the category
\(\calQ_k(\compn)\) is constructed in
\cite{2013arXiv1311.6968S}. In particular, we have an
equivalence of categories \(\calQ_k(\compn) \cong
\gmod{A_{n,k}}\) and hence a complete description of the
former category.  Using this, we are able to prove that the
functors \(\calE\) and \(\calF\) are indecomposable.  We
remark that one could expect an action of a KLR algebra on
powers of \(\calE\) and \(\calF\). However, notice that
since \(\calE^2=0\) and \(\calF^2=0\) it does not make sense
to investigate the endomorphism spaces \(\End(\calE^k)\) and
\(\End(\calF^k)\) for \(k > 1\). At the moment it is not
clear to us how one could get a 2-categorification for
\(\gl(1|1)\)--representations.

\subsubsection*{Outline of the paper}
\label{sec:outline-paper}
After this introduction, the paper contains a general section on super Howe duality and applications to categorification. In this section we recall the statement of super Howe duality and we show how it can be used to deduce a categorification of \(\gl(m|n)\)--representations from a categorification of \(\gl_k\)--representations.

The rest of the paper is concerned with the case of \(\gl(1|1)\). In Section~\ref{sec:hecke-algebra-hecke} we recall some facts about the Hecke algebra and the Hecke modules appearing \eqref{eq:177}. In Section~\ref{sec:representations-uqgl} we define the quantum superalgebra \(U_q(\gl(1|1))\) and a semisimple subcategory \(\catRep\) of representations.
In Section~\ref{reps:sec:diagr-intertw-oper} we provide a diagram calculus for the intertwining operators of representations using webs, similar to the \(\fraks\frakl_2\)--diagram calculus of \cite{MR1446615},
which allows to compute explicitly the canonical bases and the action of \(U_q(\gl(1|1))\). Section~\ref{sec:category-cato} is the technical heart of the paper and contains the definitions of the subquotient categories \(\calQ_k(\bolda)\) of \(\catO(\gl_n)\). In Section~\ref{sec:categorification} we then show how they can be used to construct a categorification of the representations in \(\catRep\).

\subsubsection*{Acknowledgements} The present work is part of the author's PhD thesis. The author would like to thank his advisor Catharina Stroppel for her help and support. The author would also like to thank the anonymous referee for many helpful comments.

\section{Categorification of \texorpdfstring{$\gl(m|n)$}{gl(m|n)}-representations}
\label{sec:howe-duality-glm}

As we mentioned in the introduction, the categorification construction that we present in the paper for \(U_q(\gl(1|1))\) can be extended to the case of \(U_q(\gl(m|n))\) for \(m,n \geq 1\). However the combinatorics for \(U_q(\gl(1|1))\) is already quite involved; developing the analogous combinatorics for general \(U_q(\gl(m|n))\) would make this work unreadable.

Nevertheless, in order to be complete, we want to present in this preliminary section  a categorification result for \(\gl(m|n)\), avoiding some of the technicalities. In order to do that, we make the following simplifications:
\begin{itemize}
  \item we consider the classical (non-quantum) version;
  \item we consider only tensor powers of the vector representation (and not their subrepresentations);
  \item we categorify only the action of the intertwining operators (and not of \(\gl(m|n)\).
\end{itemize}
We derive the categorification from super skew Howe duality instead of from Schur-Weyl duality, although the two approaches are equivalent.

\subsection{Super Howe duality}
\label{sec:super-howe-duality-1}

Let \(I_{m|n}= \{1,\ldots,m+n\}\) with a parity function \(\abs{\cdot}: I_{m|n} \mapto \Z/2\Z\) defined by
\begin{equation}
  \label{eq:116}
  \abs{i} =
  \begin{cases}
    0 & \text{if } i \leq m \\
    1 & \text{if } i > m
  \end{cases}
\end{equation}
for each \(i \in I_{m|n}\). Let also \(\C^{m|n}\) be the \(m+n\)--dimensional super vector space on basis \(\{e_i \suchthat i \in I_{m|n}\}\) such that \(\abs{e_i} = \abs{i}\), where as usual \(\abs{v}\) denotes the degree of an homogeneous element \(v \in \C^{m|n}\). Then the Lie superalgebra \(\gl(m|n)\) is the super vector space of matrices \(\End(\C^{m|n})\) equipped with the Lie super bracket
\begin{equation}
  \label{eq:117}
  [x,y] = xy - (-1)^{\abs{x}\abs{y}} yx.
\end{equation}
In particular note that \(\gl(m|0) \cong \gl(0|m) \cong \gl_m\). The Lie superalgebra \(\gl(m|n)\) acts by matrix multiplication on \(\C^{m|n}\): this is the vector representation of \(\gl(m|n)\).

If \(V\) is a super vector space, we define an action of the symmetric group \(\bbS_N\) on the tensor power \(\bigotimes^N V\)
by setting
\begin{equation}
  \label{eq:118}
  s_\ell \cdot (x_{1} \otimes \cdots \otimes x_{N}) =  (-1)^{\abs{x_\ell}\abs{x_{\ell+1}}} x_{1} \otimes \cdots \otimes x_{\ell+1} \otimes x_\ell \otimes \cdots \otimes x_N
\end{equation}
for every simple reflection \(s_\ell \in \bbS_N\). Let \(\pi^{\operatorname{S}},\pi^{\bigwedge} \in \C[\bbS_n]\) be the idempotents projecting onto the trivial and sign representations respectively. We set then
\begin{equation}
  \label{eq:119}
  \operatorname{S}^N V = \pi^{\operatorname{S}}\cdot (\textstyle\bigotimes^N V) \qquad \text{and} \qquad \textstyle\bigwedge^N V = \pi^{\bigwedge} \cdot (\bigotimes^N V).
\end{equation}
In particular, notice that if \(V\) is a vector space concentrated in  degree zero then this definitions coincide with the usual symmetric and exterior powers of \(V\).

\begin{remark}
  \label{rem:O6}
  Notice that \(\operatorname{S}^N(\C^{m|n}) \cong \bigwedge^N(\C^{n|m})\). It follows in particular that, in contrast to the classical case, \(\bigwedge^N V\) can be non-zero also for \(N\gg 0\).
\end{remark}

If \(v_1,\ldots,v_r\) is a basis of \(V\), then a basis of \(\bigwedge^N V\) is given by
\begin{equation}
  \label{eq:121}
  v_{i_1} \wedge \cdots \wedge v_{i_N} = \pi^{\bigwedge} \cdot (v_{i_1} \otimes \cdots \otimes v_{i_N})
\end{equation}
for all sequences \((i_1,\ldots,i_N)\) of indices \(i_\ell \in \{1,\ldots,r\}\) such that \(i_1 \leq i_2 \leq \cdots \leq i_N\) and if \(i_\ell = i_{\ell+1}\) then \(\abs{v_{i_\ell}} = 1\). Moreover a basis of \(\operatorname S^N V\) is given by
\begin{equation}
  \label{eq:123}
  v_{i_1} \symm \cdots \symm v_{i_N} = \pi^{\operatorname S} \cdot (v_{i_1} \otimes \cdots \otimes v_{i_N})
\end{equation}
for all sequences \((i_1,\ldots,i_N)\) of indices \(i_\ell \in \{1,\ldots,r\}\) such that \(i_1 \leq i_2 \leq \cdots \leq i_N\) and if \(i_\ell = i_{\ell+1}\) then \(\abs{v_{i_\ell}} = 0\).

We have the following result (cf.\ \cite{MR1847665}, \cite{2010arXiv1001.0074C}):

\begin{prop}[Super Howe duality]
  \label{prop:5}
  Let \(p,m,N \in \Z_{>0}\) be positive integers and \(q,n \in \Z_{\geq0}\). The natural actions of \(\gl(p|q)\) and \(\gl(m|n)\) on \(\bigwedge^N ( \C^{p|q} \otimes \C^{m|n} )\) commute with each other and generate each other's centralizer.
  As a \(\gl(m|n)\)--module, \(\bigwedge^N ( \C^{p|q} \otimes \C^{m|n} )\) decomposes as the direct sum
  \begin{equation}
    \label{eq:120}
    \displaystyle \bigoplus_{i_1 + \cdots + i_{p+q} = N} \textstyle \bigwedge^{i_1} \C^{m|n} \otimes \cdots \otimes \bigwedge^{i_p} \C^{m|n} \otimes \operatorname{S}^{i_{p+1}} \C^{m|n} \otimes \cdots \otimes \operatorname{S}^{i_{p+q}} \C^{m|n}.
  \end{equation}
\end{prop}

Note that inverting the roles of \(p|q\) and \(m|n\) we have a similar decomposition \eqref{eq:120} as a \(\gl(p|q)\)--module.

\begin{proof}
  The first part is \cite[Theorem~3.3 and Corollary~3.2]{MR1847665}. We check the decomposition \eqref{eq:120}.

  Let \(\{e_1,\ldots,e_{p+q}\}\) and \(\{f_1,\ldots,f_{m+n}\}\) be the standard bases of \(\C^{p|q}\) and \(\C^{m|n}\) respectively. We fix the following ordered basis of \(\C^{p|q} \otimes \C^{m|n}\):
  \begin{equation}
    \label{eq:125}
    e_1 \otimes f_1 , \ldots ,e_1 \otimes f_{m+n}, \ldots , e_{p+q} \otimes f_1 , \ldots, e_{p+q} \otimes f_{m+n}.
  \end{equation}
We get then a basis of \(\bigwedge^N ( \C^{p|q} \otimes \C^{m|n} )\) as in \eqref{eq:121}. Let \(M\) be equal to \eqref{eq:120}. We define an isomorphism \(\Psi\) from \(\bigwedge^N(\C^{p|q} \otimes \C^{m|n})\) to \(M\) in the following way. Given a basis vector \(w=(e_{i_1} \otimes f_{j_1}) \wedge \cdots \wedge (e_{i_N} \otimes f_{j_N})\)  of \(\bigwedge^N(\C^{p|q} \otimes \C^{m|n})\), define functions \(a,b : \{1,\ldots,p+q\} \mapto \{\bullet,1,\ldots,N\}\) by \(a(h) = \min\{\ell \suchthat i_\ell = h\}\) and \(b(h) = \max\{\ell \suchthat i_\ell = h\}\) or \(a(h)=b(h)=\bullet\) if this set is empty. Set also \(c(h)=b(h)-a(h)+1\), with the convention \( \bullet - \bullet = -1\). Then we define
\begin{equation}
\Psi(w) \in \textstyle \bigwedge^{c(1)} \C^{m|n} \otimes \cdots \otimes \bigwedge^{c(p)} \C^{m|n} \otimes \operatorname S^{c(p+1)} \otimes \cdots \otimes \operatorname S^{c(q)} \C^{m|n}\label{eq:126}
\end{equation}
to be the element
  \begin{multline}
    \label{eq:124}
      (f_{j_{a(1)}} \wedge \cdots \wedge f_{j_{b(1)}})
       \otimes  \cdots
       \otimes (f_{j_{a(m)}} \wedge \cdots \wedge f_{j_{b(m)}})\\
       \otimes (f_{j_{a(m+1)}} \symm \cdots \symm f_{j_{b(m+1)}})
       \otimes   \cdots
       \otimes (f_{j_{a(m+n)}} \symm \cdots \symm f_{j_{b(m+n)}}).
  \end{multline}
It is straightforward to check that this is indeed an element of the basis, and  that \(\Psi\) is bijective and \(\gl(m|n)\)--equivariant.
\end{proof}

\begin{remark}
  \label{catO:rem:9}
  Another kind of duality, called super Schur-Weyl duality, relates \(\gl(m|n)\) and the symmetric group \(\bbS_N\): the natural action of \(\C[\bbS_N]\) on \(V^{\otimes N}\) is \(\gl(m|n)\)--equivariant; moreover, the map \(\C[\bbS_N] \mapto \End_{\gl(m|n)}(V^{ \otimes N})\) is always surjective, and it is injective if and only if \(N \leq (m+1)(n+1)\) (see \cite{MR884183}, \cite{MR735715}).
\end{remark}

\subsection{Categorification of \texorpdfstring{$\gl(m|n)$}{gl(m|n)}}
\label{sec:super-howe-duality}

Set now \(V= \C^{m|n}\). Our goal is to construct a categorification of  \(V^{\otimes N}\) for \(N>0\).

Set \(p=N\) and \(q=0\) in Proposition~\ref{prop:5}. We have then that \(\bigwedge^N (\C^N \otimes V)\) decomposes as a \(\gl(m|n)\)--module as
\begin{equation}
  \label{eq:127}
  \bigoplus_{i_1+\cdots+i_N = N} \textstyle \bigwedge^{i_1} V \otimes \cdots \otimes \bigwedge^{i_N} V
\end{equation}
and as a \(\gl_N\)--module as
\begin{equation}
  \label{eq:128}
  \bigoplus_{j_1+\cdots + j_{m+n}=N} \textstyle \bigwedge^{j_1} \C^N \otimes \cdots \otimes \bigwedge^{j_m} \C^N \otimes \operatorname S^{j_{m+1}} \C^N \otimes \cdots \otimes \operatorname S^{j_{m+n}} \C^N.
\end{equation}
Notice that one summand of \eqref{eq:127} is in particular \(V^{\otimes N}\). A categorification of the \(\gl_N\)--module \eqref{eq:128} using the BGG category \(\catO\) has been constructed in \cite{2014arXiv1407.4267S}; we are going to use it in order to categorify the \(\gl(m|n)\)--module \eqref{eq:127}.

In order to state the categorification theorem, we need some notation.
Let us fix the standard basis \(\{v_1,\ldots,v_{m+n}\}\) of \(V=\C^{m|n}\), with
\begin{equation}
\abs{v_i}=
\begin{cases}
  0 & \text{for }i=1,\ldots,m,\\
  1 & \text{for }i=m+1,\ldots,n.
\end{cases}
\label{eq:187}
\end{equation}
Let \(\frakh \subset \gl(m|n)\) be the subalgebra of diagonal matrices. Then \(V^{\otimes N}\) decomposes as direct sum of weight spaces for the action of \(\frakh\). Let \(\Lambda\) be the set of compositions \(\lambda=(\lambda_1,\ldots,\lambda_{m+n})\) of \(N\) with at most \(m+n\) parts (that is, we allow \(\lambda_i=0\) for some indices \(i\)). Then the weight spaces of \(V^{\otimes N}\) are indexed by \(\Lambda\), and the correspondence is given by
\begin{equation}
  \label{eq:188}
  (V^{\otimes N})_\lambda = \vectorspan \{ v_{\sigma(a^\lambda_1)} \otimes \cdots \otimes v_{\sigma(a^\lambda_{N})} \suchthat \sigma \in \bbS_N \},
\end{equation}
where
\begin{equation}
  \label{eq:189}
  a^\lambda = (\underbrace{1,\ldots,1}_{\lambda_1}, \underbrace{2,\ldots,2}_{\lambda_2},\ldots, \underbrace{m+n,\ldots,m+n}_{\lambda_{m+n}}).
\end{equation}

We can now state our main result:

\begin{theorem}
  \label{thm:3}
  Given \(\lambda \in \Lambda\), let \(\frakq_\lambda \subset \gl_N\) be the standard parabolic subalgebra corresponding to the composition \((\lambda_1,\ldots,\lambda_m,1,\ldots,1)\) and \(\frakp_\lambda \subset \gl_N\) be the standard parabolic subalgebra corresponding to the composition \((1,\ldots,1,\lambda_{m+1},\ldots,\lambda_{m+n})\). Then there is an isomorphism
  \begin{equation}\label{eq:190}
      \C \otimes_\Z K(\catO^{\frakp_\lambda,\frakq_\lambda \pres}_0(\gl_N)) \longmapsto (V^{\otimes N})_\lambda
\end{equation}
sending equivalence classes of standard modules to standard basis vectors.

For each index \(i=1,\ldots,N-1\) choose a singular weight \(\lambda_i\) for \(\gl_N\) whose stabilizer under the dot action is generated by the simple reflection \(s_i\). Then defining \(\theta_i = \trasT_{\lambda_i}^0 \circ \trasT_0^{\lambda_i}\) we get a categorical action of the generators \(s_i + 1\) of \(\C[\bbS_N]\), which descends to the action \eqref{eq:118}.
\end{theorem}

We refer to Section~\ref{sec:category-cato} for the definitions of the categories appearing in \eqref{eq:190} and of the translation functors \(\trasT_{\lambda_i}^0\) and \(\trasT_0^{\lambda_i}\).

\begin{proof}
  The first claim follows from the definition of the categories \(\catO^{\frakp_\lambda,\frakq_\lambda \pres}_0(\gl_N)\) (cf.\ Section~\ref{sec:category-cato}). The second claim can be proved generalizing the proof of Theorem~\ref{thm:1}.
\end{proof}

\begin{remark}
Combining Zuckermann's/coapproximation functors and their adjoints (see \textsection\ref{sec:funct-betw-categ} for the definitions) one can define functors \(\calE_j\), \(\calF_j\) for \(j=1,\ldots,m+n-1\) between some opportune unbounded derived categories, as in \textsection\ref{sec:functors-scr-e}. These functors
commute with the functors \(\theta_i\) and give an action of \(\gl(m|n)\) at the level of the Grothendieck groups.
\end{remark}

We remark that for \(n=0\) Theorem \ref{thm:3} gives exactly the categorification of \((\C^m)^{\otimes N}\) developed in \cite{MR2567504}.

\section{The Hecke algebra and Hecke modules}
\label{sec:hecke-algebra-hecke}

In this section, which can be skipped at a first reading and used as a reference, we recall the definition of the bar involution and the canonical basis of the Hecke algebra for the symmetric group \(\bbS_n\). We then study in detail induced sign/trivial modules.

\subsection{Hecke algebra}
\label{sec:hecke-algebra}
Let \(\bbS_n\) denote the symmetric group of permutations of \(n\)
elements, generated by the simple reflections \(s_i\) for
\(i=1,\ldots,n-1\). For \(w \in \bbS_n\) we denote by \(\len(w)\) the length
of \(w\). Moreover, we denote by \(\prec\) the Bruhat order on \(\bbS_n\).

The Hecke algebra of the symmetric group \(W=\bbS_n\) is the
unital associative \(\C(q)\)--algebra \(\ucalH_n\) generated by \(\{H_i
\suchthat i=1,\ldots,n-1\}\) with relations
\begin{subequations}
  \begin{align}
    H_i H_j &= H_j H_i \qquad \text{if } \vert i-j\vert>2,\label{eq:7} \\
    H_i H_{i+1} H_i & = H_{i+1} H_i H_{i+1}, \label{eq:8}\\
    H_i^2&=(q^{-1}-q)H_i+1.\label{eq:9}
  \end{align}
\end{subequations}
Notice that we use Soergel's normalization \cite{MR1445511}, instead of the original one. However, we use the letter \(q\) as parameter in analogy with the quantum parameter of \(\Uqgl\).
It follows from \eqref{eq:9} that the elements \(H_i\) are invertible
with \(H_i^{-1}=H_i +q-q^{-1}\). For \(w \in \bbS_n\) such that \(w=s_{i_1} \cdots s_{i_r}\) is a
reduced expression, we define \(H_w=H_{i_1}\cdots H_{i_r}\). Thanks to
\eqref{eq:8}, this does not depend on the chosen reduced
expression. The elements \(H_w\) for \(w \in W\) form a basis of
\(\ucalH_n\), called \emph{standard basis}.

We can define on \(\ucalH_n\) a \emph{bar involution} by
\(\overline{H_w}=H^{-1}_{w^{-1}}\) and \(\overline{q}=q^{-1}\);
in particular \(\overline{H_i}=H_i+q-q^{-1}\). We also have a {\em
  bilinear form} \(\langle -,- \rangle\) on \(\ucalH_n\) such that the
standard basis elements are orthonormal:
\begin{equation}\label{eq:165}
  \langle H_w , H_{w'} \rangle = \delta_{w, w'} \qquad \text{for all }w, w' \in W.
\end{equation}

By standard arguments one can prove the following:

\begin{prop}[\cite{MR560412}, in the normalization of \cite{MR1445511}]\label{prop:KLBasisOnHeckeAlgebra}
  There exists a unique basis \(\{\underline H_w \suchthat w \in
  W\}\) of \(\ucalH_n\) consisting of bar-invariant elements such that
  \begin{equation}
    \underline H_w = H_w + \sum_{w' \prec w} \mathcal{P}_{w',w}(q) H_{w'}
  \end{equation}
  with \(\mathcal{P}_{w',w} \in q\Z[q]\) for every \(w' \prec w\).
\end{prop}

The basis \(\underline H_w\) is called \emph{Kazhdan-Lusztig basis} or \emph{canonical  basis} of \(\ucalH_n\).

\begin{remark}\label{remark:ConstructionOfCanonicalBasis}
  There is an inductive way to construct the canonical basis
  elements. First, note that by definition \(\canonicalHecke_e = H_e\). Then
  set \(\canonicalHecke_i = H_i + q\): since
  \(\canonicalHecke_i\) is bar invariant, we must have
  \(\canonicalHecke_{s_i} = \canonicalHecke_i\). Now suppose
  \(w= w's_i \succ w'\): then \(\canonicalHecke_{w'} \canonicalHecke_i\) is
  bar invariant and is equal to \( H_{w}\) plus a
  \(\Z[q,q^{-1}]\)--linear combination of some \(H_{w''}\) for
  \(w'' \prec w\). It follows that
  \begin{equation}
    \label{eq:10}
    \underline H_{w'}\canonicalHecke_i = \underline H_w + p \quad \text{ for some } p \in \bigoplus_{w'' \prec w} \Z \underline H_{w''} .
  \end{equation}
\end{remark}

\subsection{Induced Hecke modules}
\label{sec:induc-hecke-modul}
We will consider induced Hecke modules which are a mixed version of induced sign and induced trivial modules studied in  \cite{MR1445511}. In the following, all modules over the Hecke algebra will be right modules.

Let \(W_\frakp, W_\frakq\) be two parabolic subgroups\footnote{We use this notation because \(W_\frakp\) and \(W_\frakq\)  will correspond later to two parabolic subalgebras \(\frakp,\frakq \subset \gl_n\).} of \(W=\bbS_n\) (that is, they are generated by simple transpositions) such that the elements of \(W_\frakp\) commute with the elements of \(W_\frakq\). Note that \(W_{\frakp+\frakq}=W_\frakp \times W_\frakq\) is also a parabolic subgroup of
\(W\). Let  \(\ucalH_\frakp\), \(\ucalH_\frakq\) and \(\ucalH_{\frakp+ \frakq}\) be the corresponding Hecke algebras; they are all naturally subalgebras of \(\ucalH_n\). 
We denote by
  \(\sgnmodule_\frakp\) the \emph{sign representation} of
  \(\ucalH_\frakp\); this is the one-dimensional
  \(\C(q)\)--vector space on which each generator \(H_i \in
  \ucalH_{\frakp}\) acts as \(-q\). Similarly we denote by
  \(\trvmodule_\frakq\) the \emph{trivial
    representation} of \(\ucalH_\frakq\), which is the one-dimensional \(\C(q)\)--vector space on which each
  generator \(H_i \in \ucalH_{\frakq}\) acts as
  \(q^{-1}\). We define the \emph{mixed induced Hecke module}
\begin{equation}\label{eq:67}
\ucalM^\frakp_\frakq = \Ind_{\ucalH_{\frakp+\frakq}}^{\ucalH_n}( \sgnmodule_\frakp \boxtimes \trvmodule_\frakq) = ( \sgnmodule_\frakp \boxtimes \trvmodule_\frakq) \otimes_{\ucalH_{\frakp+\frakq}} \ucalH_n.
\end{equation}
If \(W_\frakp\) is trivial, we omit \(\frakp\) from the notation and we write \(\ucalM_\frakq\). Analogously, if \(W_\frakq\) is trivial we omit \(\frakq\) and we write \(\ucalM^\frakp\). Note that in \(\ucalM_\frakq\) and \(\ucalM^\frakp\)  are denoted respectively \(\ucalM^\frakq\) and \(\ucalN^\frakp\) in \cite{MR1445511}.

Let \(W^\frakp\), \(W^\frakq\) and \(W^{\frakp+\frakq}\) be the set of shortest coset representatives for the left quotients \(W_p \backslash W\), \(W_\frakq \backslash W\) and \(W_{\frakp+\frakq}\backslash W\) respectively. Then a basis of \(\ucalM^\frakp_\frakq\) is given by
\begin{equation}
\{N_w = 1 \otimes H_w \suchthat w \in W^{\frakp+\frakq}\}\label{eq:178}
\end{equation}
 (where \(1\)  is some chosen generator of the \(\C(q)\)--vector space \(\sgnmodule_\frakp \boxtimes \trvmodule_\frakq\)).

The action of \(\ucalH_n\) on \(\ucalM^\frakp_\frakq\) is given explicitly by the following lemma:

\begin{lemma}\label{lemma:ActionOfHOnM}
  For all \(w \in W^{\frakp+\frakq}\) we have
  \begin{equation}\label{eq:ActionOfHOnM}
    N_w \cdot H_i =
    \begin{cases}
      N_{ w s_i} & \text{if }  w s_i \in W^{\frakp+\frakq} \text{ and } \len (w s_i) > \len(w),\\
      N_{w s_i} + (q^{-1}-q) H_w & \text{if } w s_i \in W^{\frakp+\frakq} \text{ and } \len (w s_i) < \len(w),\\
      -q N_{w} &  \text{if }  w s_i=s_j w \text{ for } s_j \in W_\frakp,\\
      q^{-1} N_{w} & \text{if }  w s_i= s_j w  \text{ for } s_j \in W_\frakq.
    \end{cases}
  \end{equation}
\end{lemma}

The module \(\ucalM^\frakp_\frakq\) inherits a bar involution by setting \(\overline {N_w} =1 \otimes \overline {H_w}\). Moreover, the bilinear form \eqref{eq:165} induces a bilinear form on \(\ucalM^{\frakp}_{\frakq}\).
A canonical basis can be defined on \(\ucalM^\frakp_\frakq\) by the following generalization of Proposition~\ref{prop:KLBasisOnHeckeAlgebra}:

\begin{prop}\label{prop:18}
  There exists a unique basis \(\{\underline N_w \suchthat w \in
  W^{\frakp+\frakq}\}\) of \(\ucalM^\frakp_\frakq\) consisting of bar-invariant elements satisfying
  \begin{equation}
    \underline N_w = N_w + \sum_{w' \prec w} \mathcal{R}_{w',w}(q) N_{w'}
  \end{equation}
  with \(\mathcal{R}_{w',w} \in q\Z[q]\) for all \(w' \prec w\).
\end{prop}

As described in Remark \ref{remark:ConstructionOfCanonicalBasis}, one
can construct inductively the canonical basis of \(\ucalM^\frakp_\frakq\). In particular,
for \( W^{\frakp+\frakq} \ni w s_i \succ w\) one always has
\begin{equation}
  \label{eq:17}
  \underline N_{w} \underline H_i = \underline N_{w s_i} + p
\end{equation}
where \(p\) is a \(\Z\)--linear combination of \(\underline N_{w'}\) for \(w'
\prec w s_i\).

\subsubsection{Maps between Hecke modules I}
\label{sec:maps-between-hecke}
We will now construct maps between induced modules \(\ucalM^{\frakp}_{\frakq}\) corresponding to different pairs of parabolic subgroups \(W_\frakp\), \(W_\frakq\). First, we consider the case in which we change the subgroup \(W_\frakq\).

Let \(W_{\frakq'} \subset W_{\frakq}\) be also a parabolic subgroup of \(W\). Let us define a map \(\sfi = \sfi_{\frakq}^{\frakq'}: \ucalM^\frakp_\frakq \rightarrow \ucalM^\frakp_{\frakq'}\) by
\begin{equation}\label{eq:162}
  \sfi: N_w \longmapsto \sum_{x \in W^{\frakq'} \cap W_\frakq} q^{\len(w^{\frakq'}_\frakq) -\len(x)} N_{xw}
\end{equation}
where \(w^{\frakq'}_\frakq\) is the longest element of \(W^{\frakq'} \cap W_\frakq = (W_{\frakq'} \backslash W_\frakq)^{\short}\). Note that for \(w \in
W^{\frakp+\frakq}\) and \(x \in W^{\frakq'} \cap W_{\frakq}\) the product \(xw\) is an element of \(W^{\frakp+\frakq'}\).

The map \eqref{eq:162} is natural, in the sense that if \(W_{\frakq''} \subset W_{\frakq'}\) is another subgroup of \(W\) generated by simple reflections then \(\sfi_{\frakq}^{\frakq''}= \sfi_{\frakq'}^{\frakq''} \circ \sfi_{\frakq}^{\frakq'}\); this follows because each element of \((W_{\frakq''}\backslash W_\frakq)^{\short}\) factors in a unique way as the product of an element of \((W_{\frakq''}\backslash W_{\frakq'})^{\short}\) and an element of \((W_{\frakq'}\backslash W_\frakq)^{\short}\).

\begin{lemma}\label{lemma:FunctionBetweenM1}
  The map \(\sfi\) just defined is an injective homomorphism of
  \(\ucalH_n\)--modules that commutes with the bar involution. Moreover
  it sends the canonical basis element \(\underline N_w\) to the
  canonical basis element \(\underline N_{w^{\frakq'}_\frakq w}\).
\end{lemma}

\begin{proof}
  The injectivity is clear, because \(\sfi(N_w)\) is a linear
  combination of \(N_{w'}\) for \(w' \prec w^{\frakq'}_\frakq w\) and the
  coefficient of \(N_{w^{\frakq'}_\frakq w}\) is \(1\). To prove that \(\sfi\) is a homomorphism of \(\ucalH_n\)--modules, it is sufficient to consider the case \(W_{\frakq'}=\{e\}\). In fact, we have a commutative diagram of injective maps
  \begin{equation}
    \begin{tikzpicture}[baseline=(current bounding box.center)]
      \matrix (m) [matrix of math nodes, row sep=3em, column
      sep=2.5em, text height=1.5ex, text depth=0.25ex] {
        & \ucalM^\frakp & \\
        \ucalM^\frakp_\frakq &  & \ucalM^{\frakp}_{\frakq'}\\};
      \path[right hook->] (m-2-1) edge node[auto] {\( \sfi_{\frakq}\)} (m-1-2);
      \path[right hook->] (m-2-1) edge node[auto] {\( \sfi_{\frakq}^{\frakq'}\)} (m-2-3);
      \path[left hook->] (m-2-3) edge node[above right] {\( \sfi_{\frakq'}\)} (m-1-2);
    \end{tikzpicture} \label{eq:163}
  \end{equation}
  and if \(\sfi_\frakq\) and \(\sfi_{\frakq'}\) are both \(\ucalH_n\)--equivariant then so is \(\sfi_{\frakq}^{\frakq'}\).

  Hence let \(\sfi=\sfi_\frakq\) and let us show using
  \eqref{eq:ActionOfHOnM} that \(\sfi(N_w H_i) = \sfi(N_w) H_i\) for all
  \(i=1,\ldots,n-1\) and for each basis element \(N_w \in
  \ucalM^\frakp_\frakq\). Note first that \(W^{\frakp+\frakq} \subset
  W^{\frakp}\); moreover, if \( w s_i \in W^{\frakp+\frakq}\) then \(x w
  s_i \in W^{\frakp}\) for every \(x \in W_\frakq\), so that the first
  two cases of \eqref{eq:ActionOfHOnM} are clear. Suppose then that we
  are in the fourth case, that is \(w s_i =s_j w \) for some \(s_j \in
  W_\frakp\); then \(xws_i = x s_j w = s_j x w\) for every \(x \in
  W_\frakq\), because elements of \(W_\frakp\) commute with elements of
  \(W_\frakq\). In the third case of \eqref{eq:ActionOfHOnM} an explicit
  computation (see \cite{miophd2} for the details) shows that
  \(\sfi(N_w)H_i=\sfi(N_w H_i)\).

  It remains to show the bar invariance. Again, by the same argument as before it is sufficient to consider the case \(W_{\frakq'}=\{e\}\). It is enough to
  check it for a basis; in fact we will prove by induction that
  \(\sfi(\underline N_w)\) is bar invariant for every \(w \in
  W^{\frakp+\frakq}\). For \(w=e\), we have \(\sfi(\underline N_e) = \sfi(
  N_e) = \sum_{x \in W_\frakq} q^{\len(w_\frakq) -\len(x)} N_{x}\), that is
  well-known to be the canonical basis element for \(\ucalH_{\frakq}\)
  corresponding to the longest element of \(W_\frakq\): hence it is bar
  invariant. For the inductive step, suppose \( w s_i \succ w\) and use
  \eqref{eq:17}:
    \begin{equation}
      \label{eq:18}
      \begin{split}
        \overline{\sfi(\underline N_{ws_i })} & = \overline{\sfi(
          \underline N_w \underline H_i- p)} = \overline{\sfi(\underline N_w)\underline H_i -
          \sfi(p)} \\ &=  \sfi(\underline N_w)\underline H_i - \sfi(p) = \sfi(
        \underline N_w \underline H_i- p) = \sfi(\underline N_{ ws_i}).
      \end{split}
    \end{equation}

    The last claim follows by the uniqueness of the canonical basis
    elements, because \(\sfi(\underline N_w)\) is bar invariant and the
    coefficient of \(N_{w'}\) in its standard basis expression is
    \begin{itemize}
    \item \(1\) if \(w'=w^{\frakq'}_\frakq w\),
    \item a multiple of \(q\) if \(w' = xw''\) for some \(x \in W^{\frakq'} \cap W_{\frakq}\) and \(w'' \in W^\frakq\) with \(w'' \preceq w\) (but \(w' \neq w^{\frakq'}_\frakq w)\),
    \item \(0\) otherwise. \qedhere
    \end{itemize}
\end{proof}

Now we define a left inverse \(\sfQ: \ucalM^\frakp_{\frakq'} \rightarrow \ucalM^{\frakp}_\frakq\) of \(\sfi\) by setting
\begin{equation}
  \sfQ(N_e) = \frac{1}{c_{\frakq}^{\frakq'}} N_e, \qquad \text{where} \qquad  c_{\frakq}^{\frakq'} = \sum_{x \in W^{\frakq'} \cap W_{\frakq}} q^{\len(w^{\frakq'}_\frakq) - 2 \len(x)} .\label{eq:164}
\end{equation}
It is easy to show that \(\sfQ\) is indeed well-defined (since \(\ucalM^\frakp_\frakq\) is a quotient of \(\ucalM^\frakp_{\frakq'}\), and \(\sfQ\) is, up to a multiple, the quotient map).
Moreover
\begin{equation}
  \label{eq:21}
    \sfQ \circ \sfi(N_w) = \sfQ \bigg( \sum_{x \in W^{\frakq'} \cap
      W_\frakq}
    q^{\len(w^{\frakq'}_\frakq)-\len(x)} N_{xw} \bigg)
    = N_w
\end{equation}
for all basis elements \(N_w \in \ucalM^\frakp_\frakq\).

\subsubsection{Maps between Hecke modules II}
\label{sec:maps-between-hecke-1}

Now let us examine the case in which we change the subgroup \(W_\frakp\). Namely let \(W_{\frakp'} \subset W_{\frakp}\) be a parabolic subgroup of \(W\), and define a linear map \(\sfj = \sfj_{\frakp}^{\frakp'}:
\ucalM^\frakp_\frakq \rightarrow \ucalM^{\frakp'}_\frakq\) by
\begin{equation}
  \label{eq:19}
 \sfj: N_w \longmapsto \sum_{x \in W^{\frakp'} \cap W_\frakp} (-q)^{\len(x)} N_{xw}
\end{equation}

As for Lemma~\ref{lemma:FunctionBetweenM1} it is easy to prove that
\(\sfj\) is an injective homomorphism of \(\ucalH_n\)--modules, anyway it does
not commute with the bar involution and it does not send canonical
basis elements to canonical basis elements. Instead, \(\sfj\) sends the dual canonical basis (defined to be the basis that is dual to the canonical basis with respect to the bilinear form) to the dual canonical basis.

Define also a \(\ucalH_n\)--modules homomorphism \(\sfz: \ucalM^{\frakp'}_\frakq
\rightarrow \ucalM^\frakp_\frakq\) by setting \(\sfz(N_e)=N_e\). This gives a well-defined
homomorphism because \(\ucalM^\frakp_\frakq\) is a quotient of \(\ucalM^{\frakp'}_\frakq\) and it is cyclic.

\begin{lemma}
  The map \(\sfz\) is bar invariant and sends the canonical basis
  element \(\underline N_w \in \ucalM^{\frakp'}_\frakq\) to \(\underline
  N_w \in \ucalM^\frakp_\frakq\) if \(w \in W^{\frakp+\frakq}\) and to
  \(0\) otherwise. Moreover \(\sfz \circ \sfj = \sum_{x \in W^{\frakp'}
    \cap W_{\frakp}} q^{2 \len(x)} \id\).
\end{lemma}

\begin{proof}
  The map \(\sfz\) is bar invariant by definition: in fact obviously
  \(\sfz(N_e) = \overline{\sfz(N_e)}\), and then by multiplying with
  the \(C_i\)'s one can see that \(\sfz\) is bar invariant on a set of
  generators.

  If \(w \in W^{\frakp+\frakq}\) then it is easily seen that \(\sfz(\underline N_w) \in N_w + \sum_{w' \prec w} q \Z[q] N_{w'} \).
  By uniqueness of the canonical basis elements it has to be \(\sfz(\underline N_w) = \underline
  N_w \in \ucalM^\frakp_\frakq\). If \(w \notin W^{\frakp+\frakq}\) then by the same reasoning \(\sfz
  (\underline N_w) = 0\).

  Moreover
  \begin{equation}
      \sfz \circ \sfj(N_w)  = \sfz \bigg( \sum_{x \in W^{\frakp'} \cap W_\frakp} (-q)^{\len(x)} N_{xw} \bigg) = \sum_{x \in W^{\frakp'} \cap W_\frakp} q^{2\len( x)}  N_w,    \label{eq:20}
  \end{equation}
  hence the last assertion follows as well.
\end{proof}

\section{Representations of \texorpdfstring{$\Uqgle$}{Uq(gl(1|1))}}
\label{sec:representations-uqgl}

We recall the definition of the quantum enveloping algebra \(\Uqgle\) and of its braided structure. 
We introduce then a subcategory consisting of semisimple
representations of \(\Uqgle\), that contains the tensor powers of the
vector representation \(V\). We will study in detail the intertwining
operators using a super version of Schur-Weyl duality.

In the following, as usual, by a \emph{super} object (for
example vector space, algebra, Lie algebra, module) we
mean a \(\Z/2\Z\)--graded object. If \(X\) is such a super object
we will use the notation \(\abs{x}\) to indicate the degree of
a homogeneous element \(x \in X\).
Elements of degree \(0\) are
called \emph{even}, while elements of degree \(1\) are called
\emph{odd}.  We stress that whenever we write \(\abs{x}\) we
will always be assuming \(x\) to be homogeneous.

\subsection{The quantum enveloping superalgebra \texorpdfstring{$\Uqgle$}{Uq(gl(1|1))}}
\label{sec:quant-envel-super}

The quantum enveloping superalgebra \(\Uqgl=\Uqgle\)
is a \(q\)--deformed version of the universal enveloping algebra of
the Lie superalgebra \(\mathfrak{gl}(1|1)\).
Let \(\sfP=\Z\epsilon_1 \oplus \Z\epsilon_2\) be the
\emph{weight lattice} of \(\mathfrak{gl}(1|1)\) (see also \textsection\ref{sec:super-howe-duality-1}) and \(\sfP^*=\Z h_1 \oplus \Z h_2\) its dual with the
natural bilinear pairing \(\langle h_i, \epsilon_j \rangle=
\delta_{i,j}\), and set \(\alpha=\epsilon_1-\epsilon_2\) to be the {\em
  simple root} of \(\mathfrak{gl}(1|1)\). Then \(\Uqgl\) is defined to be
the unital superalgebra over \(\C (q)\) with generators \(E,F,\quantumq^h\,
(h \in \sfP^*)\)  in degrees \(\abs{\quantumq^h}=0\),
\(\abs{E}=\abs{F}=1\) subject to the relations
\begin{equation}
  \begin{aligned}
    \quantumq^0&=1, & \quantumq^h \quantumq^{h'} &= \quantumq^{h+h'}, &&\text{for }h,h' \in \sfP^*,\\
    \quantumq^h E&=q^{\langle h, \alpha \rangle} E \quantumq^h, & \quantumq^h F&=q^{-\langle h,\alpha \rangle}F\quantumq^h, && \text{for }h \in \sfP^*,\\
    E^2&=F^2=0, &EF+FE&= \frac{K-K^{-1}}{q-q^{-1}}, &&\text{where
    }K=\quantumq^{h_1+h_2}.
  \end{aligned}\label{reps:eq:1}
\end{equation}

The superalgebra \(\Uqgl\) is made into a Hopf superalgebra
via
the comultiplication \(\Delta\), counit \(\counit\) and
antipode \(S\)
  defined by
\begin{equation}
\begin{aligned}
  \Delta(E)&= E \otimes K^{-1}+1 \otimes E, & \Delta(F)&=F \otimes 1 + K \otimes F,\\
  S(E)&=-EK, & S(F)&=- K^{-1}F,\\
  \Delta(\quantumq^h)&=\quantumq^h \otimes \quantumq^h, &   S(\quantumq^h)&=\quantumq^{-h},\\
  \counit(E)& =\counit(F)=0, & \counit(\quantumq^h)&=1.
\end{aligned}\label{reps:eq:2}
\end{equation}
Moreover, it possesses a \emph{bar involution} defined on the generators by
\begin{equation}
  \label{reps:eq:3}
  \overline E=E, \qquad \overline F=F, \qquad \overline { \quantumq^h}  = \quantumq^{-h},
  \qquad \overline q=q^{-1}.
\end{equation}

\subsection{Representations of \texorpdfstring{$\Uqgl$}{Uq(gl(1|1))}}\label{reps:sec:representations}

It is easy to classify all simple weight representations of
\(\Uqgl\); indeed, they are all finite-dimensional (since
\(E^2=F^2=0\)). Anyway, the category of
\(\Uqgl\)--representations is not semisimple. We refer to
\cite{2013arXiv1308.2047S} for more details in general, and
we restrict in this paper to a semisimple subcategory.

For any integer \(a\) let \(V(a)\) denote the simple
\(\Uqgl\)--representation with highest weight \(a \epsilon_1\). We fix the
grading on \(V(a)\) by letting the highest weight space be in degree 0. 
Then \(V(0)\) is the \(1\)--dimensional trivial representation, and for \(a>0\), the module \(V(a)\) is a \(2\)--dimensional vector space with basis vectors \(v^a_0\) in degree \(0\) and \(v^a_1\) in degree \(1\); the action of \(\Uqgl\) is given by
\begin{equation}
\begin{aligned}
  Ev^a_0&=0, &  Fv^a_0&= v^a_1, & \quantumq^h v^a_0
  &=q^{\langle h,a \epsilon_1\rangle} v^a_0, &  Kv^a_0 & = q^{a} v^a_0,\\
  Ev^a_1&= [a]v^a_0, & Fv^a_1&=0,      & \quantumq^h v^a_1 &=q^{\langle
    h, a \epsilon_1 - \alpha \rangle} v^a_1, &  Kv^a_1 & = q^{a} v^a_1.
\end{aligned}\label{eq:36}
\end{equation}
where, as usual, \([k]\) is the quantum number defined by
\begin{equation}\label{eq:173} [k]=\frac{q^k-q^{-k}}{q-q^{-1}}. 
\end{equation}
In particular, \(V=V(1)\) is the \emph{vector representation}. 

For a sequence
\(\bolda=(a_1,\ldots,a_\ell)\) of nonnegative integers let us denote
\begin{equation}
V(\bolda)=V(a_1) \otimes \cdots \otimes V(a_\ell). \label{eq:141}
\end{equation}
Let \(\catRep\) be
the monoidal subcategory of the category of \(\Uqgl\)--representations
generated by the \(V(a)\)'s for \(a \in \N\): the objects of \(\catRep\) are
exactly \(\{V(\bolda) \suchthat \bolda \in \bigcup_{\ell \geq 0}
\N^\ell\}\). Note that this category is not abelian (it is not even
additive). Anyway, by adding all direct sums and kernels we would get
a monoidal abelian semisimple category

Since \(V(0)\) is the trivial one-dimensional representation and hence the unit of the monoidal structure, it is enough to consider sequences \(\bolda\) not containing \(0\); so, from now on, we will always suppose that our sequences consist of strictly positive integer numbers. If \(a_1+\dots+a_\ell=n\), we will often call the sequence \(\bolda\) a \emph{composition} of \(n\).
 The sequence
\begin{equation}
\compn=(\underbrace{1,\ldots,1}_{n})\label{eq:167}
\end{equation}
will be called the \emph{regular composition} of \(n\). All other
compositions of \(n\) will be called \emph{singular}. Notice that \(V(\compn)=V(1)^{\otimes n}=V^{\otimes n}\) is a tensor power of the vector representation.

\subsubsection{Projections and embeddings}
\label{sec:proj-embedd}
Let \(a,b\) be positive integers. An easy computation (see \cite[Lemma 3.3]{2013arXiv1308.2047S}) shows that \(V(a+b)\) appears as
 a direct summand of \(V(a) \otimes V(b)\) with multiplicity one.
Let us define a projection \(\Phi_{a,b}: V(a) \otimes V(b) \mapto V(a+b)\) by
\begin{equation}
  \label{eq:37}
  \begin{aligned}
    \Phi_{a,b}: v^a_1 \otimes v^b_1 & \longmapsto 0,  &
    \Phi_{a,b}: v^a_1 \otimes v^b_0 & \longmapsto q^{-b} \qbin{a+b-1}{b}  v^{a+b}_1, \\
    \Phi_{a,b}: v^a_0 \otimes v^b_1  & \longmapsto \qbin{a+b-1}{a} v^{a+b}_1, &
    \Phi_{a,b}: v^a_0 \otimes v^b_0 & \longmapsto \qbin{a+b}{a} v^{a+b}_0
  \end{aligned}
\end{equation}
and an inclusion \(\Phi^{a,b}: V(a+b) \mapto V(a) \otimes V(b)\) by
\begin{equation}
  \label{eq:38}
    \Phi^{a,b}:v^{a+b}_1  \longmapsto v^a_1 \otimes v^b_0 + q^{a} v^a_0 \otimes v^b_1, \qquad\qquad
    \Phi^{a,b}:v^{a+b}_0  \longmapsto v^a_0 \otimes v^b_0,
\end{equation}
where as usual we set
\begin{align}
  \label{eq:136}
  [k]! & = [k][k-1] \cdots [1] &&\text{for all } k\geq 1,\\
  \qbin{n}{k} & = \frac{[n]!}{[k]![n-k]!} &&\text{for all } n\geq 1,\, 1 \leq k \leq n.
\end{align}
One can check that the two maps \(\Phi_{a,b}\) and \(\Phi^{a,b}\) are indeed \(\Uqgl\)--equivariant and
\begin{equation}
  \label{eq:83}
  \Phi_{a,b} \Phi^{a,b} = \qbin{a+b}{a} \text{id}.
\end{equation}
One can also check that \(\Phi_{a,b}\) and \(\Phi^{a,b}\) commute with the bar involution (see \textsection{}\ref{reps:sec:luszt-bar-invol} below).

\subsection{Standard and canonical basis}
\label{reps:sec:luszt-bar-invol}

We briefly recall from \cite{MR2759715} some facts about the bar involution and based
modules. For a more detailed introduction see also
\cite[\textsection1.5]{MR1446615}.

\begin{definition}
  A \emph{bar involution} on a \(\Uqgl\)--module \(W\) is an involution
  \(\overline{\phantom x}\) such that \(\overline{x v}= \overline x \cdot
  \overline v \) for all \(x \in \Uqgl\), \(v \in W\).
\end{definition}

Note that \(\overline{v^a_0}=v^a_0\), \(\overline{v^a_1}=v^a_1\)  define a bar involution on \(V(a)\).

Assume we have bar involutions on \(\Uqgl\)--modules \(W,W'\). Let
 \( \Theta'  = 1 + (q^{-1} - q) E \otimes F \in \Uqgl \otimes \Uqgl\)
 and define on \(W \otimes W'\)
\begin{equation}
  \label{reps:eq:30}
  \overline{w \otimes w'} = \Theta'( \overline w \otimes \overline{w'}).
\end{equation}
Since the element \(\Theta'\) satisfies \(  \Theta' \overline \Delta(x) = \Delta(x) \Theta'\) for all \(x \in \Uqgl\), 
it follows 
that \eqref{reps:eq:30} defines a bar involution on \(W \otimes W'\). Moreover, the
identity \( (\Delta \otimes 1) ( \Theta') \Theta'_{12} = (1 \otimes
\Delta) (\Theta') \Theta'_{23}\) 
allows us to repeat
the construction for bigger tensor products, and the result is
independent of the bracketing.

Let \(a \in \Z_{>0}\). We call \(\mathbb{B}_a=\{v^a_0,v^a_1\}\) the
\emph{standard basis} of \(V(a)\). Let moreover
\(\bolda=(a_1,\ldots,a_\ell)\) be a sequence of (strictly) positive
numbers. For any sequence \(\boldeta=(\eta_1,\ldots,\eta_\ell) \in
\{0,1\}^\ell\) we let \(v^\bolda_{\boldeta} = v^{a_1}_{\eta_1} \otimes
\cdots \otimes v^{a_\ell}_{\eta_\ell}\).  The elements of
\begin{equation}
\bbB_{\bolda} = \mathbb{B}_{a_1} \otimes \cdots \otimes \mathbb{B}_{a_\ell}=\{v^\bolda_\boldeta \suchthat \boldeta \in \{0,1\}^\ell\}\label{reps:eq:31}
\end{equation}
are called
the \emph{standard basis vectors} of \(V(\bolda)\).

On the elements of \eqref{reps:eq:31} we fix a partial ordering
induced from the Bruhat ordering on permutations, as follows. The
symmetric group \(\bbS_\ell\) acts from the right on the set of
sequences \(\{0,1\}^\ell\), hence on \(\bbB_\bolda\). The weight space of
\(V(a_1) \otimes \cdots \otimes V(a_\ell)\) of weight
\((a_1 +\cdots+a_\ell )\epsilon_1 - (\ell- k) \alpha\) is spanned by the
subset \((\bbB_\bolda)_k\) of the standard basis \eqref{reps:eq:31}
consisting of elements such that \(\sum_i \eta_i= k\). The action of
\(\bbS_\ell\) on each subset \((\bbB_\bolda)_k\) is transitive;
mapping the identity \(e \in \bbS_\ell\) to the \emph{minimal element}
\begin{equation}\label{eq:168}
\underbrace{v^{a_1}_0 \otimes \cdots \otimes v^{a_k}_0}_k
\otimes \underbrace{v^{a_{k+1}}_1 \otimes \cdots \otimes v^{a_\ell}_1}_{\ell-k}
\end{equation}
determines a bijection 
\begin{equation}
((\bbS_k \times \bbS_{\ell-k})\backslash \bbS_\ell)^{\short}  \xleftrightarrow{\,1 - 1\,} (\bbB_\bolda)_k,\label{eq:206}
\end{equation}
where \(((\bbS_k \times \bbS_{\ell-k})\backslash \bbS_\ell)^{\short}\)
is the set of shortest coset representatives for \((\bbS_k \times
\bbS_{\ell-k})\backslash \bbS_\ell\).  The Bruhat order of the latter
induces a partial order on \((\bbB_\bolda)_k\) and hence on
\(\bbB_\bolda\).  Notice that the minimal element \eqref{eq:168} is
bar invariant.

We have then the following analogue of \cite[Theorem~27.3.2]{MR2759715}:
\begin{theorem}
  \label{reps:thm:2}
  In \(V(\bolda)\) for each
  standard basis element \(v^\bolda_\boldeta\)
  there is a unique bar-invariant element
  \begin{equation}
v^{\canon \bolda}_\boldeta=v^{a_1}_{\eta_1}
  \canon \cdots \canon v^{a_\ell}_{\eta_\ell}\label{eq:169}
  \end{equation}
 such that
  \(v^{\canon \bolda}_\boldeta- v^{\bolda}_\boldeta\)
 is a
  \(q\Z[q]\)--linear combination of elements of the standard basis that
  are smaller than \(v^\bolda_\boldeta\).
\end{theorem}

\begin{proof}
  The proof is completely analogous to \cite[Theorem~27.3.2]{MR2759715}.
\end{proof}

\begin{definition}\label{def:4}
  The elements \eqref{eq:169} constitute the \emph{canonical basis} of \(V(\bolda)\).
\end{definition}

\begin{example}
  \label{reps:ex:1}
  On the two-dimensional weight space of \(V \otimes V\), the bar
  involution is given by
  \begin{equation}\label{eq:150}
      \overline{v_0^1 \otimes v_1^1} = v_0^1 \otimes v_1^1, \qquad\qquad
      \overline{v_1^1 \otimes v_0^1} = v_1^1 \otimes v_0^1 + (q - q^{-1}) v_0^1 \otimes v_1^1.
  \end{equation}
  The canonical basis is then
  \begin{equation}
      v_0^1 \canon v_1^1 = v_0^1 \otimes v_1^1, \qquad\qquad
      v_1^1 \canon v_0^1 = v_1^1 \otimes v_0^1 + q  v_0^1 \otimes v_1^1.
  \end{equation}
\end{example}

\subsubsection{The bilinear form}
\label{sec:bilinear-form}
Fix a sequence of positive integers \(\bolda=(a_1,\ldots,a_\ell)\).
We define a symmetric bilinear form on \(V(\bolda)\) by setting
\begin{equation}
  \label{eq:53}
  (v^\bolda_\boldeta,v^\bolda_\boldgamma)_\bolda = q^{\sum_{i \neq j} \beta^\boldeta_i \beta^\boldeta_j} \qbin{\beta^\boldeta_1 + \cdots + \beta^\boldeta_\ell}{ \beta^\boldeta_1, \ldots, \beta^\boldeta_\ell} \delta_{\eta_1}^{\gamma_1}\cdots \delta_{\eta_\ell}^{\gamma_\ell}
\end{equation}
where \(\delta_i^j\) is the Kronecker delta,
\begin{equation}
  \label{eq:55}
  \beta^\boldeta_j = a_j - \eta_j =
  \begin{cases}
    a_j - 1 & \text{if } \eta_j=1,\\
    a_j & \text{otherwise}
  \end{cases}
\end{equation}
and
\begin{equation}
  \qbin{k_1+\cdots+k_\ell}{k_1,\ldots,k_\ell} = \frac{[k_1+\cdots+k_\ell]!}{[k_1]! \cdots [k_\ell]!}.\label{eq:5}
\end{equation}
Note that \(q^{\sum_{i \neq j} \beta^\boldeta_i \beta^\boldeta_j}\) in \eqref{eq:53} is exactly the factor needed so that the value of \eqref{eq:53} is a polynomial in \(q\) with constant term \(1\). We introduce the following non-standard notation:
\begin{equation}
  \label{eq:34}
  \begin{aligned}
    {[k]}_0 &= q^{k-1} [k],  & \qbin{a+b}{a}_0&= q^{ab} \qbin{a+b}{a},\\
    {[k]}_0! &= q^{\frac{k(k-1)}{2}} [k]! &
    \qbin{k_1+\cdots+k_\ell}{k_1,\ldots,k_\ell}_0& =
    q^{\sum_{i \neq j} k_i k_j}
    \qbin{k_1+\cdots+k_\ell}{k_1,\ldots,k_\ell}.
  \end{aligned}
\end{equation}
These are rescaled versions of the quantum numbers and
factorials and of the quantum binomial and multinomial
coefficients so that they are actual polynomials in \(q\)
with constant term \(1\).
Hence we can rewrite \eqref{eq:53} as
\begin{equation}
  \label{eq:99}
  (v^\bolda_\boldeta,v^\bolda_\boldgamma)_\bolda = \qbin{\beta^\boldeta_1 + \cdots + \beta^\boldeta_\ell}{ \beta^\boldeta_1, \ldots, \beta^\boldeta_\ell}_0 \delta_{\eta_1}^{\gamma_1}\cdots \delta_{\eta_\ell}^{\gamma_\ell} .
\end{equation}
Notice that we have
\begin{equation}\label{eq:98}
    {[k]}_0 ! = {[k]}_0 {[k-1]}_0 \cdots {[2]}_0 \quad \text{and} \quad
    \qbin{k_1+\cdots+k_\ell}{k_1,\ldots,k_\ell}_0 =
    \frac{{[k_1+\cdots+k_\ell]}_0!}{{[k_1]}_0! \cdots {[k_\ell]}_0!}.
\end{equation}

\begin{lemma}
   \label{lem:7}
     For all \(v \in V(a) \otimes V(b)\) and \(v' \in V(a+b)\) we have
   \begin{equation}
     \label{eq:56}
     (\Phi_{a,b} (v),v')_{(a+b)} = (v, q^{-ab} \Phi^{a,b}(v'))_{(a,b)}.
   \end{equation}
\end{lemma}

\begin{proof}
  This is a straightforward calculation on the basis vectors, which we omit.
\end{proof}

\begin{lemma}
  \label{lem:4}
  For all standard basis vectors \(v_\boldeta^\bolda,v_\boldgamma^\bolda \in V(\bolda)\) we have
  \begin{equation}
    \label{eq:75}
    (Fv_\boldeta^\bolda,v_\boldgamma^\bolda)_{\bolda} =\frac{q^{a_1+\cdots+a_\ell-1}}{[\beta^\boldeta_1+\cdots+\beta^\boldeta_\ell]_0} (v_\boldeta^\bolda,Ev_\boldgamma^\bolda)_{\bolda}.
  \end{equation}
\end{lemma}

\begin{proof}
  Suppose that there exists an index \(r\) such that \(\eta_i=\gamma_i\) for all \(i \neq r\) and \(\eta_r=1\), \(\gamma_r=0\) (otherwise both sides of \eqref{eq:75} are zero). Up to a sign (that we ignore, because it is the same in both formulas), we have
  \begin{equation*}
    (Fv_\boldeta^\bolda,v_\boldgamma^\bolda)_\bolda = (q^{a_1 + \cdots + a_{r-1}} v_\boldgamma^\bolda,v_\boldgamma^\bolda)_\bolda =  q^{a_1 + \cdots + a_{r-1}} \qbin{\beta^\boldgamma_1 + \cdots + \beta^\boldgamma_\ell}{ \beta^\boldgamma_1, \ldots, \beta^\boldgamma_\ell}_0
  \end{equation*}
  and
  \begin{equation*}
    (v_\boldeta^\bolda,Ev_\boldgamma^\bolda)_\bolda = ( v_\boldeta^\bolda,[a_r] q^{-a_{r+1}-\cdots -a_\ell} v_\boldeta^\bolda)_\bolda
    = [a_r] q^{-a_{r+1}-\cdots -a_\ell} \qbin{\beta^\boldeta_1 + \cdots + \beta^\boldeta_\ell}{ \beta^\boldeta_1, \ldots, \beta^\boldeta_\ell}_0.
  \end{equation*}
Since \(\beta^\boldeta_i=\beta^\boldgamma_i\) for \(i\neq r\) while \(\beta^\boldeta_r=\beta^\boldgamma_r+1=a_r\), we have
  \begin{equation}
    \label{eq:78}
    \begin{aligned}
      \frac {(v_\boldeta^\bolda,Ev_\boldgamma^\bolda)_\bolda}{(Fv_\boldeta^\bolda,v_\boldgamma^\bolda)_\bolda} & = [a_r]\frac{[\beta^\boldeta_1 + \cdots +
        \beta^\boldeta_\ell]}{[\beta^\boldeta_r]} q^{\beta^\boldeta_1+ \cdots +\hat \beta^\boldeta_r +
        \cdots + \beta^\boldeta_\ell} q^{-a_1 - \cdots- \hat a_r - \cdots - a_\ell}\\
      & = [\beta^\boldeta_1 + \cdots +
        \beta^\boldeta_\ell]_0 \, q^{1-a_1 - \cdots- a_\ell},
    \end{aligned}
  \end{equation}
  that proves the claim.
\end{proof}

\begin{remark}
  \label{rem:3}
  If we enlarge \(\Uqgl\) with a new generator \(E'\) such that
  \begin{equation}
    \label{eq:79}
    E = q  \frac{q^{2h_1} - 1}{q^2-1}E'  K^{-1}
  \end{equation}
  then we get an adjunction between \(F\) and \(E'\).
\end{remark}

\subsubsection{Dual standard and dual canonical basis}
\label{sec:dual-standard-dual}

We define the \emph{dual standard basis} \(\{v^{\dualstan\bolda}_\boldeta  \suchthat \eta \in \{0,1\}^\ell\}\) of \(V(\bolda)\) to be the basis dual to the standard basis with respect to the bilinear form \((\cdot,\cdot)_\bolda\):
\begin{equation}
  \label{eq:153}
  (v^\bolda_\boldeta,v^{\dualstan\bolda}_\boldgamma)_\bolda =
  \begin{cases}
    1 & \text{if } \boldeta=\boldgamma,\\
    0 & \text{otherwise}.
  \end{cases}
\end{equation}
Of course, since the standard basis is already orthogonal, each \(v^{\dualstan\bolda}_\boldeta\) is a multiple of \(v^\bolda_\boldeta\). In particular, one has
\begin{equation}
  \label{eq:155} \qbin{\beta^\boldeta_1+\cdots+\beta^\boldeta_\ell}{\beta^\boldeta_1,\ldots,\beta^\boldeta_\ell}_0 v^{\dualstan\bolda}_\boldeta =  v^\bolda_\boldeta.
\end{equation}

Moreover, we define the \emph{dual canonical basis} \(\{v^{\dualcanon\bolda}_\boldeta  \suchthat \eta \in \{0,1\}^\ell\}\) of \(V(\bolda)\) to be the basis dual to the canonical basis with respect to the bilinear form \((\cdot,\cdot)_\bolda\):
\begin{equation}
  \label{eq:154}
  (v^{\canon\bolda}_\boldeta,v^{\dualcanon\bolda}_\boldgamma)_\bolda =
  \begin{cases}
    1 & \text{if } \boldeta=\boldgamma,\\
    0 & \text{otherwise}.
  \end{cases}
\end{equation}

\subsection{Super Schur-Weyl duality for \texorpdfstring{$V^{\otimes n}$}{tensor powers of V}}
\label{sec:super-schur-weyl}
We now study in more detail the tensor powers of the vector representation. Let us define a linear endomorphism \(\check H\) of \(V \otimes V\) by
\begin{equation}
\begin{aligned}
  \check H (v^1_1 \otimes v^1_1) & = -q v^1_1 \otimes v^1_1, & \check H (v^1_1 \otimes v^1_0) &= v^1_0 \otimes v^1_1 +(q^{-1} -q) v^1_1 \otimes v^1_0, \\
  \check H (v^1_0 \otimes v^1_1) &= v^1_1 \otimes v^1_0, & \check H(v^1_0 \otimes v^1_0) &= q^{-1} v^1_0 \otimes
  v^1_0.
\end{aligned}\label{reps:eq:19}
\end{equation}
By an explicit computation it can be checked that \(\check H\) can be expressed in terms of a projection \eqref{eq:37} and an embedding \eqref{eq:38}:
\begin{equation}
  \label{eq:84}
  \Phi^{1,1}\Phi_{1,1} = \check H + q.
\end{equation}
It follows in particular that \(\check H\) is \(\Uqgl\)--equivariant.

We can consider on \(V^{\otimes n}\) the operators
\begin{equation}
\check H_{i} = \id^{\otimes i-1} \otimes \check H \otimes \id^{\otimes  n-i-1}.\label{eq:158}
\end{equation}
By definition, they are intertwiners for the action of \(\Uqgl\).
One can easily check that
\(\check H_i^2 = (q^{-1}-q)\check H_i + \id\).
  The category of \(\Uqgl\)--representation is braided (see \cite{2013arXiv1308.2047S}), and the endomorphism \(\check H\) is just the inverse of the braiding. From this it follows directly that \(\check H\) is equivariant and that the braid relation \(\check H_i \check H_{i+1} \check H_i = \check H_{i+1} \check H_i \check H_{i+1}\) holds for all \(i=1,\ldots,n-1\). Since clearly \(\check H_i \check H_j = \check H_j \check H_i\) for \(\abs{i-j}>1\), it follows that the operators \(\check H_i\) define on \(V^{\otimes n}\) an action of the Hecke algebra \(\ucalH_n\), which we regard as a right action.
The following result is also known as super Schur-Weyl duality. The non quantized version was originally proved by Berele and Regev (\cite{MR884183}) and independently by Sergeev (\cite{MR735715}).

\begin{prop}[\cite{MR2251378}]
  \label{prop:4}
  The map \( \ucalH_n  \rightarrow \End_{\Uqgl}(V^{\otimes n})\) defined by
 \(     H_i  \longmapsto \check H_{i}\)
  is surjective. As a module for \(\ucalH_n\) we have
  \begin{equation}
V^{\otimes n}=\bigoplus_{k=1}^n \big( S( \mu_{n,k}) \oplus S( \mu_{n,k}) \big),\label{eq:91}
\end{equation}
where \(\mu_{n,k}\)
  is the hook partition \((k,1^{n-k})\) and \(S(\mu_{n,k})\) is the \(q\)--version of the
  corresponding Specht module.
\end{prop}

As often occurs with the Hecke algebra, it is more convenient to choose generators \(C_i = H_i + q\). We introduce the Super Temperley-Lieb Algebra as follows:
\begin{subequations}
\begin{definition}\label{def:6}
  For \(n\geq 1\), the \emph{Super Temperley-Lieb Algebra}
  \(\STL_n\) is the unital associative \(\C(q)\)--algebra generated by
  \(\{C_i \mid i=1,\ldots,n-1\}\) subjected to the relations
  \begin{align}
        C_i^2 & =(q+q^{-1}) C_i, \label{eq:131}\\
        C_i C_j& =C_j C_i,\label{eq:132}\\
        C_i C_{i+1} C_i - C_i & = C_{i+1} C_i C_{i+1} - C_{i+1},\label{eq:133}
      \end{align}
for \(\abs{i-j}>1\), and
      \begin{align}
        C_{i-1} C_{i+1} C_i ((q+q^{-1})-C_{i-1}) ((q+q^{-1})-C_{i+1}) & = 0,\label{eq:134}\\
        ((q+q^{-1})-C_{i-1}) ((q+q^{-1})-C_{i+1}) C_i C_{i-1} C_{i+1}
        & = 0.\label{eq:135}
      \end{align}
\end{definition}
\end{subequations}

Since the first three relations are just the relations that the generators \(C_i = H_i + q\) satisfy in the Hecke algebra, it follows that \(\STL_n\) is a quotient of \(\ucalH_n\). Moreover, one can prove that \eqref{eq:134} and \eqref{eq:135} generate the kernel of the action on \(V^{\otimes n}\) (cf.~\cite{miophd2}), and hence we have
 \( \STL_n \cong \End_{\Uqgl}(V^{\otimes n})\).

Consider the weight space decomposition
\(  V^{\otimes n}= \bigoplus_{k=0}^n (V^{\otimes n})_k\),
where
\begin{equation}
(V^{\otimes n})_k = \{v \in V^{\otimes n} \suchthat \quantumq^h v = q^{\langle h, k\epsilon_1 + (n-k) \epsilon_2 \rangle } v   \}.\label{eq:89}
\end{equation}
Clearly, every weight space is a module for the Hecke algebra. We have:
\begin{prop}
  \label{prop:20}
  Let \(W_{\frakq}=\langle s_1,\ldots,s_{k-1}\rangle\) and \(W_{\frakp} = \langle s_{k+1},\ldots,s_{n-1}\rangle\) as subgroups of\/ \(\bbS_n\). With the notation of \textsection\ref{sec:induc-hecke-modul} we have
  \begin{equation}
    \label{eq:90}
    (V^{\otimes n})_k \cong \ucalM^\frakp_\frakq
  \end{equation}
  as right \(\ucalH_n\)--modules.
The isomorphism is given explicitly by
\begin{equation}
  \label{eq:4}
    \Psi: \ucalM^\frakp_\frakq \longrightarrow (V^{\otimes n})_k, \qquad
   \Psi: N_w  \longmapsto v^\bolda_{\boldeta_{\mathrm{min}} \cdot w},
\end{equation}
where
  \begin{equation}
\boldeta_{\mathrm{min}} = (\underbrace{ 0, \ldots, 0}_{k} , \underbrace{1,\ldots,1}_{n-k}).\label{eq:95}
\end{equation}
and \(\bbS_n\) acts on sequences of \(\{0,1\}^n\) from the right by permutations.
\end{prop}

\begin{proof}
It is straightforward to check that, by the definition of the action of \(\ucalH_n\) on \(V^{\otimes n}\) \eqref{reps:eq:19}, we have \(v_{\boldeta_{\mathrm{min}}} \cdot H_w = v_{\boldeta_{\mathrm{min}} \cdot w}\) whenever \(w \in W^{\frakp+\frakq}\). In particular, \eqref{eq:4} is a bijection. We need to show that the action of the Hecke algebra is the same on both sides. This follows comparing \eqref{eq:ActionOfHOnM} and \eqref{reps:eq:19}.
\end{proof}

As a consequence, there is a second notion of canonical basis on \((V^{\otimes n})_k\), defined using the Hecke algebra action from Section~\ref{sec:hecke-algebra-hecke}. Not surprisingly, this coincides with Lusztig canonical basis (compare with \cite[Theorem~2.5]{MR1657524}, that deals with the case of \(\mathfrak{sl}_2\)):

\begin{prop}
  \label{prop:19}
  Under the isomorphism \(\Psi\) \eqref{eq:4}, the canonical basis element \(\underline N_w\) of \(\ucalM^{\frakp}_{\frakq}\) is mapped to the canonical basis element \(v^{\canon \bolda}_{\boldeta_\text{min} \cdot w}\).
\end{prop}

\begin{proof}
  By the uniqueness results (Proposition~\ref{prop:18} and Theorem~\ref{reps:thm:2}), it is enough to show that the bar involution of \(\ucalM^{\frakp}_{\frakq}\) is mapped to the bar involution of \((V^{\otimes n})_k\) under \eqref{eq:4}. On \(\ucalM^{\frakp}_\frakq\) the bar involution is uniquely determined by \(\overline{N_e}=N_e\) and \(\overline{X H_i}  = \overline{X} \overline{H_i} = \overline{X } H^{-1}_i\) for all \(X \in \ucalM^{\frakp}_\frakq\). It is enough to show that the same holds for Lusztig bar involution on \((V^{\otimes n})_k\). Of course \(\overline{v_{\boldeta_{\text{min}}}}=v_{\boldeta_{\text{min}}}\), and one can show by standard methods (cf.\ \cite[Lemma~2.3]{MR2491760}) that
 \(   \overline{\check H_{i}(v_\boldeta)} = \check H^{-1}_i (\overline{v_\boldeta})\)
for all standard basis elements \(v_\boldeta\).
\end{proof}

The form \(\langle \cdot,\cdot \rangle\) on \(\ucalM^\frakp_\frakq\) and the form \((\cdot,\cdot)_{\bolda}\) on \((V^{\otimes n})_k\) are proportional under \(\Psi\):

\begin{lemma}
  \label{lem:19}
  Let \(\Psi\) be the isomorphism \eqref{eq:4}. Then
  \begin{equation}\label{eq:159}
    (\Psi(X),\Psi(Y))_\compn = [k]_0!\, \langle X,Y \rangle \qquad \text{for all } X,Y \in \ucalM^{\frakp}_\frakq.
  \end{equation}
\end{lemma}

\begin{proof}
  It is enough to check \eqref{eq:159} on the standard basis \(\{N_w \suchthat w \in W^{\frakp+\frakq}\}\) of \(\ucalM^{\frakp}_\frakq\). We have
  \begin{equation}
    \label{eq:160}
    (\Psi(N_w),\Psi(N_z))_\compn = (v_{\boldeta_\text{min} \cdot w} ,v_{\boldeta_{\text{min}} \cdot z} )_\compn =
    \begin{cases}
      [k]_0! & \text{if } w=z,\\
      0 & \text{otherwise}.
    \end{cases}
  \end{equation}
  By definition, this is the same as \([k]_0!\, \langle N_w,N_z \rangle\).
\end{proof}

\section{Diagrams for the intertwining operators}
\label{reps:sec:diagr-intertw-oper}

In this section we will provide a diagram calculus for the
intertwining operators in the category \(\catRep\).

\subsection{The category \texorpdfstring{$\catWeb$}{Web}}
\label{sec:category-catweb}

We start by defining a diagrammatical category \(\catWeb\). We remark that the category \(\catWeb\) is similar to the category of \(\mathrm{SL}_n\)--spiders (see \cite{MR1403861}, \cite{MR2704398}, \cite{MR2710589} and \cite{MR3263166}) which describes intertwining operators of representations of \(U_q(\gl_n)\).

A \emph{web
diagram} is an oriented plane graph with edges labeled by positive
integers. For simplicity we suppose that the edges do not have maxima or minima, and the orientation is then uniquely determined by orienting all edges upwards. Only single and triple vertices are allowed. Single vertices
must lie on the bottom (respectively, top) line if they are sources
(respectively, targets) for the corresponding edge. Around a triple vertex,
the sum of the labels of the ingoing edges must agree with the sum of the labels of the outgoing vertices; this means that
only the following labelings are allowed for arbitrary strictly positive
numbers \(a,b\):
\begin{equation}
      \begin{tikzpicture}[>=angle 45,scale=1,yscale=0.5,anchorbase]
        \draw[int] (0.5,0) node [below] {\(a\)} -- (1,1);
        \draw[int] (1.5,0) node [below] {\(b\)} -- (1,1);
        \draw[int] (1,1) -- (1,2) node [above] {\(a+b\)} ;
      \end{tikzpicture}      \qquad \qquad
      \begin{tikzpicture}[anchorbase]
        \draw[very thick,->,color=black!50!white,double] (0,0) -- (0,1);
      \end{tikzpicture}
      \qquad \qquad
      \begin{tikzpicture}[>=angle 45,scale=1,yscale=0.5,anchorbase]
        \draw[int] (0.5,2) node [above] {\(a\)} -- (1,1);
        \draw[int] (1.5,2) node [above] {\(b\)} -- (1,1);
        \draw[int] (1,1) -- (1,0) node [below] {\(a+b\)} ;
      \end{tikzpicture}
\label{eq:45}
\end{equation}
We will not draw the orientation of the edges, because they are all oriented upwards. The \emph{source} of a web is the sequence \(\bolda=(a_1,\ldots,a_\ell)\) of labels on the bottom line. The \emph{target} is the sequence \(\bolda'=(a_1',\ldots,a_s')\) on the top line.

If we have two webs \(\psi,\phi\) and the target of \(\phi\) is the same as the source of \(\psi\), then we can compose \(\psi\) and \(\phi\) by concatenating vertically. Additionally, we can always concatenate two webs \(\phi,\psi\) horizontally, putting the second on the right of the first; in this case we use a tensor product symbol:
\begin{equation*}
  \psi \circ \phi =
  \begin{aligned}
    \begin{tikzpicture}
      \draw (0,0) rectangle node {\(\phi\)}(1.5,1) rectangle node {\(\psi\)} (0,2);
    \end{tikzpicture}
  \end{aligned}
\qquad \qquad
  \phi \otimes \psi =
  \begin{aligned}
    \begin{tikzpicture}
      \draw (0,0) rectangle node {\(\phi\)} (1.5,1.5) rectangle node {\(\psi\)}(3,0);
    \end{tikzpicture}
  \end{aligned}
\end{equation*}

The category \(\catWeb'\) is the monoidal category whose objects are sequences \(\bolda=(a_1,\ldots,a_\ell)\) of strictly positive integers; a morphism from \(\bolda\) to \(\bolda'\) is a \(\C(q)\)--linear
combination of web diagrams with source \(\bolda\) and target \(\bolda'\).
Composition of morphisms corresponds to vertical concatenation of web diagrams. Horizontal concatenation of web diagrams gives, on the other side, a monoidal structure on \(\catWeb'\), whose unit is the empty sequence \(()\).

\begin{definition}\label{def:1}
We define the category \(\catWeb\) to be the quotient of \(\catWeb'\) by the
 following relations:
\begin{subequations}
\begin{align}
  \parbox{6cm}{\centering Orientation preserving isotopy\\(with source and target points fixed)} \hspace{-3.5cm} & \displaybreak[0]\\
      \begin{tikzpicture}[>=angle 45,scale=1,yscale=0.5,anchorbase]
        \draw[int] (1,-1) -- (0.5,0) node [left] {\(\vphantom{b}a\)} -- (1,1);
        \draw[int] (1,-1) -- (1.5,0) node [right] {\(b\)} -- (1,1);
        \draw[int] (1,1) -- (1,2) node [above] {\(a+b\)} ;
        \draw[int] (1,-1) -- (1,-2) node [below] {\(a+b\)} ;
      \end{tikzpicture} &= \qbin{a+b}{a}\mspace{-10mu}
      \begin{tikzpicture}[>=angle 45,scale=1,yscale=0.5,anchorbase]
        \draw[int] (0,-2) node[below] {\(a+b\)} -- (0,2) node[above] {\(a+b\)};
      \end{tikzpicture}\label{eq:O53}\displaybreak[0]\\
      \begin{tikzpicture}[>=angle 45,scale=0.7,yscale=0.5,anchorbase]
        \draw[int] (0.5,0) node [below] {\(\vphantom{b}a\)} -- (1,1);
        \draw[int] (1.5,0) node [below] {\(b\)} -- (1,1);
        \draw[int] (1,1) -- (1.5,2);
        \draw[int] (2.5,0) node [below] {\(\vphantom{b}c\)} -- (1.5,2);
        \draw[int] (1.5,2) -- (1.5,3) node [above] {\(a+b+c\)};
      \end{tikzpicture}  =
      \begin{tikzpicture}[>=angle 45,x=-1cm,scale=0.7,yscale=0.5,anchorbase]
        \draw[int] (0.5,0) node [below] {\(\vphantom{b}c\)} -- (1,1);
        \draw[int] (1.5,0) node [below] {\(b\)} -- (1,1);
        \draw[int] (1,1) -- (1.5,2);
        \draw[int] (2.5,0) node [below] {\(\vphantom{b}a\)} -- (1.5,2);
        \draw[int] (1.5,2) -- (1.5,3) node [above] {\(a+b+c\)};
      \end{tikzpicture} \quad &\text{and} \quad 
      \begin{tikzpicture}[>=angle 45,scale=0.7,y=-1cm,yscale=0.5,anchorbase]
        \draw[int] (0.5,0) node [above] {\(a\)} -- (1,1);
        \draw[int] (1.5,0) node [above] {\(b\)} -- (1,1);
        \draw[int] (1,1) -- (1.5,2);
        \draw[int] (2.5,0) node [above] {\(c\)} -- (1.5,2);
        \draw[int] (1.5,2) -- (1.5,3) node [below] {\(a+b+c\)};
      \end{tikzpicture}  =
      \begin{tikzpicture}[>=angle 45,x=-1cm,y=-1cm,scale=0.7,yscale=0.5,anchorbase]
        \draw[int] (0.5,0) node [above] {\(c\)} -- (1,1);
        \draw[int] (1.5,0) node [above] {\(b\)} -- (1,1);
        \draw[int] (1,1) -- (1.5,2);
        \draw[int] (2.5,0) node [above] {\(a\)} -- (1.5,2);
        \draw[int] (1.5,2) -- (1.5,3) node [below] {\(a+b+c\)};
      \end{tikzpicture}\label{eq:44}\displaybreak[0]\\
      \begin{tikzpicture}[>=angle 45,xscale=0.5,yscale=0.5,anchorbase,smallnodes]
        \draw[int] (0,0) node[below] {\(1\)} -- (0.5,0.5) -- (0.5,1) -- (0,1.5) -- (0,3) -- (0.5,3.5) -- (0.5,4) -- (0,4.5) node [above] {\(1\)};
        \draw[int] (1,0) node[below] {\(1\)} -- (0.5,0.5) -- (0.5,1) -- (1,1.5) -- (1.5,2) -- (1.5,2.5) -- (1,3) -- (0.5,3.5) -- (0.5,4) -- (1,4.5) node [above] {\(1\)};
        \draw[int] (2,0) node[below] {\(1\)}  -- (2,1.5) -- (1.5,2) -- (1.5,2.5) -- (2,3) -- (2,4.5) node [above] {\(1\)};
        \draw[very thick] (1.5,2) -- node[right] {\(2\)} ++(0,0.5);
        \draw[very thick] (0.5,0.5) -- node[left] {\(2\)} ++(0,0.5);
        \draw[very thick] (0.5,3.5) -- node[left] {\(2\)} ++(0,0.5);
      \end{tikzpicture} + 
      \begin{tikzpicture}[>=angle 45,xscale=0.5,yscale=0.5,anchorbase,smallnodes]
        \draw[int] (0,0) node[below] {\(1\)} -- (0,4.5) node [above] {\(1\)};
        \draw[int] (1,0) node[below] {\(1\)} -- (1,1.5) -- (1.5,2) -- (1.5,2.5) -- (1,3) -- (1,4.5) node [above] {\(1\)};
        \draw[int] (2,0) node[below] {\(1\)}  -- (2,1.5) -- (1.5,2) -- (1.5,2.5) -- (2,3) -- (2,4.5) node [above] {\(1\)};
        \draw[very thick] (1.5,2) -- node[right] {\(2\)} ++(0,0.5);
      \end{tikzpicture}  
&=
      \begin{tikzpicture}[>=angle 45,yscale=0.5,anchorbase,x=-0.5cm,smallnodes]
        \draw[int] (0,0) node[below] {\(1\)} -- (0.5,0.5) -- (0.5,1) -- (0,1.5) -- (0,3) -- (0.5,3.5) -- (0.5,4) -- (0,4.5) node [above] {\(1\)};
        \draw[int] (1,0) node[below] {\(1\)} -- (0.5,0.5) -- (0.5,1) -- (1,1.5) -- (1.5,2) -- (1.5,2.5) -- (1,3) -- (0.5,3.5) -- (0.5,4) -- (1,4.5) node [above] {\(1\)};
        \draw[int] (2,0) node[below] {\(1\)}  -- (2,1.5) -- (1.5,2) -- (1.5,2.5) -- (2,3) -- (2,4.5) node [above] {\(1\)};
        \draw[very thick] (1.5,2) -- node[left] {\(2\)} ++(0,0.5);
        \draw[very thick] (0.5,0.5) -- node[right] {\(2\)} ++(0,0.5);
        \draw[very thick] (0.5,3.5) -- node[right] {\(2\)} ++(0,0.5);
      \end{tikzpicture} + 
      \begin{tikzpicture}[>=angle 45,yscale=0.5,x=-0.5cm,anchorbase,smallnodes]
        \draw[int] (0,0) node[below] {\(1\)} -- (0,4.5) node [above] {\(1\)};
        \draw[int] (1,0) node[below] {\(1\)} -- (1,1.5) -- (1.5,2) -- (1.5,2.5) -- (1,3) -- (1,4.5) node [above] {\(1\)};
        \draw[int] (2,0) node[below] {\(1\)}  -- (2,1.5) -- (1.5,2) -- (1.5,2.5) -- (2,3) -- (2,4.5) node [above] {\(1\)};
        \draw[very thick] (1.5,2) -- node[left] {\(2\)} ++(0,0.5);
      \end{tikzpicture}\label{eq:54}  
\end{align}
\end{subequations}
\end{definition}

We define the two \emph{elementary webs} \(\webjoin_{\bolda,i}\) and \(\websplit^{\bolda,i}\) by the diagrams
\begin{align}
  \label{catO:eq:194}
  \webjoin_{\bolda,i} & =
      \begin{tikzpicture}[every node/.style={font=\normalsize},>=angle 45,xscale=1.3,yscale=0.45,baseline =(a.base)]
        \draw[int] (-1,0) node[below] {\(a_1\)} -- ++(0,2) node[above] {\(a_1\)};
        \node (a) at (-0.5,1) {\(\cdots\)};
        \draw[int] (0,0) node[below] {\(a_{i-1}\)} -- ++(0,2) node[above] {\(a_{i-1}\)};
        \draw[int] (0.65,0) node [below] {\(a_i\)} -- (1,1);
        \draw[int] (1.35,0) node [below] {\(a_{i+1}\)} -- (1,1);
        \draw[int] (1,1) -- (1,2) node[above] {\(a_{i}+a_{i+1}\)};
        \draw[int] (2,0)  node[below] {\(a_{i+2}\)} -- ++(0,2) node[above] {\(a_{i+2}\)};
        \node at (2.5,1) {\(\cdots\)};
        \draw[int] (3,0) node[below] {\(a_\ell\)} -- ++(0,2) node[above] {\(a_\ell\)};
      \end{tikzpicture}\displaybreak[0]\\
  \websplit^{\bolda,i} & =
      \begin{tikzpicture}[every node/.style={font=\normalsize},>=angle 45,yscale=0.45,xscale=1.3,baseline =(a.base)]
        \draw[int] (-1,0) node[below] {\(a_1\)} -- ++(0,2) node[above] {\(a_1\)};
        \node (a) at (-0.5,1) {\(\cdots\)} ;
        \draw[int] (0,0) node[below] {\(a_{i-1}\)} -- ++(0,2) node[above] {\(a_{i-1}\)};
        \draw[int] (0.65,2) node [above] {\(a_i\)} -- (1,1);
        \draw[int] (1.35,2) node [above] {\(a_{i+1}\)} -- (1,1) ;
        \draw[int] (1,0) node[below] {\(a_i+a_{i+1}\)} -- (1,1);
        \draw[int] (2,0) node[below] {\(a_{i+2}\)} -- ++(0,2) node[above] {\(a_{i+2}\)};
        \node at (2.5,1) {\(\cdots\)};
        \draw[int] (3,0) node[below] {\(a_\ell\)} -- ++(0,2) node[above] {\(a_\ell\)};
      \end{tikzpicture}\label{catO:eq:198}
\end{align}
and notice that the category \(\catWeb\) is generated by such elementary web diagrams. Given a sequence \(\bolda=(a_1,\ldots,a_\ell)\) we let also
\begin{equation}
  \label{catO:eq:196}
  \boldsymbol{\hat a}_i = (a_1,\ldots,a_{i-1},a_i+a_{i+1},a_{i+2},\ldots,a_\ell)
\end{equation}
be the target of \(\webjoin_{\bolda,i}\) (and the source of \(\websplit^{\bolda,i}\)).

It will be useful to consider also multivalent vertices with only one outgoing (respectively, ingoing) edge: we define them to be equal to concatenations of trivalent vertices \eqref{eq:44}. For example:
\begin{equation}
  \label{catO:eq:180}
        \begin{tikzpicture}[>=angle 45,scale=0.5,anchorbase]
        \draw[int] (0,0) node [above] {\(a\)} -- (1.5,-1.5);
        \draw[int] (1,0) node [above] {\(b\)} -- (1.5,-1.5);
        \draw[int] (2,0) node [above] {\(c\)} -- (1.5,-1.5);
        \draw[int] (3,0) node [above] {\(d\)} -- (1.5,-1.5);
        \draw[int] (1.5,-1.5) -- (1.5,-2) node [below] {\(a+b+c\)};
      \end{tikzpicture} \eqdef
        \begin{tikzpicture}[>=angle 45,scale=0.5,anchorbase]
        \draw[int] (0,0) node [above] {\(a\)} -- (1.5,-1.5);
        \draw[int] (1,0) node [above] {\(b\)} -- (0.5,-0.5);
        \draw[int] (2,0) node [above] {\(c\)} -- (1,-1);
        \draw[int] (3,0) node [above] {\(d\)} -- (1.5,-1.5);
        \draw[int] (1.5,-1.5) -- (1.5,-2) node [below] {\(a+b+c\)};
      \end{tikzpicture} =
        \begin{tikzpicture}[>=angle 45,scale=0.5,anchorbase]
        \draw[int] (0,0) node [above] {\(a\)} -- (1.5,-1.5);
        \draw[int] (1,0) node [above] {\(b\)} -- (0.5,-0.5);
        \draw[int] (2,0) node [above] {\(c\)} -- (2.5,-0.5);
        \draw[int] (3,0) node [above] {\(d\)} -- (1.5,-1.5);
        \draw[int] (1.5,-1.5) -- (1.5,-2) node [below] {\(a+b+c\)};
      \end{tikzpicture} =
        \begin{tikzpicture}[>=angle 45,scale=0.5,anchorbase]
        \draw[int] (0,0) node [above] {\(a\)} -- (1.5,-1.5);
        \draw[int] (1,0) node [above] {\(b\)} -- (2,-1);
        \draw[int] (2,0) node [above] {\(c\)} -- (2.5,-0.5);
        \draw[int] (3,0) node [above] {\(d\)} -- (1.5,-1.5);
        \draw[int] (1.5,-1.5) -- (1.5,-2) node [below] {\(a+b+c\)};
      \end{tikzpicture}
\end{equation}
Notice that this is well-defined by \eqref{eq:44}.

Let us define the web diagrams
\begin{equation}
  \label{eq:68}
  \webjoinmultiple^{n}=
  \begin{tikzpicture}[>=angle 45,scale=0.5,anchorbase,smallnodes]
    \draw[int] (0,0) node [below] {\(1\)} -- (1.5,1.5);
    \draw[int] (1,0) node [below] {\(1\)} -- (1.5,1.5);
    \node at (2,0) {\(\dotsb\)};
    \draw[int] (3,0) node [below] {\(1\)} -- (1.5,1.5);
    \draw[very thick] (1.5,1.5) --  (1.5,2) node [above] {\(n\)};  
  \end{tikzpicture} \qquad \text{and} \qquad
  \websplitmultiple_n=
  \begin{tikzpicture}[>=angle 45,scale=0.5,anchorbase,smallnodes]
    \draw[int] (1.5,-1.5) -- (0,0) node [above] {\(1\)};
    \draw[int] (1.5,-1.5) -- (1,0) node [above] {\(1\)};
    \node at (2,0) {\(\dotsb\)};
    \draw[int] (1.5,-1.5) -- (3,0) node [above] {\(1\)};
    \draw[very thick] (1.5,-2) node [below] {\(n\)} --  (1.5,-1.5);  
  \end{tikzpicture}
\end{equation}

From \eqref{eq:O53} it follows in particular that
\begin{equation}
  \label{eq:66}
  \webjoinmultiple^n \circ \websplitmultiple_n = 
        \begin{tikzpicture}[>=angle 45,scale=0.5,anchorbase,smallnodes]
        \draw[int] (1.5,-1.5) -- (0,0) node [left=-0.1cm] {\(1\)} -- (1.5,1.5);
        \draw[int] (1.5,-1.5) -- (1,0) node [left=-0.1cm] {\(1\)} -- (1.5,1.5);
        \node at (2,0) {\(\dotsb\)};
        \draw[int] (1.5,-1.5) -- (3,0) node [right=-0.1cm] {\(1\)} -- (1.5,1.5);
        \draw[very thick] (1.5,-2) node [below] {\(n\)} --  (1.5,-1.5);  
        \draw[very thick] (1.5,1.5) --  (1.5,2) node [above] {\(n\)};  
      \end{tikzpicture} = [n]!
      \begin{tikzpicture}[>=angle 45,scale=0.5,anchorbase,smallnodes]
        \draw[very thick] (1.5,-2) node [below] {\(n\)} --  (1.5,2) node [above] {\(n\)};  
      \end{tikzpicture}
\end{equation}

Given a composition \(\bolda=(a_1,\dotsc,a_\ell)\) of \(n\), notice that we have a \emph{standard inclusion}
\begin{equation}
\websplitmultiple_{a_1} \otimes \dotsb \otimes \websplitmultiple_{a_\ell}\colon \bolda \mapto \compn\label{eq:142}
\end{equation}
and a \emph{standard projection}
\begin{equation}
\webjoinmultiple^{a_1} \otimes \dotsb \otimes \webjoinmultiple^{a_\ell}\colon  \compn \mapto \bolda.\label{eq:144}
\end{equation}

\subsection{Webs as intertwiners}
\label{sec:webs-as-intertw}

Now we are going to define a monoidal functor \(\funcT: \catWeb \mapto \catRep\). On objects we set \(\funcT(\bolda) = V(\bolda)\). To define \(\funcT\) on morphisms,
it suffices to consider the two diagrams \eqref{eq:45}, which monoidally generate \(\catRep\).
We set:
  \begin{equation}
    \label{eq:29}
\funcT\left(
      \begin{tikzpicture}[>=angle 45,scale=1,yscale=0.4,anchorbase]
        \draw[int] (0.5,0) node [below] {\(a\)} -- (1,1);
        \draw[int] (1.5,0) node [below] {\(b\)} -- (1,1);
        \draw[int] (1,1) -- (1,2) node [above] {\(a+b\)} ;
      \end{tikzpicture} \right) =
      \begin{tikzpicture}[>=angle 45,scale=1,yscale=0.4,anchorbase]
      \draw[->] (3,0) node[below] {\(V(a) \otimes V(b)\)} -- node[right] {\(\Phi_{a,b}\)} ++(up:2cm)
        node[above] {\(V(a+b)\)};
      \end{tikzpicture}
 \quad \text{and} \quad\funcT\left(
      \begin{tikzpicture}[>=angle 45,scale=1,yscale=0.4,anchorbase]
        \draw[int] (0.5,2) node [above] {\(a\)} -- (1,1);
        \draw[int] (1.5,2) node [above] {\(b\)} -- (1,1);
        \draw[int] (1,1) -- (1,0) node [below] {\(a+b\)} ;
      \end{tikzpicture} \right) =
      \begin{tikzpicture}[>=angle 45,scale=1,yscale=0.4,anchorbase]
      \draw[->] (3,0) node[below] {\(V(a+b)\)} -- node[right] {\(\Phi^{a,b}\)} ++(up:2cm) node[above] {\(V(a) \otimes V(b)\)};
      \end{tikzpicture}
  \end{equation}

  \begin{prop}
    \label{prop:2}
    The assignment \eqref{eq:29} defines a dense full monoidal functor \(\funcT\colon \catWeb \mapto \catRep\).
  \end{prop}
  \begin{proof}
  First, we have to check that \(\funcT\) satisfies the relations defining \(\catWeb\). It is straightforward to check that \(\funcT\) respects relations \eqref{eq:44}. Relation \eqref{eq:O53} is satisfied thanks to \eqref{eq:83}. Relation \eqref{eq:54} is satisfied thanks to \eqref{eq:133}.

The functors \(\funcT\) is dense since, by definition, the objects of \(\catRep\) are exactly the \(V(\bolda)\) for all sequences \(\bolda\) of positive integer numbers. We prove now that \(\funcT\) is full. By Proposition~\ref{prop:4} and \eqref{eq:84} the map \(\Hom_\catWeb(\compn,\compn) \mapto \Hom_\catRep(V^{\otimes n}, V^{\otimes n})\) induced by \(\funcT\) is surjective. Since \(\Hom_{\Uqgl}(V^{\otimes m}, V^{\otimes n})=0\) unless \(m = n\), it follows more in general that the map \(\Hom_\catWeb(\compm,\compn) \mapto \Hom_\catRep(V^{\otimes m} , V^{\otimes n})\) induced by \(\funcT\) is surjective for all \(m,n\). Now each representation \(V(\bolda) \in \catRep\) embeds in some \(V^{\otimes n}\), and the corresponding inclusion and projection are images under \(\funcT\) of the standard inclusion \eqref{eq:142} and of the standard projection \eqref{eq:144}. Hence \(\funcT\) induces a surjective map \(\Hom_\catWeb(\bolda,\bolda') \mapto \Hom_\catRep(V(\bolda),V(\bolda'))\) for all sequences \(\bolda,\bolda'\).
  \end{proof}

  \begin{remark}
    \label{rem:12}
    It is actually possible to describe explicitly additional relations on the category \(\catRep\) so that \(\funcT\) descends to an equivalence of categories (cf.\ \cite[Theorem~3.3.12]{miophd2}).
  \end{remark}

In what follows, we are often going to omit to
  write the functor \(\funcT\) and consider a web just as a
  homomorphism of the corresponding representations.

\subsubsection{Matrix coefficients}
\label{reps:sec:matrix-coefficients-1}

Let \(\phi\) be a web from \(\bolda=(a_1,\ldots,a_\ell)\) to \(\bolda'=(a_1',\ldots,a_{\ell'}')\). Given \(\boldeta \in \{0,1\}^\ell, \boldgamma \in \{0,1\}^{\ell'}\), we can consider the matrix coefficient
\begin{equation}
  \label{reps:eq:42}
  \langle \phi (v^\bolda_\boldeta), v^{\bolda'}_\boldgamma \rangle,
\end{equation}
which is the coefficient of \(v^{\bolda'}_\boldgamma\) in \(\phi(v^\bolda_\boldeta)\) when expressed in the standard basis. We
 represent it  by a labeled web diagram
  \begin{equation}\label{reps:eq:56}
    \begin{tikzpicture}[>=angle 45,scale=1,yscale=0.9,anchorbase]

      \draw[int] (1,0) -- ++(0,0.5);
      \draw[int] (1.5,0.2) -- ++(0,0.3);
      \draw[int] (3,0) -- ++(0,0.5);

      \draw[int] (1,1.5) -- ++(0,0.5);
      \draw[int] (1.5,1.5) -- ++(0,0.5);
      \draw[int] (3,1.5) -- ++(0,0.3);

      \node at (2.25,0.25) {\(\cdots\)};
      \node at (2.25,1.75) {\(\cdots\)};

      \draw (0.5,0.5) rectangle node {\(\phi\)}(3.5,1.5);

      \mydrawdown{(1,0)};

      \mydrawup{(1.5,0.2)};
      \mydrawdown{(3,0)};
      \mydrawup{(1,2)};
      \mydrawup{(1.5,2)};
      \mydrawdown{(3,1.8)};
      \draw[decorate,decoration={brace,mirror,raise=0.1cm}] (1,0) -- node[below=0.3cm] {\(\ell\)} (3,0);
      \draw[decorate,decoration={brace,raise=0.1cm}] (1,2) -- node[above=0.3cm] {\(\ell'\)} (3,2);
    \end{tikzpicture}
  \end{equation}
where the \(i\)--th line below is labeled by \(\up\) if
\(\eta_i=0\) and by \(\down\) if \(\eta_i=1\), and the \(i\)--th line above is labeled by \(\up\) if \(\gamma_i=0\) and by \(\down\) if \(\gamma_i=1\).

Diagrams provide a convenient way to compute matrix coefficients, as we are going to explain. Fix a diagram \(\phi\)
and suppose that we want to compute the coefficient \eqref{reps:eq:42}.
We start with the picture \eqref{reps:eq:56}. Then we label every edge of
the graph with \(\up\) and \(\down\), in all possible ways. Such a
``completely labeled'' graph is evaluated according to the
local rules in Figure~\ref{reps:fig:evaluations} (the missing label possibilities are evaluated to zero,
and the total evaluation is obtained via multiplication).
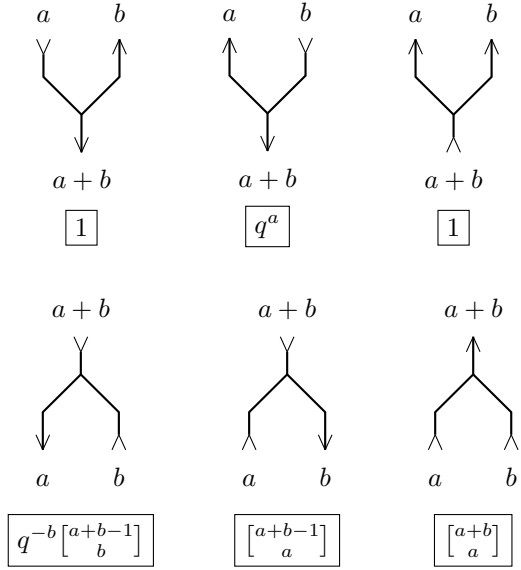
\begin{figure}
\tikzset{
    every picture/.style={
        show background rectangle,
        inner frame sep=4pt,
        background rectangle/.style={
            draw=none
        }
    }
}
  \begin{equation*}
  \begin{gathered}
    \begin{aligned}
      \begin{tikzpicture}
        \draw[int] (1.5,1) node[below=1mm] {$a+b$}--
        (1.5,1.5) -- (1,2) -- ++(0,0.3) node[above=3mm] {$a$};
        \draw[int] (1.5,1.5) -- (2,2) -- ++(0,0.5) node[above=1mm] {$b$};
        \mydrawdown{(1.5,1)}; \mydrawdown{(1,2.3)};
        \mydrawup{(2,2.5)};
        \node[ev] at (1.5,0) {$1$};
      \end{tikzpicture}
    \end{aligned}\qquad
    \begin{aligned}
      \begin{tikzpicture}
        \draw[int] (1.5,1) node[below=1mm] {$a+b$}-- (1.5,1.5) -- (1,2) -- ++(0,0.5) node[above=1mm] {$a$};
        \draw[int] (1.5,1.5) -- (2,2) -- ++(0,0.3) node[above=3mm] {$b$};
        \mydrawdown{(1.5,1)}; \mydrawup{(1,2.5)};
        \mydrawdown{(2,2.3)}; \node[ev] at (1.5,0) {$q^a$};
      \end{tikzpicture}
    \end{aligned}\qquad
    \begin{aligned}
      \begin{tikzpicture}
        \draw[int] (1.5,1.2) node[below=3mm] {$a+b$}--
        (1.5,1.5) -- (1,2) -- ++(0,0.5) node[above=1mm] {$a$};
        \draw[int] (1.5,1.5) -- (2,2) -- ++(0,0.5) node[above=1mm] {$b$};
        \mydrawup{(1.5,1.2)}; \mydrawup{(1,2.5)}; \mydrawup{(2,2.5)};
        \node[ev] at (1.5,0) {$1$};
      \end{tikzpicture}
    \end{aligned}\\
    \begin{aligned}
      \begin{tikzpicture}
        \draw[int] (1.5,2.3) node[above=3mm]
        {$a+b$}-- (1.5,2) -- (1,1.5) -- ++(0,-0.5)
        node[below=1mm] {$\vphantom{b}a$};
        \draw[int] (1.5,2) -- (2,1.5) --
        ++(0,-0.3) node[below=3mm] {$b$};
        \mydrawdown{(1.5,2.3)};
        \mydrawdown{(1,1)};
        \mydrawup{(2,1.2)};
        \node[ev] at (1.5,-0.2) {$q^{-b}\qbin{a+b-1}{b}$};
      \end{tikzpicture}
    \end{aligned}\qquad
    \begin{aligned}
      \begin{tikzpicture}
        \draw[int] (1.5,2.3) node[above=3mm]
        {$a+b$}-- (1.5,2) -- (1,1.5) -- ++(0,-0.3)
        node[below=3mm] {$\vphantom{b}a$}; \draw[int] (1.5,2) -- (2,1.5) --
        ++(0,-0.5) node[below=1mm] {$b$};
        \mydrawdown{(1.5,2.3)};
        \mydrawup{(1,1.2)};
        \mydrawdown{(2,1)};
        \node[ev] at (1.5,-0.2) {$\qbin{a+b-1}{a}$};
      \end{tikzpicture}
    \end{aligned}\qquad
    \begin{aligned}
      \begin{tikzpicture}
        \draw[int] (1.5,2.5) node[above=1mm]
        {$a+b$}-- (1.5,2) -- (1,1.5) -- ++(0,-0.3)
        node[below=3mm] {$\vphantom{b}a$}; \draw[int] (1.5,2) -- (2,1.5) --
        ++(0,-0.3) node[below=3mm] {$b$};
        \mydrawup{(1.5,2.5)};
        \mydrawup{(1,1.2)};
        \mydrawup{(2,1.2)};
        \node[ev] at (1.5,-0.2) {$\qbin{a+b}{a}$};
      \end{tikzpicture}
    \end{aligned}
  \end{gathered}
\end{equation*}
\caption{Evaluation of elementary diagrams.}\label{reps:fig:evaluations}
\end{figure}
To evaluate the initial picture, sum the evaluations over all possible ``complete labeling''.

\subsubsection{Canonical basis}
\label{reps:sec:canonical-basis}

Fix a sequence \(\bolda=(a_1,\ldots,a_\ell)\) and consider a standard basis element \(v^\bolda_\boldeta\) of \(V(\bolda)\). This standard basis element is represented by a (trivial) diagram, obtained as follows: take the identity web \(\bolda \mapto \bolda\) and label the edges from the left to the right with an \(\up\) if \(\eta_i=0\)  and a \(\down\) if \(\eta_i=1\), (as in Figure~\ref{fig:standard-basis-diagram}). We call it the \emph{standard basis diagram} corresponding to \(v^\bolda_\boldeta\).

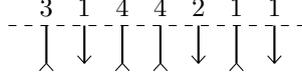
\begin{figure}
  \centering
  \begin{tikzpicture}[yscale=0.5]
    \draw[dashed] (0,1) -- ++(4,0);
    \draw[int] (0.5,0)  -- ++(0,1) node[above] {\(3\)};
    \draw[int] (1,0)  -- ++(0,1) node[above] {\(1\)};
    \draw[int] (1.5,0)  -- ++(0,1) node[above] {\(4\)};
    \draw[int] (2,0)  -- ++(0,1) node[above] {\(4\)};
    \draw[int] (2.5,0)  -- ++(0,1) node[above] {\(2\)};
    \draw[int] (3,0)  -- ++(0,1) node[above] {\(1\)};
    \draw[int] (3.5,0)  -- ++(0,1) node[above] {\(1\)};
    \mydrawup{(0.5,0)};    
    \mydrawdown{(1,0)};    
    \mydrawup{(1.5,0)};  
    \mydrawup{(2,0)};    
    \mydrawdown{(2.5,0)};
    \mydrawup{(3,0)};
    \mydrawdown{(3.5,0)};
  \end{tikzpicture}
  \caption{The standard basis diagram for \(v^{(3,1,4,4,2,1,1)}_{(0,1,0,0,1,0,1)}\).}
  \label{fig:standard-basis-diagram}
\end{figure}

Starting from this standard basis diagram, one
can obtain the corresponding canonical basis element as follows.  For every
consecutive \(\down\up\) (in this order), join the corresponding two
edges as follows:
\begin{equation}\label{reps:eq:61}
  \begin{aligned}
    \begin{tikzpicture}[yscale=0.8]
      \draw[int] (0.25,1) -- ++(0,1);
      \draw[int] (1,1) -- ++(0,1);
      \mydrawdown{(0.25,1)}
      \mydrawup{(1,1)}
    \end{tikzpicture}
  \end{aligned}\qquad\leadsto\qquad
  \begin{aligned}
    \begin{tikzpicture}[yscale=0.8]
      \draw[int] (1.5,1) --
      (1.5,1.375) -- (1.125,2) ; \draw[int] (1.5,1.375) -- (1.875,2);
      \mydrawdown{(1.5,1)}
    \end{tikzpicture}
  \end{aligned}
\end{equation}
At the same time, as in the picture, replace the labeling \(\down\up\)  at the bottom  with a  \(\down\).
Repeat this process using at each step also the
\(\down\)'s created in the previous steps, until no more \(\down\up\) is
left.  At the end, we will obtain some diagram
\(C(v^\bolda_\boldeta)\) that we call the \emph{canonical basis
  diagram} corresponding to \(v^\bolda_\boldeta\) (see Figure \ref{fig:canonical-basis-diagram} and Example~\ref{ex:1} below). Notice that the diagram \(C(v^\bolda_\boldeta)\) is a web diagram together with a labeling at the bottom (but no labeling at the top).

\begin{remark}
  Note that this canonical basis diagram is obtained joining
  recursively each edge labeled by a \(\down\) with all immediately
  following edges labeled by \(\up\)'s. If we use multiple vertices (as
  defined by \eqref{catO:eq:180}), we can construct the canonical basis
  diagram in just one step. In particular, the construction is
  independent of the order in which we consider the pairs \(\down\up\).\label{rem:10}
\end{remark}

\begin{figure}
  \centering
  \begin{tikzpicture}[yscale=1]
    \draw[dashed] (0,1) -- ++(4,0);
    \draw[int] (0.5,0)  -- ++(0,1) node[above] {\(3\)};
    \draw[int] (1.25,0.75)  -- (1,1) node[above] {\(1\)};
    \draw[int] (1.625,0.25) -- (1.25,0.5) -- (1.25,0.75)  -- (1.5,1) node[above] {\(4\)};
    \draw[int] (1.625,0) -- ++(0,0.25) -- (2,0.5)  -- (2,1) node[above] {\(4\)};
    \draw[int] (2.75,0) -- (2.75,0.75)  -- (2.5,1) node[above] {\(2\)};
    \draw[int] (2.75,0.75)  -- (3,1) node[above] {\(1\)};
    \draw[int] (3.5,0)  -- ++(0,1) node[above] {\(1\)};
    \mydrawup{(0.5,0)};    
    \mydrawdown{(1.625,0)};    
    \mydrawdown{(2.75,0)};
    \mydrawdown{(3.5,0)};
  \end{tikzpicture}
  \caption{The canonical basis diagram for \(v^{(3,1,4,4,2,1,1)}_{(0,1,0,0,1,0,1)}\).}
  \label{fig:canonical-basis-diagram}
\end{figure}
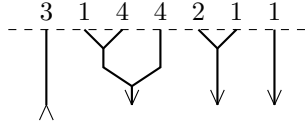

We claim that canonical basis diagrams correspond to canonical basis elements via \(\funcT\). Indeed,
the diagram \(C(v^\bolda_\boldeta)\) has  an underlying web \(W(v^\bolda_\boldeta)\) (by removing the labeling at the bottom) that
represents some embedding \(V(\bolda')\mapto V(\bolda)\), where \(\bolda'\) is some
composition that is refined by \(\bolda\). This web carries on the bottom
the labels of a standard basis element \(v^{\bolda'}_{\boldeta'}\) of \(V(\bolda')\). Then we have:

\begin{prop}
  \label{prop:14}
  The diagram \(C(v^\bolda_\boldeta)\) gives the canonical basis element \(v^{\canon \bolda}_\boldeta\) of \(V(\bolda)\), in the sense that
  \begin{equation}
    \label{eq:151}
    \funcT(W(v^\bolda_\boldeta))(v^{\bolda'}_{\boldeta'}) = v^{\canon \bolda}_{\boldeta}.
  \end{equation}
\end{prop}

\begin{proof}
  By the construction, the labeling at the
  bottom of \(C(v^{\bolda}_\boldeta)\) is a sequence of \(\up\)'s
  followed by a sequence of \(\down\)'s.
  Hence \(v^{\bolda'}_{\boldeta'}\)
  is the minimal element (see \eqref{eq:168}), and therefore
  it is also a canonical basis element.  Now, the image
  \(\funcT(W(v^\bolda_\boldeta)) (v^{\bolda'}_{\boldeta'})\)
  under \(\funcT(W(v^\bolda_\boldeta))\)
  of \(v^{\bolda'}_{\boldeta'}\)
  is a bar-invariant element of \(V(\bolda)\)
  (since \(\funcT(\phi)\)
  sends bar-invariant elements to bar-invariant elements for
  all webs \(\phi\)).
  Examining the evaluation rules
  (Figure~\ref{reps:fig:evaluations}), one sees immediately
  that, when expressed in terms of the standard basis of
  \(V(\bolda)\),
  the coefficients of
  \(\funcT(W(v^\bolda_\boldeta)) (v^{\bolda'}_{\boldeta'})\)
  are all in \(q\Z[q]\)
  except for the coefficient of \(v^\bolda_\boldeta\),
  which is \(1\), whence the claim.
\end{proof}

\begin{example}
  \label{ex:1}
  Let \(\bolda=(3,1,4,4,2,1,1)\) and consider the element \(v^{\bolda}_{(0,1,0,0,1,0,1)} \in V(\bolda)\). The corresponding standard  and canonical basis diagrams are pictured in Figures \ref{fig:standard-basis-diagram} and \ref{fig:canonical-basis-diagram}. In particular, evaluating the canonical basis diagram according to the rules in Figure~\ref{reps:fig:evaluations}, we get the corresponding canonical basis element
  \begin{multline}
v^{\canon \bolda}_{(0,1,0,0,1,0,1)}=
      v^{\bolda}_{(0,1,0,0,1,0,1)}
      +  q v^{\bolda}_{(0,0,1,0,1,0,1)}
      + q^5 v^{\bolda}_{(0,0,0,1,1,0,1)}\\
      + q^2 v^{\bolda}_{(0,1,0,0,0,1,1)}
      + q^3 v^{\bolda}_{(0,0,1,0,0,1,1)}
      + q^7 v^{\bolda}_{(0,0,0,1,0,1,1)}.    \label{eq:140}
  \end{multline}
  For example, the coefficients  \(1\) of \(v^{\bolda}_{(0,1,0,0,1,0,1)}\) and \(q^3\) of \( v^{\bolda}_{(0,0,1,0,0,1,1)}\) in \eqref{eq:140} are obtained by evaluating the diagrams
  \begin{equation*}
\begingroup
\tikzset{every node/.style={font=\small}}
   \begin{tikzpicture}[yscale=1,anchorbase]
    \draw[dashed] (0,1.5) -- ++(4,0);
    \draw[int] (0.5,0)  -- ++(0,1);
    \draw[int] (1.25,0.75)  -- (1,1);
    \draw[int] (1.625,0.25) -- (1.25,0.5) -- (1.25,0.75)  -- (1.5,1);
    \draw[int] (1.625,0) -- ++(0,0.25) -- (2,0.5)  -- (2,1);
    \draw[int] (2.75,0) -- (2.75,0.75)  -- (2.5,1);
    \draw[int] (2.75,0.75)  -- (3,1);
    \draw[int] (3.5,0)  -- ++(0,1);
    \begin{scope}[xshift=-0.15cm,yshift=-0.1cm]
    \path (0.5,0)  -- ++(0,1) node[above] {\(3\)};
    \path (1.25,0.75)  -- (1,1) node[above] {\(1\)};
    \path (1.625,0.25) -- (1.25,0.5) -- (1.25,0.75)  -- (1.5,1) node[above] {\(4\)};
    \path (1.625,0) -- ++(0,0.25) -- (2,0.5)  -- (2,1) node[above] {\(4\)};
    \path (2.75,0) -- (2.75,0.75)  -- (2.5,1) node[above] {\(2\)};
    \path (2.75,0.75)  -- (3,1) node[above] {\(1\)};
    \path (3.5,0)  -- ++(0,1) node[above] {\(1\)};
    \end{scope}
    \draw[int] (0.5,1) -- ++(0,0.5);
    \draw[int] (1,1) -- ++(0,0.5);
    \draw[int] (1.5,1) -- ++(0,0.5);
    \draw[int] (2,1) -- ++(0,0.5);
    \draw[int] (2.5,1) -- ++(0,0.5);
    \draw[int] (3,1) -- ++(0,0.5);
    \draw[int] (3.5,1) -- ++(0,0.5);
    \mydrawup{(0.5,0)};    
    \mydrawdown{(1.625,0)};    
    \mydrawdown{(2.75,0)};
    \mydrawdown{(3.5,0)};
    \mydrawup{(0.5,1.5)};
    \mydrawdown{(1,1.5)};
    \mydrawup{(1.5,1.5)};
    \mydrawup{(2,1.5)};
    \mydrawdown{(2.5,1.5)};
    \mydrawup{(3,1.5)};
    \mydrawdown{(3.5,1.5)};
  \end{tikzpicture}\quad \text{and} \quad
   \begin{tikzpicture}[yscale=1,anchorbase]
    \draw[dashed] (0,1.5) -- ++(4,0);
    \draw[int] (0.5,0)  -- ++(0,1);
    \draw[int] (1.25,0.75)  -- (1,1);
    \draw[int] (1.625,0.25) -- (1.25,0.5) -- (1.25,0.75)  -- (1.5,1);
    \draw[int] (1.625,0) -- ++(0,0.25) -- (2,0.5)  -- (2,1);
    \draw[int] (2.75,0) -- (2.75,0.75)  -- (2.5,1);
    \draw[int] (2.75,0.75)  -- (3,1);
    \draw[int] (3.5,0)  -- ++(0,1);
    \begin{scope}[xshift=-0.15cm,yshift=-0.1cm]
    \path (0.5,0)  -- ++(0,1) node[above] {\(3\)};
    \path (1.25,0.75)  -- (1,1) node[above] {\(1\)};
    \path (1.625,0.25) -- (1.25,0.5) -- (1.25,0.75)  -- (1.5,1) node[above] {\(4\)};
    \path (1.625,0) -- ++(0,0.25) -- (2,0.5)  -- (2,1) node[above] {\(4\)};
    \path (2.75,0) -- (2.75,0.75)  -- (2.5,1) node[above] {\(2\)};
    \path (2.75,0.75)  -- (3,1) node[above] {\(1\)};
    \path (3.5,0)  -- ++(0,1) node[above] {\(1\)};
    \end{scope}
    \draw[int] (0.5,1) -- ++(0,0.5);
    \draw[int] (1,1) -- ++(0,0.5);
    \draw[int] (1.5,1) -- ++(0,0.5);
    \draw[int] (2,1) -- ++(0,0.5);
    \draw[int] (2.5,1) -- ++(0,0.5);
    \draw[int] (3,1) -- ++(0,0.5);
    \draw[int] (3.5,1) -- ++(0,0.5);
    \mydrawup{(0.5,0)};    
    \mydrawdown{(1.625,0)};    
    \mydrawdown{(2.75,0)};
    \mydrawdown{(3.5,0)};
    \mydrawup{(0.5,1.5)};
    \mydrawup{(1,1.5)};
    \mydrawdown{(1.5,1.5)};
    \mydrawup{(2,1.5)};
    \mydrawup{(2.5,1.5)};
    \mydrawdown{(3,1.5)};
    \mydrawdown{(3.5,1.5)};
  \end{tikzpicture}
\endgroup
  \end{equation*}
  respectively.
\end{example}

\begin{remark}
  \label{rem:1}
  Notice that for the regular composition \(\compn\) of \(n\) and for some standard basis element \(v^\compn_\boldeta\), the underlying web diagram \(W(v^\compn_\boldeta)\) of the corresponding canonical basis diagram \(C(v^\compn_\boldeta)\) is a diagram of the type \(\websplitmultiple_{a_1} \otimes \dotsb \otimes \websplitmultiple_{a_\ell}\), where \(\bolda\) is some composition of \(\compn\) depending on \(\boldeta\), and it is labeled at the bottom by a sequence of \(\up\)'s followed by a sequence of \(\down\)'s.
\end{remark}

\subsubsection{Action of \texorpdfstring{$E$}{E} and \texorpdfstring{$F$}{F}}
\label{reps:sec:action-e-f}

Using our diagram calculus we can easily compute the action of \(F\) on canonical basis elements (in an analogous way as \cite{MR1446615} for \(\mathfrak{sl}_2\)).

\begin{prop}
  \label{prop:15}
  Fix some representation \(V(\bolda)\) and consider a canonical basis element \(v^{\canon\bolda}_\boldeta\). We have
  \begin{equation}
    \label{eq:70}
    F(v^{\canon\bolda}_\boldeta) =
    \begin{cases}
      v^{a_1}_{1} \canon v^{a_1}_{\eta_2} \canon \cdots \canon v^{a_\ell}_{\eta_\ell} & \text{if\/ } \eta_1=0,\\
      0 & \text{otherwise}.
    \end{cases}
  \end{equation}
\end{prop}

\begin{proof}
  Suppose that \(\eta_i=0\) for \(i=1,\ldots,h\), while \(\eta_{h+1}=1\) (possibly \(h=0\)). The canonical basis diagram \(C(v^{\bolda}_\boldeta)\) is
  
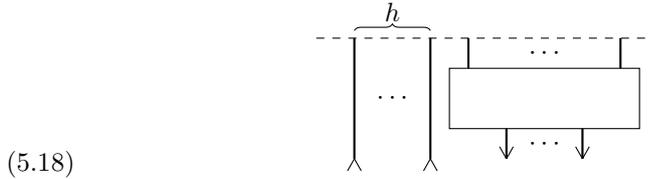
\begin{equation}
    \label{eq:71}
    \begin{tikzpicture}[yscale=0.8]
      \draw[dashed] (0,0) -- ++(4.5,0);
      \draw[int] (0.5,-2) -- ++(0,2);
      \draw[int] (1.5,-2) -- ++(0,2);
      \node at (1,-1) {\(\cdots\)};
      \draw[decorate,decoration={brace,raise=0.1cm}] (0.5,0) -- node[above=0.1cm] {\(h\)} (1.5,0);
      \draw[int] (2,-0.5) -- ++(0,0.5);
      \draw[int] (4,-0.5) -- ++(0,0.5);
      \node at (3,-0.25) {\(\cdots\)};
      \draw[int] (2.5,-2) -- ++(0,0.5);
      \draw[int] (3.5,-2) -- ++(0,0.5);
      \node at (3,-1.75) {\(\cdots\)};
      \draw (1.75,-1.5) rectangle (4.25,-0.5);
      \mydrawup{(0.5,-2)};
      \mydrawup{(1.5,-2)};
      \mydrawdown{(2.5,-2)};
      \mydrawdown{(3.5,-2)};
    \end{tikzpicture}
  \end{equation}
  where there are some vertices in the box. We can also represent it as
  
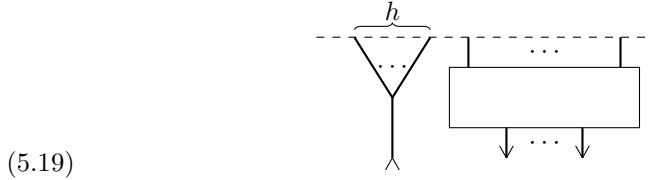
\begin{equation}
    \label{eq:72}
    \begin{tikzpicture}[yscale=0.8]
      \draw[dashed] (0,0) -- ++(4.5,0);
      \draw[int] (1,-2) -- ++(0,1) -- (0.5,0);
      \draw[int] (1,-1) -- (1.5,0);
      \node at (1,-0.5) {\(\cdots\)};
      \draw[decorate,decoration={brace,raise=0.1cm}] (0.5,0) -- node[above=0.1cm] {\(h\)} (1.5,0);
      \draw[int] (2,-0.5) -- ++(0,0.5);
      \draw[int] (4,-0.5) -- ++(0,0.5);
      \node at (3,-0.25) {\(\cdots\)};
      \draw[int] (2.5,-2) -- ++(0,0.5);
      \draw[int] (3.5,-2) -- ++(0,0.5);
      \node at (3,-1.75) {\(\cdots\)};
      \draw (1.75,-1.5) rectangle (4.25,-0.5);
      \mydrawup{(1,-2)};
      \mydrawdown{(2.5,-2)};
      \mydrawdown{(3.5,-2)};
    \end{tikzpicture}
  \end{equation}
  because this is the same element according to our diagrammatic calculus. On the bottom we read the labels of \(v^{\bolda'}_{\boldgamma}\) for some composition \(\bolda'\) refining \(\bolda\), where \(\boldgamma=(0,1,\ldots,1)\). We can easily compute
  \begin{equation}
    \label{eq:73}
    F(v^{a'_1}_0 \otimes v^{a'_2}_1 \otimes \cdots \otimes v^{a'_{r'}}_1) = v^{a'_1}_1 \otimes v^{a'_2}_1 \otimes \cdots \otimes v^{a'_{r'}}_1.
  \end{equation}
  Hence
  \begin{equation}\label{eq:149}
    F \left(
      \begin{aligned}
        \begin{tikzpicture}[yscale=0.8]
          \draw[dashed] (0,0) -- ++(4.5,0);
          \draw[int] (1,-2) -- ++(0,1) -- (0.5,0);
          \draw[int] (1,-1) -- (1.5,0);
          \node at (1,-0.5) {\(\cdots\)};
          \draw[decorate,decoration={brace,raise=0.1cm}] (0.5,0) -- node[above=0.1cm] {\(h\)} (1.5,0);
          \draw[int] (2,-0.5) -- ++(0,0.5);
          \draw[int] (4,-0.5) -- ++(0,0.5);
          \node at (3,-0.25) {\(\cdots\)};
          \draw[int] (2.5,-2) -- ++(0,0.5);
          \draw[int] (3.5,-2) -- ++(0,0.5);
          \node at (3,-1.75) {\(\cdots\)};
          \draw (1.75,-1.5) rectangle (4.25,-0.5);
          \mydrawup{(1,-2)};
          \mydrawdown{(2.5,-2)};
          \mydrawdown{(3.5,-2)};
        \end{tikzpicture}
      \end{aligned}
\right) =
      \begin{aligned}
        \begin{tikzpicture}[yscale=0.8]
          \draw[dashed] (0,0) -- ++(4.5,0);
          \draw[int] (1,-2) -- ++(0,1) -- (0.5,0);
          \draw[int] (1,-1) -- (1.5,0);
          \node at (1,-0.5) {\(\cdots\)};
          \draw[decorate,decoration={brace,raise=0.1cm}] (0.5,0) -- node[above=0.1cm] {\(h\)} (1.5,0);
          \draw[int] (2,-0.5) -- ++(0,0.5);
          \draw[int] (4,-0.5) -- ++(0,0.5);
          \node at (3,-0.25) {\(\cdots\)};
          \draw[int] (2.5,-2) -- ++(0,0.5);
          \draw[int] (3.5,-2) -- ++(0,0.5);
          \node at (3,-1.75) {\(\cdots\)};
          \draw (1.75,-1.5) rectangle (4.25,-0.5);
          \mydrawdown{(1,-2)};
          \mydrawdown{(2.5,-2)};
          \mydrawdown{(3.5,-2)};
        \end{tikzpicture}
      \end{aligned}
  \end{equation}
  that is exactly our assertion.
\end{proof}

By \eqref{eq:75} it follows that \(E\) sends the dual canonical basis to the dual canonical basis (up to a multiple):

\begin{prop}
  \label{prop:21}
  Fix some representation \(V(\bolda)\) and consider a dual canonical basis element \(v^{\dualcanon\bolda}_\boldeta\). We have
  \begin{equation}
    \label{eq:166}
    E(v^{\dualcanon\bolda}_\boldeta) =
    \begin{cases}
      \displaystyle \frac{[\beta^\boldeta_1+\cdots+\beta^\boldeta_\ell]_0}{q^{a_1+\cdots+a_\ell-1}}v^{a_1}_{0} \dualcanon v^{a_1}_{\eta_2} \dualcanon \cdots \dualcanon v^{a_\ell}_{\eta_\ell} & \text{if\/ } \eta_1=1,\\
      0 & \text{otherwise}.
    \end{cases}
  \end{equation}
\end{prop}

\section{Subquotient categories of \texorpdfstring{$\catO$}{O}}
\label{sec:category-cato}

We now move to  the technical core of the paper, in which we construct the categorification using subquotient categories of \(\catO\). We have tried to separate the categorification itself (in Section \ref{sec:categorification}) from the Lie theoretical tools, which we collect now. The reader who is interested in the categorification but not in all the Lie theoretical details may want to skim quickly the following notions and pass then to Section \ref{sec:categorification}.

We  start with a quick reminder about Serre quotient categories (\textsection\ref{sec:serre-quotients}). We will then give two equivalent definitions of the categories \(\catO^{\frakp,\frakq\pres}_\lambda\) (\textsection\ref{sec:gener-parab-subc} and \textsection\ref{sec:categ-frakp-pres}) and describe their properly stratified structure. Finally, in \textsection\ref{sec:funct-betw-categ} we introduce the functors between these categories that we will use for categorifying the action of \(\Uqgl\) and of the intertwining operators in the next section.

\subsection{Serre quotients and projectively presented modules}
\label{sec:serre-quotients}

Let \(\calA\) be some abelian category which is equivalent to
the category of finite-dimensional modules over some
finite-dimensional \(\C\)--algebra. Let \(\{L(\lambda)
\suchthat \lambda \in \Lambda\}\) be the simple objects of
\(\calA\) up to isomorphism. For all \(\lambda \in \Lambda\) let
\(P(\lambda)\) be the projective cover of \(L(\lambda)\). Let
\(P= \bigoplus_{\lambda \in \Lambda} P(\lambda)\) be a minimal
projective generator and let \(R= \End_{\calA}(P)\). Then we
have an equivalence of categories
\(\calA \cong \rmod{R}\)
via the functor \(\Hom_\calA(P,\blank)\).

\subsubsection{Serre subcategories}
\label{sec:serre-subcategories}

A non-empty full subcategory \(\calS \subset \calA\) is called
a \emph{Serre subcategory} if it is closed under subobjects,
quotients and extensions. For a subset \(\Gamma \subseteq
\Lambda\) define \(\calS_\Gamma\) to be the full subcategory of
\(\calA\) consisting of the modules with all composition
factors of type \(L(\gamma)\) for \(\gamma \in \Gamma\). Then
\(\calS_\Gamma\) is obviously a Serre subcategory of
\(\calA\). Let \(I_\Gamma\) be the two-sided ideal of
\(R=\End_\calA(P)\) generated by all endomorphisms which
factor through some \(P(\eta)\) for \(\eta \notin
\Gamma\). Notice that if we let \(e_\lambda\) for \(\lambda \in
\Lambda\) be the idempotent projecting onto
\(\End_\calA(P(\lambda)) \subset R\) and
\(e^\perp_\Gamma=\sum_{\eta \notin \Gamma} e_\eta\) then
\(I_\Gamma = R e_\Gamma^\perp R\). Then
\begin{equation}
  \label{eq:13}
  \calS_\Gamma \cong \rmod{R/I_\Gamma}.
\end{equation}

A complete set of pairwise non-isomorphic simple objects in \(\calS_\Gamma\) is given by the \(L(\gamma)\)'s for \(\gamma \in \Gamma\) and each of them has a projective cover \(P^{\calS_\Gamma}(\gamma)\) in \(\calS_\Gamma\), which is the biggest quotient of \(P(\gamma)\) which lies in \(\calS_\Gamma\).

\subsubsection{Serre quotients}
\label{sec:serre-quotients-1}

Given a Serre subcategory \(\calS \subset \calA\) one defines the \emph{quotient category} \(\calA/\calS\) to be the category with the same objects of \(\calA\) and with morphisms
 \( \Hom_{\calA/\calS} ( M,N) = \varinjlim \Hom_\calA (M',N/N')\),
where the direct limit is taken over all pairs \(M' \subseteq M\), \(N' \subseteq N\) such that \(M/M' \in \calS\) and \(N' \in \calS\). The Serre quotient comes with an exact quotient functor \(Q\colon \calA \to \calA/\calS\) (see \cite{MR0232821}).

Also in this case, we have an equivalence of categories
  \begin{equation}
    \label{catO:eq:201}
    \calA/\calS_\Gamma \cong \rmod{\End_\calA(P^\perp_\Gamma)},
  \end{equation}
where \(P^\perp_\Gamma=\bigoplus_{\eta \in \Lambda - \Gamma} P(\eta)\) (see for example \cite[Proposition~33]{MR2774639}).   The quotient functor is \(Q = \Hom_\calA (P^\perp_\Gamma, \blank)\). In particular, we can deduce from \eqref{catO:eq:201} the abelian structure of \(\calA/\calS_\Gamma\). Notice that \(\End_\calA(P^\perp_\Gamma) = e^\perp_\Gamma R e^\perp_\Gamma\) where \(e^\perp_\Gamma = \sum_{\gamma \in \Lambda - \Gamma} e_\gamma\). 

A complete set of pairwise non-isomorphic simple objects in \(\calA/\calS_\Gamma\) is given by the \(L(\eta)\)'s for \(\eta \in \Lambda - \Gamma\), with projective covers \(P(\eta)\).
 
\subsubsection{Presentable modules}
\label{sec:presentable-modules}

Let \(\calC\) be an additive subcategory of  \(\calA\). We define the category of \emph{\(\calC\)--presentable objects} to be the full subcategory of \(\calA\) consisting of all objects \(M \in \calA\) having a presentation
 \( Q_1 \longrightarrow Q_2 \twoheadlongrightarrow M\)
with \(Q_1,Q_2 \in \calC\). Given a projective object \(P \in \calA\) we let \(\Add(P)\) be the additive full subcategory of \(\calA\) consisting of all objects which admit a direct sum decomposition with summands being direct summands of \(P\), and we consider the category \(\overline{\Add(P)}\) of \emph{\(P\)--presentable} or \emph{\(\Add(P)\)--presentable objects}. By \cite[Proposition~5.3]{MR0349747}, the category \(\overline{\Add(P)}\) is equivalent to \(\rmod{\End_\calA(P)}\). In particular, if \(P= P^\perp_\Gamma\) as in \eqref{catO:eq:201}, then we have
\begin{equation}
\overline{\Add(P^\perp_\Gamma)} \cong \rmod{\End_\calA(P^\perp_\Gamma)} \cong \calA/\calS_\Gamma.\label{catO:eq:203}
\end{equation}
Notice that this gives an equivalence between  \(\calA/\calS_\Gamma\) and a full subcategory of \(\calA\).

\begin{remark}
  \label{rem:11}
  If \(M,N \in \calA/\calS_\Gamma\) then by definition \(M\) and \(N\) are also objects of \(\calA\) and we can consider both the homomorphism spaces \(\Hom_{\calA}(M,N)\) and \(\Hom_{\calA/\calS_\Gamma}(M,N)\): they are in general different. But notice that if \(M\) and \(N\), as objects of \(\calA\), are \(P^\perp_\Gamma\)--presentable, then the two homomorphism spaces coincide by \eqref{catO:eq:203}. In the following, we will most of the time only deal with objects of Serre quotient categories which are also presentable.
\end{remark}

\subsection{Subquotient categories of \texorpdfstring{$\catO$}{O}}
\label{sec:gener-parab-subc}

Let us fix a positive integer \(n\). Let \(\gl_n\) be the general
Lie algebra of \(n \times n\) matrices with the standard Cartan
decomposition \(\gl_n=\frakn^- \oplus \frakh \oplus \frakn^+\) and let \(\frakb = \frakh \oplus \frakn^+\) be the standard Borel subalgebra. Consider
the integral BGG category \(\catO=\catO(\gl_n)=\catO(\gl_n,\frakb)\): this is the full
subcategory of finitely generated \(U(\gl_n)\)--modules that are weight
modules for the action of \(\frakh\) with integral weights and that are
locally \(\frakn^+\)--finite.
We recall some standard facts on the category \(\catO\); for more details we refer to
\cite{MR2428237}.

The category \(\catO\) is a highest weight category
(\cite{MR961165}). For a weight \(\lambda\) of \(\gl_n\) we let
\(M(\lambda)\) be the Verma module with highest weight \(\lambda\). We let
\(L(\lambda)\) be the unique simple quotient of \(M(\lambda)\) and
\(P(\lambda)\) be its projective cover. The modules \(L(\lambda)\) for
\(\lambda\) running over the integral weights of \(\gl_n\) give a full set
of pairwise non-isomorphic simple objects in \(\catO\).

We consider the dot action of the Weyl group \(\bbS_n\) on \(\frakh^*\), given by \(w \cdot \lambda= w(\lambda+\rho) -\rho\). Two simple objects \(L(\lambda)\), \(L(\mu)\) are in the same block of \(\catO\) if and only if \(\lambda\) and \(\mu\) are in the same \(\bbS_n\)--orbit under the dot action. For an integral dominant weight \(\lambda\) we let \(\catO_\lambda\) be the block of \(\catO\) containing \(L(\lambda)\). We have then a block decomposition \(\catO= \bigoplus_\lambda
\catO_\lambda\).

If \(P_\lambda\) is a minimal projective generator of \(\catO_\lambda\), let \(A_\lambda=\End_{\catO_\lambda}(P_\lambda)\). Then \(\catO_\lambda \cong \rmod{A_\lambda}\). The algebra \(A_\lambda\) can be given a natural positive grading (see \cite{MR1029692}, \cite{MR1322847}) which turns it into a Koszul algebra. Then one defines the \emph{graded version} of \(\catO_\lambda\) to be \(\catOZ_\lambda = \rgmod{A_\lambda}\).

Given a standard parabolic subalgebra \(\frakp \subset \gl_n\) with
Levi factor \(\frakl\), let \(\catO^{\frakp}\) be the full subcategory
of \(\catO\) consisting of modules that, viewed as \(U(\frakl)\)--modules,
are direct sums of finite-dimensional simple \(\frakl\)--modules. Let
\(W_\frakp \subset \bbS_n\) be the standard parabolic subgroup corresponding
to \(\frakp\), and let \(W^\frakp\) be the set of shortest coset representatives for \({W_\frakp}\backslash{ \bbS_n}\). Then \(\catO^{\frakp}\) is also the full Serre
subcategory of \(\catO\) generated by the simple objects \(L(x \cdot
\lambda )\) for \(\lambda\) dominant and \(x \in W^{\frakp}\) such that \(x \bbS_\lambda \subseteq W^\frakp\). We denote by \(P^{\frakp}(x \cdot \lambda)\)
the projective cover of \(L(x \cdot \lambda)\) in \(\catO^\frakp\) and
by \(M^{\frakp}(x \cdot \lambda)\) the corresponding parabolic Verma
module. The block decomposition of \(\catO\) induces a block
decomposition \(\catO^{\frakp}= \bigoplus_\lambda \catO^{\frakp}_\lambda\).

Let \(e_\frakp^\perp \in A_\lambda
= \End(\scrP(\lambda))\) be the idempotent projecting onto
the direct sum of the projective modules \(P(x \cdot
\lambda)\) for \( x \in \bbS_n\) such that \(x\bbS_\lambda
\not\subseteq W^\frakp\). Then 
\(\End(\scrP^{\frakp}(\lambda)) = A_\lambda / A_\lambda
e_\frakp^\perp A_\lambda\) and
\begin{equation}
\catO^\frakp_\lambda \cong
\rmod{\big(A_\lambda/A_\lambda e_\frakp^\perp A_\lambda\big)}.\label{eq:33}
\end{equation}
Since the idempotent \(e_\frakp^\perp\) is homogeneous, the latter quotient algebra inherits a graded structure. In
particular, there is a graded version \(\catOZ^\frakp_\lambda =
\rgmod{\big(A_\lambda/ A_\lambda e_\frakp^\perp A_\lambda\big)}\).

\subsubsection{Generalized parabolic subcategories}
\label{sec:gener-parab-subc-1}

Let now \(\frakp,\frakq\) be two orthogonal standard parabolic subalgebras of \(\gl_n\) (by orthogonal we mean that the corresponding subsets \(\Pi_\frakp,\Pi_\frakq\) of the simple roots \(\Pi\) of \(\gl_n\) are orthogonal; this is equivalent to imposing that \(\frakp+\frakq\) is also a parabolic subalgebra of \(\gl_n\) and \(\frakp \cap \frakq = \frakb\)).
Let \(W_\frakp, W_\frakq\) be the corresponding parabolic
subgroups of the Weyl group \(\bbS_n\). Note that, since \(\frakp\) and \(\frakq\) are orthogonal, \(W_\frakp \times W_\frakq\) is
also a subgroup of \(\bbS_n\). Consider the general Lie algebras
\(\gl_\frakp,\gl_\frakq \subset \gl_n\) with Weyl groups \(W_\frakp\) and \(W_\frakq\) respectively, so that \(\frakp = \gl_\frakp + \frakb\) and \(\frakq = \gl_\frakq + \frakb\).

Following \cite{MR2450613}, we let
\(\mathcal P_\frakq = \Add(P(w_\frakq \cdot 0))\) be the additive
  subcategory of \(\catO(\gl_\frakq)\) generated by the anti-dominant
  indecomposable projective module \(P(w_\frakq \cdot 0)\), where
  \(w_\frakq \in W_\frakq\) is the longest element.
Let also \(\overline{\mathcal P_\frakq}\) be the category of \(\calP_\frakq\)--presentable modules

Let \(\fraka=\fraka_{\frakp+\frakq} = (\gl_\frakp \oplus \gl_\frakq) +
\frakh\) and define \(\frakn_{\frakp+\frakq}\) by \(\frakp + \frakq =
\fraka \oplus \frakn_{\frakp+\frakq}\). Given a \(\gl_\frakq\)--module \(M\), we denote by \(\frakE^\fraka M \) the \(\fraka\)--module obtained by extending the action by \(0\).
Let \(\calP^\frakp_\frakq\) be the additive closure
of the full subcategory of \(\fraka\)--modules which have the form \(E
\otimes \frakE^\fraka P\), where \(E\) is a simple finite-dimensional \(\mathfrak
a\)--module and \(P \in \overline {\mathcal P_\frakq}\) is a projective object. Finally, let \(\calA_\frakq^\frakp = \overline {\mathcal P^\frakp_\frakq}\) be the category of \(\mathcal
P^\frakp_\frakq\)--presentable \(\mathfrak a\)--modules.
In other words, \(\calP^\frakp_\frakq\) is the category
\begin{equation}
\langle E \otimes \frakE^\fraka P(w_\frakq \cdot 0)
\suchthat E \text{ is a simple finite-dimensional \(\mathfrak a\)--module}\rangle
\label{catO:eq:2}
\end{equation}
and \(\mathcal A_\frakq^\frakp = \overline{ \mathcal P^\frakp_\frakq}\).

\begin{definition}
  We define \(\catO\{\frakp+\frakq,\calA_{\frakq}^{\frakp}\}\) to be the full
  subcategory of \(\gl_n\)--modules which are:
  \begin{enumerate}[(GP1),label=(GP\arabic*)]
  \item\label{item:1} finitely generated;
  \item\label{item:2} locally \(\mathfrak n_{\frakp+\frakq}\)--finite;
  \item\label{item:3} direct sum of objects of \(\mathcal A^\frakp_\frakq\) as
    \(\mathfrak a\)--modules.
  \end{enumerate}\label{catO:def:2}
\end{definition}

\begin{remark}
  The categories
  \(\catO\{\frakp+\frakq,\calA_{\frakq}^{\frakp}\}\) fall into
  a more general family of categories that were first
  introduced in \cite{MR1921761} (called \emph{generalized
    parabolic subcategories} of \(\catO\)) and then
  generalized in \cite{MR2057398}. Our definition follows
  \cite{MR2450613}, and in particular is a special case of
  \cite[Definition 32]{MR2450613}. However, in
  \cite{MR2450613} only the trivial block is studied, while
  we are interested also in singular blocks. Notice that
  the category \(\calA^\frakp_\frakq\) is admissible (in the
  sense of \cite[\textsection 6.3]{MR2450613}) by
  \cite[Lemma 33]{MR2450613}.\label{rem:13}
\end{remark}

\begin{lemma}
  \label{lem:24}
  The category \(\catO\{\frakp+\frakq,\calA_{\frakq}^{\frakp}\}\) is a subcategory of \(\catO^\frakp\).
\end{lemma}

\begin{proof}
  Conditions \ref{item:2} and \ref{item:3} together imply that
  modules of \(\catO\{\frakp+\frakq,\calA_{\frakq}^{\frakp}\}\) are
  locally \(\mathfrak n^+\)--finite; condition \ref{item:3} also
  implies that modules of
  \(\catO\{\frakp+\frakq,\calA_{\frakq}^{\frakp}\}\) are weight modules
  for \(\frakh\); hence \(\catO\{\frakp+\frakq,\calA_{\frakq}^{\frakp}\}\)
  is a subcategory of \(\catO\). By condition \ref{item:3}, moreover,
  objects of \(\catO\{\frakp+\frakq,\calA_{\frakq}^{\frakp}\}\) are
  direct sums of finite dimensional simple \(\gl_\frakp\)--modules. Hence
  \(\catO\{\frakp+\frakq,\calA_{\frakq}^{\frakp}\}\) is a subcategory of
  \(\catO^{\frakp}\).
\end{proof}

It follows in particular that the block decomposition \(\catO^\frakp =
\bigoplus_\lambda \catO^\frakp_\lambda\) induces a direct sum
decomposition
\(\catO\{\frakp+\frakq,\calA_{\frakq}^{\frakp}\} = \bigoplus_\lambda \catO\{\frakp+\frakq,\calA_{\frakq}^{\frakp}\}_\lambda\).

\begin{lemma}
  \label{lem:14}
  We have the following inclusions of full subcategories:
  \begin{enumerate}[(i)]
  \item\label{item:8} if \(\frakp' \subset \frakp\) then \(\catO\{\frakp+\frakq,\calA_{\frakq}^{\frakp}\} \subset \catO\{\frakp'+\frakq,\calA_{\frakq}^{\frakp'}\}\);
  \item\label{item:9} if \(\frakq' \subset \frakq\) then \(\catO\{\frakp+\frakq,\calA_{\frakq}^{\frakp}\} \subset \catO\{\frakp+\frakq',\calA_{\frakq'}^{\frakp}\}\).
  \end{enumerate}
\end{lemma}

We warn the reader, however, that the second inclusion will not be an exact inclusion of abelian categories (once we will have defined the abelian structure on the categories \(\catO\{\frakp+\frakq,\calA_{\frakq}^{\frakp}\}\), see \textsection\ref{sec:categ-frakp-pres}).

\begin{proof}
  Let \(M \in \catO\{\frakp+\frakq,\calA_{\frakq}^{\frakp}\}\). By
  definition, \(M\) is finitely generated and locally
  \(\frakn^+\)--finite.
  Write \(M = \bigoplus_\alpha M_\alpha\) as an
  \(\fraka_{\frakp+\frakq}\)--module, with \(M_\alpha \in
  \calA^{\frakp}_{\frakq}\). Let \(P_\alpha \mapto Q_\alpha \surto
  M_\alpha\) be a \(\calP^\frakp_\frakq\)--presentation of \(M_\alpha\). Considering this as a sequence of \(\fraka_{\frakp'+\frakq}\)--modules (resp.\ \(a_{\frakp+\frakq'}\)--modules), we see that it
  is enough to show that
  \begin{enumerate}[(i)]
  \item\label{item:10} every object of \(\calP^\frakp_\frakq\) decomposes, as an
    \(\fraka_{\frakp'+\frakq}\)--module, into a direct sum of
    objects of \(\calP^{\frakp'}_{\frakq}\);
  \item\label{item:11} every object of \(\calP^{\frakp}_{\frakq}\)
    decomposes, as an \(\fraka_{\frakp+\frakq'}\)--module, into a direct sum
    of objects of \(\calP^\frakp_{\frakq'}\).
  \end{enumerate}
  Since \ref{item:10} is straightforward (every object of \(\calP^{\frakp}_{\frakq}\) is, as an \(\fraka_{\frakp'+\frakq}\)--module, an object of \(\calP^{\frakp'}_{\frakq}\)), let us verify \ref{item:11}. For this it is enough to check that, for every dominant integral weight \(\lambda\) of \(\gl_\frakq\), the anti-dominant projective module \(P(w_\frakq \cdot \lambda) \in \catO(\gl_\frakq)\) decomposes, as a \(\gl_{\frakq'}\)--module, as direct sum of objects of type \(E \otimes P(w_{\frakq'} \cdot \mu)\) for some weight \(\mu\) of \(\gl_{\frakq'}\) and some finite dimensional \(\gl_{\frakq'}\)--module \(E\). This follows because \(\catO(\gl_\frakq) \ni P(w_{\frakq'} \cdot \lambda) = U(\gl_\frakq) \otimes_{U(\frakq'\cap \gl_\frakq)} P(w_{\frakq'} \cdot \lambda|_{\gl_{\frakq'}})\), and \(P(w_\frakq \cdot \lambda)\) can be obtained from \(P(w_{\frakq'} \cdot \lambda)\) in \(\catO(\gl_\frakq)\) by tensoring with finite-dimensional modules.
\end{proof}

\subsection{The parabolic categories of \texorpdfstring{$\frakp$}{p}-presentable modules}
\label{sec:categ-frakp-pres}

We will give now another definition of the blocks of \(\catO\{\frakp+\frakq,\calA_{\frakq}^{\frakp}\}\).
Let \(\lambda\) be a dominant integral weight for \(\gl_n\) with stabilizer
\(\bbS_\lambda\) under the dot action.
Define
  \begin{equation}
    \Lambda^{\frakp}_\frakq(\lambda) = \left\{ w \in (\bbS_n/\bbS_\lambda)^{\short} \, \left|\,
      \begin{aligned}
        &w\bbS_\lambda \subset W^\frakp\\
        &w\bbS_\lambda \cap w_\frakq W^\frakq\neq \emptyset
      \end{aligned}
\right.\right\}.\label{eq:39}
  \end{equation}
Notice that \(w_\frakq W^\frakq\) is simply the set of longest coset representatives for \(W_\frakq \backslash \bbS_n\). If \(\frakp = \frakb\) or \(\frakq=\frakb\) in the following we will omit them from the notation. If \(\lambda\) is regular then in particular
\(  \Lambda^\frakp_\frakq(\lambda) = \{w_\frakq w \suchthat w \in W^{\frakp+\frakq}\}\) is the set of elements of \(\bbS_n\) that are shortest coset representatives for \(W_\frakp \backslash \bbS_n\) and longest coset representatives for \(W_\frakq \backslash \bbS_n\).
Let
\begin{equation}
\scrP^{\frakp}_{\frakq} (\lambda) = \bigoplus_{x \in \Lambda^\frakp_\frakq(\lambda)} P^{\frakp}(x \cdot \lambda)\label{eq:137}
\end{equation}
and as in \textsection\ref{sec:category-cato} let \(\Add(\scrP^{\frakp}_{\frakq}(\lambda))\) be the full subcategory of \(\catO^{\frakp}_\lambda\) consisting of all modules which admit a direct sum decomposition with summands being direct summands of \(\scrP^{\frakp}_{\frakq}(\lambda)\).

\begin{definition}
  We define the category \(\catO^{\frakp,\frakq\pres}_\lambda\) to be
  the full subcategory of \(\catO^{\frakp}_\lambda\) which consists of
  all \(\Add(\scrP^{\frakp}_{\frakq}(\lambda))\)--presentable modules.\label{def:3}
\end{definition}

\begin{prop}
  \label{prop:1}
  For all integral dominant weights \(\lambda\), the categories \(\catO\{\frakp+\frakq,\calA_{\frakq}^{\frakp}\}_\lambda\) and \(\catO^{\frakp,\frakq\pres}_\lambda\) coincide as subcategories of \(\catO_\lambda\).
\end{prop}

\begin{proof}
  First we show the inclusion \(\catO^{\frakp,\frakq\pres}_\lambda \subseteq \catO\{\frakp+\frakq,\calA_{\frakq}^{\frakp}\}_\lambda\). Consider \(P^{\frakp}(w_\frakq \cdot \lambda)\) in \(\catO^\frakp_\lambda\).
Let \( L(\lambda|_{\gl_\frakp})  \boxtimes P(w_\frakq \cdot \lambda|_{\gl_\frakq} ) \in \catO(\gl_{\frakp+\frakq})\) denote the \((\gl_\frakp \oplus \gl_\frakq)\)--module obtained as external tensor product of the finite-dimensional simple \(\gl_\frakp\)--module \(L(\lambda|_{\gl_\frakp}) \in \catO(\gl_\frakp)\)  and  the anti-dominant indecomposable projective module \(P(w_\frakq \cdot \lambda|_{\gl_\frakq} ) \in \catO(\gl_\frakq)\). Consider it as an \(\fraka\)--module by extending the action to \(\frakh\) with the weight \(\lambda\), and then as a \((\frakp + \frakq)\)--module by letting \(\frakn_{\frakp+\frakq}\) act by zero.
By the analogue of the BGG construction of projective modules in \(\catO\) \cite{MR0407097}, we have
  \begin{equation}
P^{\frakp}(w_\frakq \cdot \lambda) = U(\gl_n) \otimes_{U(\frakp + \frakq)} \big( L(\lambda|_{\gl_\frakp}) \boxtimes P(w_\frakq \cdot \lambda|_{\gl_\frakq} )\big).\label{eq:130}
\end{equation}
Since \(U(\gl_n)\) decomposes as direct sum of finite-dimensional modules for the adjoint action of \(\gl_\frakp \oplus \gl_\frakq\), it follows that \eqref{eq:130}, as an \(\fraka\)--module, decomposes as direct sum of objects of \(\calP^\frakp_\frakq\). By tensoring \eqref{eq:130} with finite dimensional \(\gl_n\)--modules we can obtain all projective modules \(P^{\frakp}(x \cdot \lambda)\) for \(x \in \Lambda^\frakp_\frakq(\lambda)\); since \(\calP^\frakp_\frakq\) is closed under tensor product with finite dimensional modules, it follows that each \(P^\frakp(x \cdot \lambda)\) for \(x \in \Lambda^\frakp_\frakq(\lambda)\) decomposes as direct sum of objects of \(\calP^\frakp_\frakq\). Now, if \(M \in \catO_\lambda^{\frakp,\frakq\pres}\) then we have a presentation \(Q_1 \mapto Q_2 \surto M\) with \(Q_1,Q_2 \in \Add(\scrP^{\frakp}_{\frakq}(\lambda))\). Considering this as a sequence of \(\fraka\)--modules, it follows that \(M\) decomposes as a direct sum of objects of \(\calA^\frakp_\frakq=\overline \calP^\frakp_\frakq\).

Now let us show the other inclusion \(\catO\{\frakp+\frakq,\calA_{\frakq}^{\frakp}\}_\lambda \subseteq \catO^{\frakp,\frakq\pres}_\lambda\). Let \(M \in
\catO\{\frakp+\frakq,\calA_{\frakq}^{\frakp}\}_\lambda\).  By Lemma \ref{lem:24} we have \(M
\in \catO^{\frakp}_\lambda\). As an \(\fraka\)--module, \(M\) is generated
by elements of weight \(x \cdot \lambda\) with \(sx < x\) for any simple
reflection \(s \in W_\frakq\) (i.e.\ \(x \cdot \lambda\) is an anti-dominant
weight for \(\gl_\frakq\)). Of course then this is also true as a
\(\gl_n\)--module. Hence the projective cover \(Q\) of \(M\) in
\(\catO^{\frakp}_\lambda\) is an element of
\(\Add(\scrP^{\frakp}_{\frakq}(\lambda))\). Let \(K=\ker(Q \surto M)\) in
\(\catO^{\frakp}_\lambda\), and consider the short exact sequence \(K
\into Q \surto M\) as a sequence of \(\fraka\)--modules.
Since all objects of \(\calA^\frakp_\frakq\)
are finitely generated, we may assume (taking direct summands) that \(K \into Q \surto M\) is a short exact sequence of finitely generated
\(\fraka\)--modules, that is, we can suppose \(M \in \calA^\frakp_\frakq\) and, by the first paragraph, \(Q \in \calP^\frakp_\frakq\).
We can write \(Q=Q_M \oplus Q'\) where \(Q_M\) is the projective cover of
\(M\), and \(K=Q' \oplus \ker(Q_M \surto M)\). Since \(M \in
\calA^\frakp_\frakq\), we have a presentation \(P_M \mapto Q_M \surto M\) with \(P_M, Q_M \in \calP^\frakp_\frakq\), hence we have a surjective map \(P_M \surto \ker(Q_M \surto M)\) and therefore a surjective map \(P' \surto K\) for \(P' = Q' \oplus P_M \in \calP^\frakp_\frakq\). Since as an \(\fraka\)--module \(P'\) is generated by elements of weight \(x \cdot \lambda\) with \(sx<x\) for any simple reflection \(s \in W_\frakq\), the same holds for \(K\). Hence we can apply the same construction we did for \(M\) to \(K\) and get a presentation \(P \mapto Q \surto M\) with \(P,Q \in \Add(\scrP^{\frakp}_{\frakq}(\lambda))\).
\end{proof}

For \(\frakp=\frakb\) and \(\lambda=0\) we get the category
\(\catO^{\frakq\pres}_0\) of \cite{MR2139933}. The results of
\cite[\textsection{}2]{MR2139933} carry over to the case of an arbitrary
integral weight \(\lambda\). For instance, we have:
\begin{lemma}
  The category \(\catO^{\frakq\pres}_\lambda\) is an abelian category
  with a simple preserving duality and is equivalent to
  \(\lmod{\End(\scrP_{\frakq}(\lambda))}\). For \(x \in
  \Lambda_\frakq(\lambda)\) the modules \(P(x \cdot \lambda)\) are
  obviously objects of \(\catO^{\frakq\pres}_\lambda\). Each \(P(x \cdot
  \lambda)\) has a unique simple quotient \(S(x \cdot \lambda)\) in
  \(\catO^{\frakq\pres}_\lambda\), and the \(S(x \cdot \lambda)\) for \(x
  \in \Lambda_\frakq(\lambda)\) give a full set of pairwise non
  isomorphic simple objects of \(\catO^{\frakq\pres}_\lambda\).\label{lem:22}
\end{lemma}

We want to extend these results to the general case \(\frakp \neq
\frakb\). First, let us recall the definition of the Zuckermann's functor
\(\frakz : \catO \mapto \catO^{\frakp}\). Given \(M \in \catO\), the object
\(\frakz M\) is the largest quotient of \(M\) that lies in
\(\catO^{\frakp}\). The functor \(\frakz\) is right exact and \(\frakz P(x
\cdot \lambda) = P^{\frakp}(x \cdot \lambda)\) for each \(\lambda \in
\Lambda^\frakp(\lambda)\).

\begin{lemma}
  \label{lem:16}
  The category \(\catO^{\frakp,\frakq\pres}_\lambda\) coincides with the full subcategory of objects of \(\catO^{\frakq\pres}_\lambda\) that are in \(\catO^{\frakp}_\lambda\).
\end{lemma}

\begin{proof}
  Since both are full subcategories of \(\catO(\gl_n)\), we need only to prove that they have the same objects. Let \(M \in \catO^{\frakq\pres}_\lambda \cap \catO^{\frakp}_\lambda\) and consider a presentation \(P \mapto Q \mapto M \mapto 0\) with \(P,Q \in \Add(\scrP_{\frakq}(\lambda))\). Applying \(\frakz\) yields a presentation \(\frakz P \mapto \frakz Q \mapto M \mapto 0\) with \(\frakz P, \frakz Q \in \Add(\scrP^{\frakp}_{\frakq}(\lambda))\).

  The other inclusion follows by Proposition~\ref{prop:1}  and Lemma~\ref{lem:14}.
\end{proof}

\begin{lemma}
  \label{lem:15}
  The category \(\catO^{\frakp,\frakq\pres}_\lambda\) is the Serre subcategory of \(\catO^{\frakq\pres}_\lambda\) generated by the simple objects \(S(x \cdot \lambda)\) for \(x \in \Lambda^\frakp_\frakq(\lambda)\).
\end{lemma}

\begin{proof}
  First let us prove that \(S(x \cdot \lambda) \in
  \catO^{\frakq\pres}_\lambda\) is in \(\catO_\lambda^{\frakp}\) if
  \(x \in \Lambda^\frakp_\frakq(\lambda)\). Let \(P\mapto Q \surto S(x \cdot
  \lambda)\) be a presentation of \(S(x \cdot \lambda)\) with \(P,Q \in
  \Add(\scrP_{\frakq})\).  Applying the Zuckermann's functor \(\frakz\)
  yields a presentation of \(\frakz S(x \cdot \lambda)\) with \(\frakz P,
  \frakz Q \in \Add(\scrP_{\frakq}^{\frakp})\).  Since \(\frakz P(x
  \cdot \lambda) \neq 0\) (because \(L(x \cdot \lambda)\) is a quotient
  of \(P(x \cdot \lambda)\) in \(\catO\)) and \(S(x \cdot \lambda)\) is a
  quotient of \(P(x \cdot \lambda)\), it follows that \(\frakz S(x \cdot
  \lambda)\neq 0\). On the other side, \(\frakz S(x \cdot \lambda) \in
  \catO^{\frakq\pres}\) by Lemma~\ref{lem:16}. But \(\frakz S(x \cdot
  \lambda)\) is a non-zero quotient in \(\catO^{\frakq\pres}\) of the
  simple module \(S(x \cdot \lambda)\), hence \(\frakz S(x \cdot \lambda)
  = S(x \cdot \lambda)\). It follows that \(S(x \cdot \lambda) \in \catO^{\frakp,\frakq\pres}_\lambda\).

  On the other side, if \(x \in \Lambda_\frakq(\lambda)\) but \(x \notin
  \Lambda^\frakp_\frakq(\lambda)\), then clearly \(S(x \cdot \lambda) \notin
  \catO^{\frakp}_\lambda\). Since \(\catO^{\frakp}_\lambda\) is
  closed under extensions, it follows that the objects of
  \(\catO^{\frakq\pres}_\lambda\) that are also in \(\catO^{\frakp}\)
  are exactly the objects whose composition factors are of type
  \(S(x \cdot \lambda)\) for \(x \in \Lambda^\frakp_\frakq(\lambda)\).
\end{proof}

It follows that the modules \(S(x \cdot \lambda)\) for \(x \in \Lambda^\frakp_\frakq(\lambda)\) give a full set of pairwise non-isomorphic simple objects of \(\catO^{\frakp,\frakq\pres}_\lambda\). Moreover, the projective cover of \(S(x \cdot \lambda)\) is \(P^{\frakp}(x \cdot \lambda)\).

\subsubsection{The graded abelian structure}
\label{sec:abelian-structure}

The category \(\catO^{\frakp,\frakq\pres}_\lambda\) is equivalent to the category of finitely generated (right) modules over \(\End(\scrP^{\frakp}_{\frakq}(\lambda))\):
\begin{equation}
  \label{eq:74}
  \catO^{\frakp,\frakq\pres}_\lambda \cong \rmod{\End\big(\scrP^\frakp_\frakq(\lambda)\big)}.
\end{equation}
 Via this equivalence we can define on \(\catO^{\frakp,\frakq\pres}_\lambda\) a natural abelian structure. However, as we already pointed out, this abelian structure is not induced by the abelian structure of \(\catO_\lambda\).

The algebra
\(\End(\scrP^{\frakp}_{\frakq}(\lambda))\) can be obtained from
\(A_\lambda = \End(\scrP(\lambda))\) in two steps. First, let \(e_\frakp^\perp
\in \End(\scrP(\lambda))\) be the idempotent projecting onto the direct
sum of the projective modules \(P(x \cdot \lambda)\) for \( x \notin
\Lambda^\frakp(\lambda)\). Then \(\End(\scrP^{\frakp}(\lambda)) =
A_\lambda / A_\lambda e^\frakp A_\lambda\). Moreover, let \(\overline
e_\frakq \in A_\lambda / A_\lambda e_\frakp^\perp A_\lambda\) be the
idempotent projecting onto the direct sum of the projective modules
\(P^{\frakp}(x \cdot \lambda)\) for \(x \in
\Lambda^\frakp_\frakq(\lambda)\). Then
\(\End(\scrP^{\frakp}_{\frakq}(\lambda)) = \overline e_\frakq( A_\lambda /
A_\lambda e_\frakp^\perp A_\lambda) \overline e_\frakq\).

By Lemma~\ref{lem:15}, the two steps can be done also in the inverse
order: let \(e_\frakq \in A_\lambda\) be the idempotent projecting onto
the direct sum of the projective modules \(P(x \cdot \lambda)\) for \(x
\in \Lambda_\frakq(\lambda)\) (notice that \(\overline e_\frakq = e_\frakq + A_\lambda e_\frakp^\perp A_\lambda\)). Then \(\End(\scrP_{\frakq}(\lambda)) = e_\frakq
A_\lambda e_\frakq\). Moreover, let \(f_\frakp^\perp = e_\frakq e_\frakp^\perp e_\frakq \in e_\frakq A_\lambda e_\frakq\) be
the idempotent projecting onto the direct sum of the projective
modules \(P(x \cdot \lambda)\) for \(x \notin
\Lambda^\frakp_\frakq(\lambda)\). Then
\(\End(\scrP^{\frakp}_{\frakq}(\lambda))= (e_\frakq A_\lambda e_\frakq) / ( e_\frakq
A_\lambda e_\frakq f_\frakp^\perp e_\frakq A_\lambda e_\frakq)\). It follows that
\begin{equation}
(e_\frakq A_\lambda e_\frakq) / ( e_\frakq A_\lambda e_\frakq f_\frakp^\perp e_\frakq A_\lambda e_\frakq) = \overline e_\frakq(A_\lambda/A_\lambda e_\frakp^\perp A_\lambda) \overline e_\frakq.\label{eq:1}
\end{equation}
As far as we understand, this is not a trivial result, but instead a  consequence of Lemma~\ref{lem:15}.

Recall that the algebra \(A_\lambda\) has a natural grading. Since the idempotents \(e_\frakp^\perp\) and \(\overline e_\frakq\) are homogeneous, this induces a grading on the algebra \eqref{eq:1}. Summarizing:

\begin{prop}
  \label{prop:6}
  The category \(\catO_\lambda^{\frakp,\frakq\pres}\) is equivalent to the category of finite-dimensional (right) modules over a finite-dimensional positively graded algebra.
\end{prop}

We will denote by \(\catOZ^{\frakp,\frakq\pres}_\lambda\) the graded
version of \(\catO^{\frakp,\frakq\pres}_{\lambda}\), that is the
category of finitely generated \emph{graded} modules over the algebra
\eqref{eq:1}. There is an obvious forgetful functor \(\frakf:
\catOZ^{\frakp,\frakq\pres}_{\lambda} \mapto
\catO^{\frakp,\frakq\pres}_{\lambda}\), which forgets the grading. By a
\emph{graded lift} of an object \(M \in
\catO^{\frakp,\frakq\pres}_{\lambda}\) we mean an object \(\tilde M \in
\catOZ^{\frakp,\frakq\pres}_{\lambda}\) such that \(\frakf \tilde M =
M\). The object \(M\) is then called \emph{gradable}.  Of course, not all
modules are gradable (cf.\ \cite{MR2005290}), but all interesting ones
will be. In particular, the techniques of \cite{MR2005290} ensure that
simple and indecomposable projective modules are gradable, both as
objects of \(\catO_\lambda\) and of \(\catO^{\frakp,\frakq\pres}_\lambda\)
(although the grading is different). We take their standard graded
lifts to be determined by requiring that the simple head is
concentrated in degree \(0\). We will use the notation \(qM = M\langle
1\rangle\) for the graded shift of a module.

\subsubsection{The properly stratified structure}
\label{sec:prop-strat-strct}

The results of \cite[Section~2]{MR2139933} extend to the categories \(\catO^{\frakp,\frakq\pres}_\lambda\). Let us briefly sketch them.
First, we notice that the category \(\catO^{\frakp,\frakq\pres}_\lambda\)  inherits from \(\catO_\lambda\) a simple-preserving duality:
\begin{lemma}
  The algebra \eqref{eq:1} inherits from \(A_\lambda\) an anti-automorphism; this induces a
  simple-preserving duality \(*\) on
  \(\catO^{\frakp,\frakq\pres}_\lambda\).\label{lem:25}
\end{lemma}
\begin{proof}
  The category \(\catO\) has a simple-preserving duality (see for example \cite[\textsection{}3.2]{MR2428237}), which restricts to a simple-preserving duality on \(\catO^\frakp_\lambda\). This defines an anti-automorphism on \(\End_{\catO^\frakp}(\scrP^\frakp(\lambda)) = A_\lambda/A_\lambda e_\frakp^\perp A_\lambda\), which is the identity on the idempotents projecting onto the indecomposable projective modules \(P^\frakp(w \cdot \lambda)\). Hence this restricts to an anti-automorphism of \(\overline e_\frakq(A_\lambda/A_\lambda e_\frakp^\perp A_\lambda) \overline e_\frakq\), see \eqref{eq:1}, which induces, by the equivalence of categories \eqref{eq:74}, a simple-preserving duality on \(\catO^{\frakp,\frakq\pres}_\lambda\).
\end{proof}

Let \(x \in \Lambda^{\frakp}_\frakq(\lambda)\). The module \(P^{\frakp}(x \cdot \lambda)\) is the projective cover of
\(S(x \cdot \lambda)\) in \(\catO^{\frakp,\frakq\pres}_\lambda\). Given two modules \(M\), \(N\) the trace of \(M\) in \(N\) is defined to be the sum of all images of maps \(f\colon M \mapto N\), in formulas \(\Trace_M N = \sum_{f \colon M \mapto N} \image f\). Then we have
\(S(x \cdot \lambda) = P^{\frakp}(x \cdot \lambda) /
\Trace_{\scrP^{\frakp}_{\frakq}(\lambda)}(\rad P^{\frakp}(x \cdot \lambda))\) as \(\gl_n\)--modules.
Let
\(P^{\frakp}(\prec x) = \bigoplus_{w \in \Lambda^\frakp_\frakq(\lambda), w \prec x}P^{\frakp}(w \cdot
\lambda)\) and set \(\Delta (x \cdot \lambda) = P^{\frakp}(x \cdot
\lambda) / \Trace_{P^{\frakp}(\prec x) }P^{\frakp}(x \cdot \lambda)\). As in
\cite[Lemma~2.8]{MR2139933}, one can show that the modules \(\Delta(x
\cdot \lambda)\) satisfy a universal property, and as in
\cite[Proposition~2.9]{MR2139933} this can be used to show that
  \begin{equation}
    \label{eq:16}
    \Delta(x \cdot \lambda ) \cong U(\gl_n) \otimes_{U(\frakp + \frakq)} P^{(\fraka)}(x \cdot \lambda) ,
  \end{equation}
where \(P^{(\fraka)}(x \cdot \lambda)\) is the projective cover in \(\catO^{\frakp \cap \fraka}(\fraka)\) of the highest weight module with highest weight \(x \cdot \lambda\). Moreover, one can define
  \begin{equation}
    \label{eq:22}
    \oDelta(x \cdot \lambda ) \cong U(\gl_n) \otimes_{U(\frakp+\frakq)} S^{(\fraka)}(x \cdot \lambda) ,
  \end{equation}
where \(S^{(\fraka)}(x \cdot \lambda)\) is the simple module in \(\calA^\frakp_\frakq\) with highest weight \(x \cdot \lambda\).

We recall the definition of a graded \emph{properly stratified algebra} in the sense of \cite{MR2057398} (see also \cite{MR1921761}, \cite{MR2344576}).
\begin{definition}\label{def:2}
  Let \(B\) be a finite-dimensional associative graded algebra over a
  field \(\K\) with a simple-preserving duality and with equivalence
  classes of simple modules \(\{ \mathbbol L(\lambda)\langle j \rangle \suchthat \lambda \in \Lambda, j \in \Z\}\) where \((\Lambda,\prec)\) is a partially ordered finite set. For each \(\lambda
  \in \Lambda\) let:
  \begin{enumerate}[(i)]
  \item \label{soerg:item:8} \(\mathbbol P(\lambda)\) denote the
    projective cover of \(\mathbbol L(\lambda)\),
  \item \label{soerg:item:9} \(\mathbbol \Delta(\lambda)\) be the
    maximal quotient of \(\mathbbol P(\lambda)\) such that
    \([\mathbbol\Delta(\lambda):\mathbbol L(\mu)\langle i \rangle ]=0\) for all \(\mu
    \succ \lambda\), \(i \in \Z\),
  \item \label{soerg:item:16} \(\overline{\mathbbol \Delta}(\lambda)\)
    be the maximal quotient of \(\mathbbol \Delta(\lambda)\) such that
    \([\rad \overline{\mathbbol \Delta}(\lambda):\mathbbol L(\mu)\langle i \rangle]=0\)
    for all \(\mu \succeq \lambda\), \(i \in \Z\).
  \end{enumerate}
  Then \(B\) is \emph{properly stratified} if the following conditions hold for
  every \(\lambda \in \Lambda\):
  \begin{enumerate}[label=(PS\arabic*),ref=\arabic*]
  \item \label{soerg:item:17} the kernel of the canonical epimorphism
    \(\mathbbol P(\lambda) \surto \mathbbol\Delta(\lambda)\) has a
    filtration with subquotients isomorphic to graded shifts of
    \(\mathbbol\Delta(\mu)\), \(\mu \succ \lambda\);
  \item \label{soerg:item:18} the kernel of the canonical epimorphism
    \(\mathbbol\Delta(\lambda) \surto \overline{\mathbbol
      \Delta}(\lambda)\) has a filtration with subquotients isomorphic
    to graded shifts of \(\overline{\mathbbol \Delta}(\lambda)\);
  \item \label{soerg:item:19} the kernel of the canonical epimorphism
    \(\overline{\mathbbol \Delta}(\lambda) \surto \mathbbol
    L(\lambda)\) has a filtration with subquotient isomorphic to graded
    shifts of \(\mathbbol L(\mu)\), \(\mu \prec \lambda\).
  \end{enumerate}
\end{definition}
The modules \(\mathbbol\Delta(i)\) and \(\overline{\mathbbol\Delta}(i)\) are called \emph{standard} and \emph{proper standard} modules respectively.
The same argument as for \cite[Theorem~2.16]{MR2139933} gives:

\begin{theorem}
  \label{thm:4}
  The algebra \(\End(\scrP^{\frakp}_{\frakq}(\lambda))\) with the
  order induced by the Bruhat order on \(\Lambda^\frakp_\frakq(\lambda)\) is
  a graded properly stratified algebra. The modules \(\Delta(x \cdot \lambda)\) and
  \(\oDelta(x \cdot \lambda)\) are the standard and proper standard
  modules respectively.
\end{theorem}

It is easy to show that also the modules \(\Delta(x \cdot \lambda)\) and \(\overline \Delta(x \cdot \lambda)\) are gradable. Again, we choose their standard lifts by requiring the simple heads to be concentrated in degree \(0\).

\subsection{Functors between categories \texorpdfstring{$\catO^{\frakp,\frakq\pres}_\lambda$}{O}}
\label{sec:funct-betw-categ}
We conclude this section examining the natural functors that can be defined between the categories we have introduced. In particular, for \(\frakp' \supseteq \frakp\) and \(\frakq' \supseteq \frakq\) we will define functors
\begin{equation*}
\begin{tikzpicture}[baseline=(A1.base)]
  \node (A1) at (0,0) {\(\catOZ^{\frakp,\frakq\pres}_\lambda\)};
  \node (A2) at (3,0) {\(\catOZ^{\frakp',\frakq\pres}_\lambda\)};
  \draw[->,bend left=10] (A1) to node[above] {\(\frakz\)} (A2);
  \draw[->,bend left=10] (A2) to node[below] {\(\frakj\)} (A1);
\end{tikzpicture} \qquad \text{and} \qquad
\begin{tikzpicture}[baseline=(A1.base)]
  \node (A1) at (0,0) {\(\catOZ^{\frakp,\frakq\pres}_\lambda\)};
  \node (A2) at (3,0) {\(\catOZ^{\frakp,\frakq'\pres}_\lambda\)};
  \draw[->,bend left=10] (A1) to node[above] {\(\frakQ\)} (A2);
  \draw[->,bend left=10] (A2) to node[below] {\(\fraki\)} (A1);
\end{tikzpicture}
\end{equation*}

\subsubsection{Zuckermann's functors}
\label{sec:zuck-funct}
Suppose \(\frakp'\) is also a standard parabolic subalgebra of \(\gl_n\) with \(\frakp' \subset \frakp\). Let us fix an integral dominant weight \(\lambda\). We have then an inclusion functor \(\frakj :
\catO^{\frakp}_\lambda \mapto \catO^{{\frakp'}}_\lambda\). Since the
abelian structure of \(\catO^{\frakp}_\lambda\) is the restriction of
the abelian structure of \(\catO^{{\frakp'}}_\lambda\), this is an
exact functor. Using Lemma~\ref{lem:14}, we see that this restricts to an exact functor \(\frakj:
\catO^{\frakp,\frakq\pres}_\lambda \to
\catO^{{\frakp'},\frakq\pres}_\lambda\).

The left adjoint of \(\frakj : \catO^{\frakp}
\mapto \catO^{{\frakp'}}\) is the \emph{Zuckermann's functor} \(\mathfrak z :
\catO^{{\frakp'}}_\lambda \mapto \catO^{\frakp}_\lambda\), defined on \(M
\in \catO^{{\frakp'}}_\lambda\) by taking the maximal quotient that
lies in \(\catO^{\frakp}_\lambda\). The functor \(\mathfrak z\) is
right exact, but not exact in general. Being right exact, \(\frakz\) sends a presentation \(P \mapto Q \surto M\) with \(P,Q \in \Add(\scrP^{{\frakp'}}_{\frakq}(\lambda))\) to a presentation \(\frakz P \mapto \frakz Q \surto \frakz M\) of \(\frakz M\) with \(\frakz P, \frakz Q \in \Add(\scrP^{\frakp}_{\frakq}(\lambda))\), hence it restricts to a functor \(\mathfrak z :
\catO^{{\frakp'},\frakq\pres}_\lambda \mapto \catO^{\frakp,\frakq\pres}_\lambda\).

Notice that the definitions of \(\frakj\) and \(\frakz\) make sense in the graded setting too, hence we have also adjoint functors
\begin{equation}
\begin{tikzpicture}[baseline = (A.base)]
\node (A) at (0,0) {\(\catOZ^{{\frakp'},\frakq\pres}_\lambda\)};
\node (B) at (3,0) {\(\catOZ^{\frakp,\frakq\pres}_\lambda\)};
\draw[->,bend left=10] (A) to node[above] {\(\frakz\)} (B);
\draw[->,bend left=10] (B) to node[below] {\(\frakj\)} (A);
\end{tikzpicture}
\label{eq:195}
\end{equation}

\subsubsection{Coapproximation functors}
\label{sec:coappr-funct-1}

Suppose that \(\frakq'\) is a standard parabolic subalgebra of \(\gl_n\) with \(\frakq' \subseteq \frakq\) and let us fix an integral dominant weight
\(\lambda\). According to Lemma~\ref{lem:14}, we have an inclusion
functor \(\mathfrak i : \catO^{\frakp,\frakq\pres}_\lambda \to
\catO^{\frakp,{\frakq'}\pres}_\lambda\). This is right exact but
not left exact in general (cf.\ \cite[Example 2.3]{MR2139933} for an example).

Its right adjoint \(\mathfrak Q : \catO^{\frakp,{\frakq'}\pres}_\lambda
\mapto \catO^{\frakp,\frakq\pres}_\lambda\) is called {\em
  coapproximation}, and can be described Lie theoretically as
follows. Take \(M \in \catO^{\frakp,{\frakq'}\pres}_\lambda\), and let
\(p: Q \twoheadrightarrow \Trace_{\scrP^{\frakp}_{\frakq}(\lambda)}(M)\)
be a projective cover in \(\catO^{\frakp,\frakq'\pres}_\lambda\) (notice
that \(\scrP^{\frakp}_\frakq(\lambda)\) is a direct summand of
\(\scrP^{\frakp}_{\frakq'}(\lambda)\) and in particular an object of
\(\catO^{\frakp,\frakq'\pres}_\lambda\)). Then define \(\mathfrak Q(M) =
Q/\Trace_{\scrP^{\frakp}_{\frakq}(\lambda)}(\ker p)\). It is
easy to verify that \(\frakQ\) is just a Serre quotient functor, and hence it is exact; indeed, it corresponds under the equivalence of categories \eqref{eq:74} to \(\Hom_{\catO^{\frakp,\frakq'\pres}_\lambda}(\scrP^\frakp_\frakq(\lambda),\blank)\). Its left adjoint \(\fraki\), on the other hand, corresponds to the induction functor \(\blank \otimes_{\End(\scrP^\frakp_\frakq(\lambda))} \End(\scrP^\frakp_{\frakq'}(\lambda))\). In particular, there are graded lifts
\begin{align}
  \fraki = \blank \otimes_{\End(\scrP^\frakp_\frakq(\lambda))} \End(\scrP^\frakp_{\frakq'}(\lambda)) \colon \catOZ^{\frakp,\frakq\pres}_\lambda & \longrightarrow
  \catOZ^{\frakp,\frakq'\pres}_\lambda \\
  \frakQ = \Hom_{\catOZ^{\frakp,\frakq'\pres}_\lambda}(\scrP^\frakp_\frakq(\lambda),\blank) \colon
  \catOZ^{\frakp,{\frakq'}\pres}_\lambda & \longrightarrow
  \catOZ^{\frakp,\frakq\pres}_\lambda.
\end{align}

Our next goal is to compute the action of \(\frakQ\) on proper standard modules.  We will need the following easy fact:

\begin{lemma}
  \label{lem:10}
  Let \(\frakq' \subset \frakq\) and let \(w \in \Lambda^\frakp_{\frakq'}(\lambda)\). Then there exists a unique \(x \in W_\frakq\) such that \(xw \in \Lambda^\frakp_\frakq(\lambda)\) and \(\len(xw)=\len(x)+\len(w)\).
\end{lemma}

\begin{proof}
  Let \(\bbS_\lambda\) be the stabilizer of the weight \(\lambda\). Since \(\frakp\) is orthogonal to \(\frakq\), we may assume \(\frakp=\frakb\). Moreover, since \(\Lambda_{\frakq'} (\lambda) \subseteq (\bbS_n/\bbS_\lambda)^\short\), it is clearly sufficient to prove the result for \(w \in (\bbS_n/\bbS_\lambda)^\short\). Then the lemma is simply a statement about double cosets. Let \(z\) be the shortest element in the double coset \(W_\frakq w \bbS_\lambda\). Then all shortest coset representatives for \(\bbS_n/\bbS_\lambda\) contained in \(W_\frakq w \bbS_\lambda\) can be obtained as \(yz\) for \(y \in W_\frakq\) (and in particular \(w=y_1 z\) for \(y_1 \in W_\frakq\)). Let \(y_0 \in W_\frakq\) be the shortest element such that \(y_0 z\bbS_\lambda \cap (W_\frakq \backslash \bbS_n)^\longrep \neq \emptyset\) (this exists, since this is the unique element such that \(y_0 z w_\lambda\) is the longest element of the double coset \(W_\frakq w \bbS_\lambda\), where \(w_\lambda\) is the longest element of \(\bbS_\lambda\)). Setting \(x=y_0 y_1^{-1} \) we get the claim.
\end{proof}

First we suppose that we are in the extreme case \(\frakq'=\frakb\), and we compute the action of \(\frakQ\) on Verma modules.

\begin{prop}\label{prop:9}
  Consider the coapproximation functor \(\frakQ:
  \catOZ^{\frakp}_\lambda \to
  \catOZ^{\frakp,\frakq\pres}_\lambda\). Let \(w \in
  \Lambda^{\frakp}(\lambda)\), and let \(x \in W_\frakq\) be the element
  given by Lemma~\ref{lem:10} such that \(xw \in
  \Lambda^{\frakp}_\frakq(\lambda)\). Then we have \(\mathfrak Q
  M^{\frakp} (w \cdot \lambda) = q^{\len(x)}\overline \Delta(xw\cdot
  \lambda)\).
\end{prop}

For the proof we will need some preliminary results in the ungraded setting.

\begin{lemma}
  \label{lem:5}
  Suppose \(w \in \Lambda^\frakp(\lambda)\), let \(M(w\cdot \lambda)\) be
  a Verma module in \(\catO_\lambda\) and \(M^{\frakp}(w \cdot \lambda)\)
  be its parabolic quotient in \(\catO^{\frakp}_\lambda\). Then for
  every simple reflection \(s \in W_\frakq\) such that \(\len(sw)> \len
  (w)\) the map \(M^{\frakp}(sw \cdot \lambda) \mapto M^{\frakp}(w \cdot
  \lambda)\) induced at the quotient by the inclusion \(M(sw \cdot
  \lambda) \hookrightarrow M(w \cdot \lambda)\) is injective.
\end{lemma}

\begin{example}
  \label{ex:2}
  Notice that in the statement of the lemma it is essential to assume that the simple reflection \(s\) is orthogonal to the parabolic subalgebra \(\frakp\). As a counterexample when this is not true, consider the regular block \(\catO^\frakp_0(\gl_3)\), where \(\frakp \subset \gl_3\) is the standard parabolic subalgebra corresponding to the composition \((2,1)\). Then the inclusion \(M(s_2 \cdot 0) \into M(0)\) of Verma modules in \(\catO(\gl_3)\) induces a map \(M^\frakp(s_2 \cdot 0 ) \mapto M^\frakp(0)\) which is not injective (the kernel is isomorphic to the simple module \(L(s_2 s_1 \cdot 0)\)).
\end{example}

\begin{proof}[Proof of Lemma \ref{lem:5}]
  Let \(v_{sw},v_w\) be the highest weight vectors of \(M(sw\cdot \lambda)\) and
  \(M(w \cdot \lambda)\) respectively. Then (cf.\ \cite[\textsection
  1.4]{MR2428237}) the inclusion \(M(sw \cdot \lambda) \hookrightarrow
  M(w \cdot \lambda)\) is determined by \(v_{sw} \mapsto f_{\alpha_s}^k
  v_w\) for some \(k \in \N\), where \(f_{\alpha_s} \in \mathfrak n^-\) is the standard generator of \(U(\gl_n)\) corresponding to the simple root \(\alpha_s\). This
  indeed defines an injective map because the Verma modules are free
  as \(U(\mathfrak n^-)\)--modules and \(U(\mathfrak n^-)\) has no zero-divisors, both by the PBW Theorem.

  Let \(\gl_n=\frakp \oplus \mathfrak u^-_{\frakp}\). The parabolic
  Verma modules can be defined through parabolic induction, hence they
  are free as \(U(\mathfrak u^-_{\frakp})\)--modules (although in general
  not of rank one). Since the simple reflection \(s\) is orthogonal to
  the set of reflections \(W_\frakp\), the element \(f_{\alpha_s}\) lies
  in \(U(\mathfrak u^-_{\frakp})\) and the map on the quotients is again
  given by multiplication by it. By the same argument as before, this
  map has to be injective.
\end{proof}

\begin{lemma}
  \label{lem:1}
  With the same notation as before, \(\coker \big(M^{\frakp}(sw \cdot \lambda) \hookrightarrow M^{\frakp}(w \cdot \lambda)\big)\) has only composition factors of type \(L(y \cdot \lambda)\) with \(sy > y\).
\end{lemma}

\begin{proof}
  The inclusion is given by multiplication by \(f^k_{\alpha_s}\). By the
  PBW Theorem, it follows immediately that the cokernel is locally
  \(\langle f^k_{\alpha_s} \rangle_{k \in \N}\)--finite, hence all its
  composition factors are indexed by elements of \(\bbS_n\) that are
  shortest coset representatives for \(\langle s \rangle \backslash
  \bbS_n\).
\end{proof}

\begin{lemma}
  \label{lem:6}
  For every \(w \in \Lambda^\frakp(\lambda)\) and \(x \in W_\frakq\) such that \(xw \in \Lambda^\frakp(\lambda)\) we have
  \(q^{\ell(x)} \mathfrak Q M^{\frakp} (xw \cdot \lambda) =
  \mathfrak Q M^{\frakp} (w \cdot \lambda)\).
\end{lemma}

\begin{proof}
  Of course, it is sufficient to prove it for a simple reflection \(s
  \in W_\frakq\). Then the result follows from Lemma~\ref{lem:1} if we apply
  the exact functor \(\mathfrak Q\) to the short exact
  sequence
  \begin{equation}
    q M^{\frakp}(sw \cdot \lambda) \hookrightarrow M^{\frakp}(w \cdot \lambda) \twoheadrightarrow Q. \qedhere \label{eq:40}
\end{equation}
\end{proof}

\begin{lemma}
  \label{lem:18}
  Let \(w \in \Lambda^\frakp_\frakq(\lambda)\). Then \(\frakQ M^\frakp(w \cdot \lambda) = \oDelta (w \cdot \lambda)\).
\end{lemma}

\begin{proof}
  The projective module \(P^{\frakp}(w \cdot \lambda)\) has a filtration
  by parabolic Verma modules in \(\catO^{\frakp}_\lambda\). Hence the
  projective module \(P^{\frakp}(w \cdot \lambda)=\frakQ P^\frakp(w
  \cdot \lambda)\) in \(\catO^{\frakp,\frakq\pres}_\lambda\) has a
  filtration by modules \(\mathfrak Q M^\frakp(y \cdot \lambda)\) for \(y
  \in \Lambda^\frakp(\lambda)\), \(y \preceq w\).

Now the proper standard module \(\overline \Delta(w \cdot \lambda)\) is
defined to be the maximal quotient \(Q\) of \(P^{\frakp}( w \cdot \lambda)\)
in \(\catO^{\frakp,\frakq\pres}_\lambda\) satisfying
\begin{equation}
[\rad
Q:S(z \cdot \lambda)] = 0 \qquad \text{for all }z
\preceq w.\label{eq:176}
\end{equation}
Obviously the quotient
\(\mathfrak Q M^{\frakp}(w \cdot \lambda)\) at the top of
\(P^{\frakp}( w \cdot \lambda)\) satisfies \eqref{eq:176}.
Any bigger quotient contains the simple head of some \(\frakQ M^\frakp (y \cdot \lambda)\) for \(y \prec w\).
Consider such a \(y\) and let \(x' \in W_\frakq\) be the element given by Lemma~\ref{lem:10} for \(y\). By Lemma~\ref{lem:6} the simple head in \(\catO^{\frakp,\frakq\pres}\) of \(\frakQ M^\frakp(y \cdot \lambda)\) is the simple head of \(\frakQ M^\frakp(x'y \cdot \lambda)\);
but this is the simple head of \(\frakQ P^\frakp(x' y \cdot \lambda)\), that is \(S(x'y \cdot \lambda)\). Notice that \(x'y \preceq w\) (this follows because \(y \prec w\) and both \(x'y,w \in \Lambda^\frakp_\frakq(\lambda)\)). Hence \(\frakQ M^\frakp(w \cdot \lambda)\) is indeed the maximal quotient satisfying \eqref{eq:176}.
\end{proof}

The proof of the proposition follows now easily:

\begin{proof}[Proof of Proposition~\ref{prop:9}]
  By Lemma~\ref{lem:6} we have \(\frakQ M^\frakp(w \cdot \lambda) = q^{\len(x)} \frakQ M^\frakp(xw \cdot \lambda)\) and by Lemma~\ref{lem:18} this is \(q^{\len(x)} \oDelta(xw \cdot \lambda)\).
\end{proof}

From Proposition~\ref{prop:9} one can directly deduce:

\begin{corollary}
  \label{cor:2}
  Let \(\frakq'\) be a standard parabolic subalgebra of \(\gl_n\) with
  \(\frakq' \subset \frakq\) and consider the coapproximation functor
  \(\frakQ: \catOZ^{\frakp,\frakq'\pres}_\lambda \to
  \catOZ^{\frakp,\frakq\pres}_\lambda\). Let \(w \in
  \Lambda^{\frakp}_{\frakq'}(\lambda)\) and let \(x \in W_\frakq\) be the element given by Lemma~\ref{lem:10}. Then we have \(\frakQ \oDelta(w
  \cdot \lambda) = q^{\len(x)} \oDelta (xw \cdot \lambda)\).
\end{corollary}

\begin{proof}
  Let \(\frakQ_{\frakq'}: \catOZ^{\frakp}_\lambda \to
  \catOZ^{\frakp,\frakq'\pres}_\lambda\) and \(\frakQ_{\frakq}:
  \catOZ^{\frakp}_\lambda \mapto \catOZ^{\frakp,\frakq\pres}_\lambda\) be
  the coapproximation functors. It follows from the definition that
  \(\frakQ \circ \frakQ_{\frakq'} = \frakQ_\frakq\). By Proposition~\ref{prop:9} we have
  \(\frakQ_{\frakq'} M^{\frakp} (w \cdot \lambda) = \oDelta(w \cdot \lambda) \) and \(\frakQ_{\frakq} M^{\frakp} (w
    \cdot \lambda) = q^{\len(x)} \oDelta(xw \cdot \lambda)\), and the
    claim follows.
\end{proof}

The coapproximation functor \(\frakQ\) enables us to compute proper standard filtrations of standard modules:

\begin{prop}
  \label{prop:23}
  Suppose that \(\frakq\) has only one block (that is, \(W_\frakq \cong \bbS_k\) for some integer \(k\)) and let \(\lambda\) be a dominant regular weight. Then
\begin{enumerate}[(i)]
\item for all \(w \in \Lambda^\frakp_\frakq(\lambda)\) the proper standard filtration of the standard module \(\Delta(w \cdot \lambda) \in \catO^{\frakp,\frakq\pres}_\lambda\) has length \(k!\)
\item \label{item:14} in the Grothendieck group of \(\catOZ^{\frakp,\frakq\pres}_\lambda\) we have
    \(\Delta(w \cdot \lambda)] = [k]_0! \, [\oDelta(w \cdot \lambda)\).
\end{enumerate}
\end{prop}

\begin{proof}
  Since \(\lambda\) is regular, \(w\) is a longest coset representative for \({W_\frakq}\backslash{\bbS_n}\), hence \(w=w_\frakq w'\). It is well-known that in a Verma flag of the projective module \(P^\frakp(w \cdot \lambda)\) all Verma modules \(M^\frakp(xw'\cdot \lambda)\) for \(x \in W_\frakq\) appear exactly once. Applying \(\frakQ\), by Proposition~\ref{prop:9}  we get a filtration of \(P^\frakp(w\cdot \lambda)\) in \(\catO^{\frakp,\frakq\pres}_\lambda\) with \(\oDelta(w \cdot \lambda)\) appearing exactly \(k!\) times. Of course, this is the part of the filtration that builds the standard module \(\Delta(w \cdot \lambda)\). By the Kazhdan-Lusztig conjecture, in the Grothendieck group of \(\catOZ^{\frakp}\) we have
  \begin{equation}
    \label{eq:182}
    [P^{\frakp}(w \cdot \lambda)] \in \sum_{x \in W_\frakq} q^{\len(w_\frakq)-\len(x)} [M^\frakp(xw' \cdot \lambda)] + \sum_{x \in W_\frakq, z \prec w'} q\Z[q] [M^\frakp(xz \cdot \lambda)].
  \end{equation}
Applying \(\frakQ\) and considering only the part of the filtration that builds \(\Delta(w \cdot \lambda)\) we get
  \begin{equation}
    \label{eq:208}
    [\Delta(w \cdot \lambda)] = \sum_{x \in W_\frakq} q^{2(\len(w_\frakq)-\len(x))} [\oDelta(w \cdot \lambda)],
  \end{equation}
  which gives \ref{item:14}.
\end{proof}

\subsubsection{Graded lifts of translation functors}
\label{sec:grad-transl-funct}

Let \(\lambda,\mu\) be weights with stabilizers \(\bbS_\lambda\) and \(\bbS_{\mu}\). We denote by  \(\trasT_\lambda^\mu : \catO_\lambda(\gl_n) \mapto \catO_\mu (\gl_n)\) the translation functor (see \cite{MR2428237}). In \cite{MR2005290} graded lifts are introduced if \(\lambda\) is regular and \(\mu\) semi-regular (that is, \(\bbS_\mu\) is generated by one simple reflection), or the opposite. We need to work in a more general case.

We suppose that \(\bbS_\lambda \subset \bbS_\mu\). We will use the expressions \emph{translation onto the wall} and \emph{translation out of the wall} to indicate the translation functors \(\trasT_\lambda^\mu\) and \(\trasT_\mu^\lambda\) respectively (notice that in the literature these expressions are often used only when \(\lambda\) is regular). As in \cite[Section~8]{MR2005290}, it follows that the translation functors \(\trasT_\lambda^\mu\) and \(\trasT_\mu^\lambda\) have graded lifts, unique up to a shift, which we fix by imposing \(\trasT_\lambda^\mu M( \lambda) = P(\mu)\) and \(\trasT_\mu^\lambda M(\mu) = P(x_0 \cdot \lambda)\), where \(x_0\) is the longest element in \((\bbS_{\mu}/\bbS_\lambda)^{\short}\). We have then graded adjunctions
  \begin{equation}
    \label{eq:59}
    \trasT_\mu^\lambda \adjunction q^{\len(x_0)} \trasT_\lambda^\mu \qquad \text{and} \qquad \trasT_\lambda^\mu \adjunction q^{- \len(x_0)} \trasT_\mu^\lambda.
  \end{equation}
We refer to \cite[Section~4.4]{miophd2} for more details.

\subsubsection{Translation functors in \texorpdfstring{$\catOZ^{\frakp,\frakq\pres}$}{O}}
\label{sec:transl-funct}

Translation functors preserve the subcategories we have introduced:

\begin{lemma}
  \label{lem:21}
  Given two dominant weights \(\lambda,\mu\), the translation functor \(\trasT_\lambda^\mu\) restricts to a functor \(\trasT_\lambda^\mu :\catO^{\frakp,\frakq\pres}_\lambda \to \catO^{\frakp,\frakq\pres}_\mu\). Moreover, translation functors commute with the functors \(\frakj,\frakz,\fraki,\frakQ\).
\end{lemma}

\begin{proof}
  It follows directly from the definition that tensoring with a finite dimensional \(\gl_n\)--module defines an exact endofunctor of the category \(\catO\{\frakp+\frakq,\calA^{\frakp}_\frakq\}\). In particular, the translation functor \(\trasT_\lambda^\mu\) preserves the category \(\catO\{\frakp+\frakq,\calA^{\frakp}_\frakq\}\).

Since \(\frakj,\fraki\) are inclusions, it follows that \(\trasT_\lambda^\mu\) commutes with them. By adjunction, it commutes also with \(\frakz,\frakQ\).
\end{proof}

Of course we have also the graded version
\begin{equation}
\trasT_{\lambda}^{\mu}: \catOZ^{\frakp,\frakq\pres}_\lambda \mapto \catOZ^{\frakp,\frakq\pres}_\mu.\label{eq:41}
\end{equation}
 We will need the following easy result to compute the action of translation functors \eqref{eq:41} in the category \(\catOZ^{\frakp,\frakq\pres}\):

\begin{lemma}
  \label{lem:9}
  Let \(\bbS_\lambda,\bbS_\mu\) be standard parabolic subgroups of \(\bbS_n\) with \(\bbS_\lambda \subset \bbS_\mu\). Then for every \(w \in (\bbS_n/\bbS_\lambda)^\short\) there exist unique elements \(w' \in (\bbS_n/\bbS_{\mu})^\short\), \(x \in (\bbS_{\mu}/\bbS_\lambda)^\short\) such that \(w=w'x\). Moreover \(\len(w)=\len(w')+\len(x)\).
\end{lemma}

\begin{proof}
  The element \(w\) determines some coset
  \(w\bbS_{\mu}\), in which there is a unique shortest coset
  representative \(w'\). Hence \(w = w' x\) for some \(x \in
  \bbS_{\mu}\) with \(\len(w)=\len(w')+\len(x)\). Since \(w \in
  (\bbS_n / \bbS_\lambda)^\short\) we have \(\len(wt) > \len(w)\) for all
  \(t \in \bbS_\lambda\); but then also \(\len(xt) > \len(x)\) for all \(t
  \in \bbS_\lambda\), hence \(x \in (\bbS_{\mu}/\bbS_\lambda)^\short\).
\end{proof}

Now we compute how translation functors act on proper standard modules. First, we consider translation onto the wall:

\begin{prop}
  \label{prop:10}
  Let \(\lambda,\mu\) be dominant weights with stabilizers \(\bbS_\lambda,\bbS_{\mu}\) respectively, and suppose \(\bbS_\lambda \subset \bbS_\mu\). Let \(w \in \Lambda^\frakp_\frakq(\lambda)\), and write \(w=w'x\) as given by Lemma~\ref{lem:9}. Then we  have
  \begin{equation}
    \label{eq:47}
    \trasT_{\lambda}^\mu \overline \Delta(w \cdot \lambda) \cong
    \begin{cases}
      q^{- \len(x)} \overline \Delta( w' \cdot \mu) & \text{if } w' \in \Lambda^\frakp_\frakq(\mu),\\
      0 & \text{otherwise}.
    \end{cases}
  \end{equation}
\end{prop}

\begin{proof}
  First, we compute in the usual category \(\catO(\gl_n)\). It is well-known that translating a Verma module to the wall gives a Verma
  module. In fact if we forget the grading then \(\trasT_\lambda^\mu
  M(w \cdot \lambda) \cong M(w' \cdot \mu)\)
  (cf.\ \cite[Theorem~7.6]{MR2428237}). The graded version can be
  computed generalizing \cite[Theorem~8.1]{MR2005290}, and is
  \(\trasT_\lambda^\mu M(w \cdot \lambda) = q^{- \len(x)} M(w' \cdot
  \mu)\).

  Now since the functors \(\mathfrak z\) and \(\mathfrak Q\) commute with
  \(\trasT_\lambda^\mu\), using Proposition~\ref{prop:9} we have
\begin{equation}
  \label{eq:48}
    \trasT_\lambda^\mu \overline \Delta(w \cdot \lambda) \cong \trasT_\lambda^\mu
    \mathfrak Q \mathfrak z M(w \cdot \lambda) \cong \mathfrak Q \mathfrak
    z \trasT_\lambda^\mu M(w \cdot \lambda) \cong q^{- \len (x)}
    \mathfrak Q \mathfrak z M(w' \cdot \mu).
\end{equation}
If \(w' \notin \Lambda^\frakp_\frakq(\mu)\) then \(\mathfrak z M(w' \cdot \mu)\cong 0\). Otherwise we get \( \trasT_\lambda^\mu \oDelta(w \cdot \lambda) \cong q^{-\len(x)}\oDelta(w' \cdot \mu)\).
\end{proof}

Now let us compute translation of proper standard modules out of the wall:

\begin{prop}
  \label{prop:11}
  Let \(\lambda,\mu\) be dominant weights with stabilizers \(\bbS_\lambda,\bbS_{\mu}\) respectively, and suppose \(\bbS_\lambda \subset \bbS_\mu\). Then for every \(w \in \Lambda^\frakp_\frakq(\mu)\) we have
  \begin{equation}
    \label{eq:139}
    [\trasT_{\mu}^\lambda \overline \Delta(w \cdot \mu)] =
      \sum_{y \in  (\bbS_{\mu}/ \bbS_{\lambda})^{\short}} q^{\len(y_0)- \len(y)+\len(x_y)} [ \overline \Delta( x_y wy \cdot \lambda) ],
  \end{equation}
  where \(y_0\) is the longest element of \((\bbS_{\mu}/ \bbS_{\lambda})^{\short}\), and for every \(y \in (\bbS_{\mu}/\bbS_\lambda)^\short\) the element \(x_y\) is the element given by Lemma~\ref{lem:10} for \(wy \in \Lambda^\frakp(\lambda)\).
\end{prop}

Note that \(w \in \Lambda^\frakp_\frakq(\mu)\) implies that \(w \bbS_{\mu} \subseteq W^\frakp\); but as \(\bbS_\lambda \subseteq \bbS_{\mu}\) we have then \(wy \bbS_{\lambda} \subseteq W^\frakp\), and in particular \(wy \in \Lambda^\frakp(\lambda)\) for all \(y \in (\bbS_{\mu}/\bbS_\lambda)^{\short}\).

\begin{proof}
  Consider \(M(w \cdot \mu)\) in \(\catOZ\). Then we have
  \begin{equation}
    \label{eq:49}
    [\trasT_{\mu}^\lambda M(w \cdot \mu)] =
      \sum_{y \in (\bbS_{\mu}/ \bbS_{\lambda})^{\short}} q^{\len(y_0) - \len(y)} [M( wy \cdot \lambda)].
  \end{equation}
  This is well-known in the ungraded setting (see for example \cite[Theorem~7.12]{MR2428237}); the graded version follows as in \cite{MR2120117}. Notice that \(wy\bbS_\lambda \subseteq w \bbS_\mu \subseteq W^\frakp\) for all \(y \in ({\bbS_\mu}/{\bbS_\lambda})^{\textrm{short}}\). In particular, \(\frakz M(wy \cdot \lambda) \not\cong 0\) for all \(y \in ({\bbS_\mu}/{\bbS_\lambda})^{\textrm{short}}\). Since the Zuckermann's functor \(\frakz\) is exact on modules that admit a Verma flag with Verma modules \(M(z \cdot \lambda)\) such that \(\frakz M(z \cdot \lambda) \not\cong 0\) (\cite[Lemma~5.5.5]{miophd2}), we can apply \(\frakz\) to both sides of \eqref{eq:49}. Hence we get in \(\catOZ^\frakp\):
  \begin{equation}
    \label{eq:207}
    [\trasT_{\mu}^\lambda M^\frakp(w \cdot \mu)] =
      \sum_{y \in (\bbS_{\mu}/ \bbS_{\lambda})^{\short}} q^{\len(y_0) - \len(y)} [M^\frakp( wy \cdot \lambda)].
  \end{equation}
Now we can apply the exact functor \(\frakQ\) to both sides. Using Proposition~\ref{prop:9} and the commutativity of \(\mathfrak Q\) with \(\trasT_\mu^\lambda\) we obtain the claim.
\end{proof}

Now we compute translations of projective modules out of the wall:

\begin{prop}\label{prop:7}
  Let \(\lambda,\mu\) be dominant weights with stabilizers \(\bbS_\lambda,\bbS_{\mu}\) respectively, and suppose that \(\bbS_\lambda \subset \bbS_\mu\). Then for every \(w \in \Lambda^{\frakp}_\frakq(\mu)\) we have in \(\catOZ^{\frakp,\frakq\pres}\):
  \begin{equation}
    \trasT_\mu^{\lambda} P^\frakp(w \cdot
    \mu) = P^\frakp(wy_0 \cdot \lambda)\label{eq:42}
  \end{equation}
  where  \(y_0\) is the longest element of \((\bbS_{\mu}/\bbS_{\lambda})^{\short}\).
\end{prop}

\begin{proof}
  Let \(P (w \cdot \lambda) \in \catOZ\). By \cite[Theorem~7.11]{MR2428237}
  we have \(\trasT_\mu^{\lambda} P (w\cdot \mu) = P (wy_0 \cdot
  \lambda)\) as ungraded modules. By \eqref{eq:207}, the top
  Verma module is not shifted under translation, hence this also
  holds as graded modules. Applying the Zuckermann's functor \(\frakz\) we get \eqref{eq:42} in \(\catOZ^\frakp\), hence also in \(\catOZ^{\frakp,\frakq\pres}\).
  Notice that we get for free that \(w y_0 \in \Lambda^\frakp_\frakq(\lambda)\) (although it would be easy to check it directly).
\end{proof}

Using the adjunctions \eqref{eq:59} we can then compute translations of simple modules onto the wall:

\begin{prop}\label{prop:8}
  Let \(\lambda,\mu\) be dominant weights with stabilizers \(\bbS_\lambda,\bbS_{\mu}\) respectively, and suppose that \(\bbS_\lambda \subset \bbS_\mu\). Let \(y_0\) be the longest element of \((\bbS_{\mu}/\bbS_{\lambda})^{\short}\). Then for every \(w \in \Lambda^\frakp_\frakq(\lambda)\) we have in \(\catOZ^{\frakp,\frakq\pres}\)
  \begin{equation}
    \trasT_\lambda^{\mu} S(w \cdot
    \lambda) =
    \begin{cases}
      q^{-\len(y_0)} S(z \cdot \mu) & \text{if } w = zy_0 \text{ for some }z \in \Lambda^\frakp_\frakq(\mu)
      \in \bbS_{\mu}, \\
      0 & \text{otherwise.}
    \end{cases}
\label{eq:26}
  \end{equation}
\end{prop}

\begin{proof}
  We use the previous result together with the adjunction \(\trasT_{\mu}^{\lambda} \adjunction q^{\len(y_0)} \trasT_\lambda^\mu\). For every projective module \(P^\frakp(z \cdot\mu) \in \catOZ^{\frakp,\frakq\pres}_\mu\) we have
  \begin{equation}
    \label{eq:32}
    \Hom(\trasT_\mu^{\lambda} P^\frakp(z \cdot\mu), S(w \cdot\lambda)) \cong \Hom(P^\frakp(z \cdot\mu), q^{\len(y_0)}\trasT_{\lambda}^\mu S(w \cdot \lambda)).
  \end{equation}
  The left hand side is \(0\) unless \(w =zy_0\), in which case it is \(\C\), and the claim follows.
\end{proof}

\section{The categorification}
\label{sec:categorification}

In this section, which contains the main theorems of the paper, we
 construct explicitly the categorification of the representations
in the category \(\catRep\). We will define the
categorification itself in \textsection\ref{sec:grothendieck-group}
and construct the action of the intertwining operators and of \(\Uqgl\)
in \textsection\ref{sec:action-intertwiners} and
\textsection\ref{sec:functors-scr-e} respectively. In
\textsection\ref{sec:form} we will categorify the bilinear form
\eqref{eq:53} and in \textsection\ref{sec:canonical-basis} we will
prove that the indecomposable projective modules categorify the
canonical basis.

\paragraph{Notation}
\label{sec:notation}

For every composition \(\bolda\) of some \(n\) we fix, once and
forever, a dominant integral weight \(\lambda_\bolda\)
for \(\gl_n\) with stabilizer \(\bbS_\bolda\) under the dot
action. We suppose for future notational convenience that if
\(\compn\) is the regular composition of \(n\)
\eqref{eq:167} then \(\lambda_\compn=0\). Fix now a
positive integer \(n\) and \(k \in \{0,\ldots,n\}\). If
\(\Pi=\{\alpha_1,\ldots,\alpha_{n-1}\}\) is the set of the
simple roots of \(\gl_n\), we let \(\frakp\) and \(\frakq\) be the
standard parabolic subalgebras of \(\gl_n\) with corresponding
sets of simple roots
\(\Pi_\frakq=\{\alpha_1,\ldots,\alpha_{k-1}\}\) and
\(\Pi_\frakp=\{\alpha_{k},\ldots,\alpha_{n-1}\}\), so that
\( \bbS_k \times \bbS_{n-k} \cong W_{\frakp+\frakq} \subseteq
\bbS_n\). We set
\begin{equation}
\Lambda_k(\bolda) =
\Lambda^\frakp_\frakq(\lambda_\bolda) \qquad \text{and} \qquad \calQ_{k}(\bolda)
= \catOZ^{\frakp,\frakq\pres}_{\lambda_\bolda}.\label{eq:205}
\end{equation}
From now on, for \(w \in
\Lambda_k(\bolda)\) we denote by \(S_{\bolda,k}(w) \in \calQ_k(\bolda)\) the simple
module \(S(w \cdot \lambda_\bolda)\) and by \(Q_{\bolda,k}(w)\) its projective cover
\(P^\frakp(w \cdot \lambda_\bolda)\). We let also \(\Delta_{\bolda,k}(w)\) and
\(\oDelta_{\bolda,k}(w)\) be the corresponding standard and proper standard module. We will sometimes omit the subscripts \(k\) and \(\bolda\) when there will be no risk of confusion.

Fix a positive integer \(n\) and a composition \(\bolda\) of \(n\).

\subsection{Combinatorics of tableaux}
\label{sec:comb-tabl}

Given an integer \(k\) with \(0 \leq k \leq n\), recall that a \emph{hook
  partition} of shape \((n-k,k)\) is made of a row of length
\(n-k\) and a column of length \(k\), arranged as shown in
Figure \ref{fig:hook}. We call the first row just the \emph{row} and the first column just the \emph{column} of the hook partition. Keep in mind that for us the box in
the corner belongs to the row, but not to the column. Therefore we display the column slightly detached from the row.  If
\(\bolda =(a_1,\ldots,a_\ell)\) is a composition of \(n\), a
\emph{\((n-k,k)\)--tableau of type \(\bolda\)} is a tableau filled with
the integers
\begin{equation}
  \label{eq:62}
  \underbrace{1,\ldots,1}_{a_1\text{ times}},\underbrace{2,\ldots,2}_{a_2\text{ times}},\ldots,\underbrace{\ell,\ldots,\ell}_{a_\ell\text{ times}}.
\end{equation}

If we number the boxes of the hook partition of shape \((n-k,k)\) from
\(1\) to \(n\) starting with the column from the bottom to the top and
ending with the row from the left to the right, then the permutation
group \(\bbS_n\) acts from the left on the set of \((n-k,k)\)--tableaux of
type \(\bolda\) permuting the boxes. The stabilizer of this action is
\(\bbS_\bolda\).

Define the \emph{minimal} \((n-k,k)\)--tableau \(T_\bolda^{\mathrm{min}}\)
of type \(\bolda\) to be the tableau obtained putting the numbers
\eqref{eq:62} in order first in the column, from the bottom to the
top, then in the row, from the left to the right (see Figure
\ref{fig:admissible-hook}). Set also
\begin{equation}
T_\bolda(w) = w \cdot T_\bolda^{\mathrm{min}} \label{eq:96}
\end{equation}
for each \(w \in \bbS_n\).
Then we can define a bijection \(w \mapsto T_\bolda(w)\) between
  \((\bbS_n/\bbS_\bolda)^{\textrm{short}}\) and \((n-k,k)\)--tableaux of type
  \(\bolda\).

\begin{figure}
  \centering
  \begin{tikzpicture}[scale=0.5]
    \draw (0,0) rectangle (3,1);
    \draw (1,0) -- (1,1);
    \draw (2,0) -- (2,1);
    \begin{scope}[yshift=-0.2cm]
      \draw (0,0) rectangle (1,-2);
      \draw (0,-1) -- (1,-1);
    \end{scope}
    \begin{scope}[xshift=7cm]
      \draw (0,0) rectangle (2,1);
      \draw (1,0) -- (1,1);
      \begin{scope}[yshift=-0.2cm]
        \draw (0,0) rectangle (1,-3);
        \draw (0,-1) -- (1,-1);
        \draw (0,-2) -- (1,-2);
      \end{scope}
    \end{scope}
  \end{tikzpicture}
\caption{Hook partitions of shape (3,2) and (2,3).}\label{fig:hook}
\end{figure}
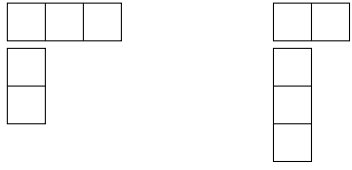

\begin{figure}
  \centering
  \begin{tikzpicture}[scale=0.5]
    \draw (0,0) rectangle (3,1);
    \draw (1,0) -- (1,1);
    \draw (2,0) -- (2,1);
    \node at (0.5,0.5) {\(3\)};
    \node at (1.5,0.5) {\(4\)};
    \node at (2.5,0.5) {\(4\)};
    \begin{scope}[yshift=-0.2cm]
      \draw (0,0) rectangle (1,-4);
      \draw (0,-1) -- (1,-1);
      \draw (0,-2) -- (1,-2);
      \draw (0,-3) -- (1,-3);
      \node at (0.5,-0.5) {\(3\)};
      \node at (0.5,-1.5) {\(2\)};
      \node at (0.5,-2.5) {\(2\)};
      \node at (0.5,-3.5) {\(1\)};
    \end{scope}
  \end{tikzpicture}
  \hspace{1cm}
  \begin{tikzpicture}[scale=0.5]
    \draw (0,0) rectangle (3,1);
    \draw (1,0) -- (1,1);
    \draw (2,0) -- (2,1);
    \node at (0.5,0.5) {\(2\)};
    \node at (1.5,0.5) {\(3\)};
    \node at (2.5,0.5) {\(4\)};
    \begin{scope}[yshift=-0.2cm]
      \draw (0,0) rectangle (1,-4);
      \draw (0,-1) -- (1,-1);
      \draw (0,-2) -- (1,-2);
      \draw (0,-3) -- (1,-3);
      \node at (0.5,-0.5) {\(4\)};
      \node at (0.5,-1.5) {\(2\)};
      \node at (0.5,-2.5) {\(3\)};
      \node at (0.5,-3.5) {\(1\)};
    \end{scope}
  \end{tikzpicture}
  \hspace{1cm}
  \begin{tikzpicture}[scale=0.5]
    \draw (0,0) rectangle (3,1);
    \draw (1,0) -- (1,1);
    \draw (2,0) -- (2,1);
    \node at (0.5,0.5) {\(2\)};
    \node at (1.5,0.5) {\(2\)};
    \node at (2.5,0.5) {\(4\)};
    \begin{scope}[yshift=-0.2cm]
      \draw (0,0) rectangle (1,-4);
      \draw (0,-1) -- (1,-1);
      \draw (0,-2) -- (1,-2);
      \draw (0,-3) -- (1,-3);
      \node at (0.5,-0.5) {\(1\)};
      \node at (0.5,-1.5) {\(3\)};
      \node at (0.5,-2.5) {\(3\)};
      \node at (0.5,-3.5) {\(4\)};
    \end{scope}
  \end{tikzpicture}
  \hspace{1cm}
  \begin{tikzpicture}[scale=0.5]
    \draw (0,0) rectangle (3,1);
    \draw (1,0) -- (1,1);
    \draw (2,0) -- (2,1);
    \node at (0.5,0.5) {\(1\)};
    \node at (1.5,0.5) {\(3\)};
    \node at (2.5,0.5) {\(4\)};
    \begin{scope}[yshift=-0.2cm]
      \draw (0,0) rectangle (1,-4);
      \draw (0,-1) -- (1,-1);
      \draw (0,-2) -- (1,-2);
      \draw (0,-3) -- (1,-3);
      \node at (0.5,-0.5) {\(2\)};
      \node at (0.5,-1.5) {\(2\)};
      \node at (0.5,-2.5) {\(3\)};
      \node at (0.5,-3.5) {\(4\)};
    \end{scope}
  \end{tikzpicture}
\caption{These are \((3,4)\)--tableaux of type \((1,2,2,2)\). The leftmost tableau is the minimal one. Notice that only the last one is admissible.}\label{fig:admissible-hook}
\end{figure}
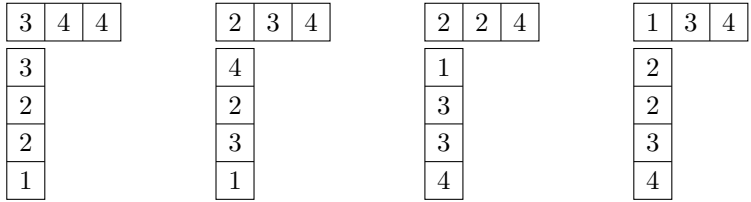

\newlength{\oldparindent}
\setlength{\oldparindent}{\parindent}

\noindent\begin{minipage}[t]{\linewidth-2cm}
  \setlength{\parindent}{\oldparindent}
   We say that a tableau is \emph{admissible} if:
  \begin{enumerate}[(a)]
  \item\label{item:12} the entries in the row are strictly increasing (from left to
    right),
  \item\label{item:13} the entries in the column are non-increasing (from the bottom
    to the top),
  \end{enumerate}
  as shown in the picture on the right.
  For an example see the last tableau in Figure \ref{fig:admissible-hook}.
\end{minipage}
\hfill
\begin{minipage}[c][0cm][t]{1.5cm}
    \begin{tikzpicture}[anchorbase,scale=0.5]
    \draw (0,0) rectangle (3,1);
    \node at (1.5,0.5) {\(<\)};
    \begin{scope}[yshift=-0.2cm]
      \draw (0,0) rectangle (1,-3);
      \node[rotate=90] at (0.5,-1.5) {\(\geq\)};
    \end{scope}
  \end{tikzpicture}
\end{minipage}
\hspace{0.05cm}

\begin{prop}
  \label{prop:12}
  The bijection
  \begin{equation}
    \begin{aligned}
      (\bbS_n/\bbS_\bolda)^\short & \xleftrightarrow{\,1-1\,}
      \{(n-k,k)\text{-tableaux of type }\bolda\}\\
      w & \xmapsto{\hphantom{\,1-1\,}} T_\bolda(w)
    \end{aligned}\label{eq:191}
  \end{equation}
  restricts to a bijection
  \begin{equation}
    \Lambda_k(\bolda) \xleftrightarrow{\,1-1\,} \{\text{admissible } (n-k,k)\text{-tableaux of type }\bolda\}.\label{eq:192}
  \end{equation}
\end{prop}

\begin{proof}
  Given \(w \in (\bbS_n/\bbS_\bolda)^\short\), it is enough to observe
  that the condition \ref{item:12} is equivalent to \(w \in W^\frakp\)
  and the condition \ref{item:13} is equivalent to \(w \bbS_\bolda \cap
  w_\frakq W^{\frakq} \neq \varnothing\).

\end{proof}

\subsection{The Grothendieck group}
\label{sec:grothendieck-group}

Let \(\calA\) be an abelian category. We recall that the Grothendieck group \(K(\calA)\) is the quotient of the free \(\Z\)--module on generators \([M]\) for \(M\in \calA\) modulo the relation \([B]=[A]+[C]\) for each short exact sequence \(A \into B \surto C\). If the category \(\calA\) is graded then \(K(\calA)\) becomes a \(\Z[q,q^{-1}]\)--module under \(q[M]=[qM]=[M \langle 1 \rangle]\). For an abelian graded category \(\calA\) we let moreover
\begin{equation}
K^{\C(q)}(\mathcal A) = \C(q) \otimes_{\Z[q,q^{-1}]} K(\mathcal A).\label{eq:85}
\end{equation}

Fix an integer \(0 \leq k \leq n\). 
A basis of the Grothendieck group \(K(\calQ_k(\bolda))\) as a
\(\Z[q,q^{-1}]\)--module is given by the simple modules
\(S_{\bolda,k} (w)\) for \(w \in \Lambda_k(\bolda)\). Since \(\calQ_k(\bolda)\) is properly stratified, the matrix which expresses the proper standard modules in the basis given by the simple modules is lower triangular (with respect to the ordering \(\prec\)), with ones on the diagonal. Hence
equivalence classes of the proper standard modules also give a basis. On the other side, the standard modules do not give a basis over \(\Z[q,q^{-1}]\) in general (although they always give a basis of \(K^{\C(q)}(\calQ_k(\bolda))\) over \(\C(q)\)).

According to Proposition~\ref{prop:12}, the set \(\Lambda_k(\bolda)\) is in bijection with the set of admissible \((n-k,k)\)--tableaux of type \(\bolda\). For \(w \in \Lambda_k(\bolda)\) let \(v_{(w)}= v^\bolda_\boldeta \in V(\bolda)\), where
\begin{equation}
  \label{eq:63}
  \eta_i = \begin{cases}
    1 & \text{if the number \(i\) appears in the row of \(T_\bolda(w)\)},\\
    0 & \text{otherwise}.
  \end{cases}
\end{equation}
We write also \(v_{(T_\bolda(w))}=v_{(w)}\).
We can then define an isomorphism
\begin{equation}
  \label{eq:30}
  \begin{aligned}
    K^{\C(q)}(\calQ_k(\bolda)) & \longrightarrow V(\bolda)_{k}\\
    [\overline \Delta_{\bolda,k}(w)] & \longmapsto \frac{1}{(v_{(w)},v_{(w)})_\bolda} v_{(w)}.
  \end{aligned}
\end{equation}
Notice that if \(\bolda=(a_1,\ldots,a_\ell)\) then for \(k<n-\ell\) the category \(\calQ_k(\bolda)\) is empty. We set
\begin{equation}
\calQ(\bolda) = \bigoplus_{k=n-\ell}^n \calQ_k(\bolda)\label{eq:183}
\end{equation}
and we get an isomorphism
\begin{equation}
  \label{eq:50}
  K^{\C(q)}(\calQ(\bolda)) \cong V(\bolda).
\end{equation}

\begin{remark}\label{rem:14}
  Notice that for \(M \in \calQ_k(\bolda)\)
  we have \([(q M)^*] = q^{-1} M^*\)
  in the Grothendieck group, where \(M^*\)
  denotes the dual of \(M\).
  This follows immediately since the duality comes from an
  hones dual space construction on modules over the algebra
  corresponding to \(\calQ_k(\bolda)\),
  see Lemma~\ref{lem:25}.
\end{remark}

\subsection{Categorification of the intertwiners}
\label{sec:action-intertwiners}

Let \(\Cat\) be the category whose objects are finite direct sums of the
categories \(\calQ_k(\bolda)\) for all \(n \geq 0\), \(0 \leq k \leq n\) and
for all compositions \(\bolda\) of \(n\), and whose morphisms are all
functors between these categories. We define a functor \(\funcF :
\catWeb \mapto \Cat\) as follows. If \(\bolda=(a_1,\ldots,a_\ell)\) is an
object of \(\catWeb\) with \(n=\sum a_i\), then we set
\begin{equation}
\funcF(\bolda) =  \calQ(\bolda).  \label{eq:o51}
\end{equation}
If \(\lambda_\bolda, \lambda_{\bolda'}\) with \(n= \sum a_i = \sum a'_j\) are the fixed dominant weights
of \(\gl_n\) with stabilizers \(\bbS_\bolda, \bbS_{\bolda'}\) let us denote
\(\trasT_\bolda^{\bolda'}=\trasT_{\lambda_\bolda}^{\lambda_{\bolda'}}\). Then we define \(\funcF\) on the elementary webs \eqref{catO:eq:194} and \eqref{catO:eq:198} by
\begin{equation}
\funcF(\webjoin_{\bolda,i})=\trasT_\bolda^{\boldsymbol{\hat a}_i} \qquad \text{and} \qquad \funcF(\websplit^{\bolda,i}) = \trasT_{\boldsymbol{\hat a}_i}^\bolda \label{eq:199}
\end{equation}
where \(\boldsymbol{\hat a}_i\) was defined in
\eqref{catO:eq:196}. 

\begin{lemma}
  \label{lem:23}
  The assignment \eqref{eq:199} defines a functor \(\funcF\colon \catWeb \mapto \Cat\).
\end{lemma}

\begin{proof}
  We need to check that translation functors satisfy isotopy invariance and the relations (\ref{eq:O53}--\ref{eq:54}). 
  It is known that these relations are satisfied by translation functors on \(\catOZ\) (see \cite[Section~4.5]{miophd2} for detailed proofs). Recall from Lemma~\ref{lem:21} that the translation functors restrict to the subquotient categories \(\calQ_k(\bolda)\). Of course these restricted translation functors also satisfy the relations (\ref{eq:O53}--\ref{eq:54}).
\end{proof}

The functor \(\funcF\) categorifies the functor \(\funcT\) (cf.\ \textsection\ref{reps:sec:diagr-intertw-oper}):

\begin{theorem}
  \label{thm:1}
  The following diagram commutes:
  \begin{equation}
\begin{tikzpicture}[baseline=(current bounding box.center)]
  \matrix (m) [matrix of math nodes, row sep=3em, column
  sep=5.5em, text height=1.5ex, text depth=0.25ex] {
     & \Cat \\
    \catWeb & \catRep \\};
  \path[->] (m-2-1) edge node[above left] {\( \funcF \)} (m-1-2);
  \path[->] (m-1-2) edge node[auto] {\( K^{\C(q)} \)} (m-2-2);
  \path[->] (m-2-1) edge node[below] {\( \funcT\)} (m-2-2);
\end{tikzpicture}\label{eq:52}
\end{equation}
\end{theorem}

\begin{proof}
  Let \(\bolda=(a_1,\ldots,a_\ell)\) and
  \(\bolda'=\hat\bolda_i\).
We need
  to show that
  \(K^{\C(q)}(\trasT_\bolda^{\bolda'})=\funcT(\webjoin_{\bolda,i})\) and
  \(K^{\C(q)}(\trasT_{\bolda'}^{\bolda})=\funcT(\websplit_{\bolda,i})\). Of course it is sufficient to check this on the basis of proper standard modules. Hence it suffices to check that
  \begin{align}
    [\trasT_\bolda^{\bolda'} \oDelta_{\bolda,k}( w)] &= \funcT(\webjoin_{\bolda,i}) [\oDelta_{\bolda,k}(w)],\label{eq:184}\\
[\trasT_{\bolda'}^{\bolda} \oDelta_{\bolda',k}( w')]& = \funcT(\websplit^{\bolda,i}) [\oDelta_{\bolda',k} (w')]\label{eq:185}
  \end{align}
for all \(w \in \Lambda_k(\bolda)\) and \(w' \in \Lambda_k(\bolda')\) (for all possible values of \(k\)).

Let us fix \(k\) and start with \eqref{eq:184}. Fix \(w \in
\Lambda_k(\bolda)\) and write \(w=w'x\) with \(w' \in
(\bbS_n/\bbS_{\bolda'})^\short\), \(x \in
(\bbS_{\bolda'}/\bbS_\bolda)^\short\) as given by Lemma~\ref{lem:9}.
By Proposition~\ref{prop:10} we have
\begin{equation}\label{eq:197}
\trasT_\bolda^{\bolda'}
\overline \Delta(w \cdot \lambda) =
\begin{cases}
q^{- \len(x)} \overline \Delta(w'
\cdot \mu) &\text{if }w' \in \Lambda_k(\bolda'),
\\ 0 & \text{otherwise.}
\end{cases}
\end{equation}
In what follows, we
only write the \(i\)--th and \((i+1)\)--th tensor factors of \(v_{(w)}\) and
the \(i\)--th tensor factor of \(v_{(w')}\), since the other ones are
clearly the same. Let \(T_\bolda(w)\) be the \((n-k,k)\)--tableau of type
\(\bolda\) corresponding to \(w\), and notice that the tableau
\(T_{\bolda'}(w)\) can be obtained from \(T_\bolda(w)\) by decreasing by
one all entries greater or equal to \(i+1\).

We have four cases (see Figure \ref{fig:four-cases}):
\begin{enumerate}[label=(\alph*)]
\item \label{item:4} If \(v_{(w)}=v^{a_i}_1 \otimes v^{a_{i+1}}_1\) then
  \(T_\bolda(w)\) has both an entry \(i\) and an entry \(i+1\) in the
  row. Then \(T_{\bolda'}(w)\) has two entries \(i\) in the row, and is
  not admissible; of course this also holds for \(T_{\bolda'}(w')\)
  since \(w' = w x^{-1}\). Hence \(w' \notin \Lambda_k(\bolda')\) and
  \(\trasT_\bolda^{\bolda'} \overline \Delta_{\bolda,k}(w)=0\).
\item \label{item:5} If \(v_{(w)}=v^{a_i}_1 \otimes v^{a_{i+1}}_0\) then
  \(T_\bolda(w)\) has an entry \(i\) but no entry \(i+1\) in the row. It is
  easy to see that in this case \(x\) is a permutation of length
  \(a_{i+1}\) composed with the longest element of
  \(({\bbS_{a_i+a_{i+1}-1}}/({\bbS_{a_i-1} \times
  \bbS_{a_{i+1}}}))^\short\) and therefore
\begin{equation}
\trasT_\bolda^{\bolda'} \overline
  \Delta_{\bolda,k}(w) = q^{- a_{i+1}} q^{-(a_i-1)a_{i+1}}
  \overline \Delta_{\bolda',k}(w').\label{eq:o88}
  \end{equation}
\item \label{item:6} If \(v_{(w)}=v^{a_i}_0 \otimes
  v^{a_{i+1}}_1\) then \(T_\bolda(w)\) has an entry \(i+1\) but
  no entry \(i\) in the row. Then \(x\) is the longest element
  of \(({\bbS_{a_i+a_{i+1}-1}}/({\bbS_{a_i} \times
  \bbS_{a_{i+1}-1}}))^\short\) and therefore
\begin{equation}
  \trasT_\bolda^{\bolda'} \overline \Delta_{\bolda,k}(w)
  = q^{-a_i(a_{i+1}-1)} \overline \Delta_{\bolda',k}(w').\label{eq:o89}
  \end{equation}
\item \label{item:7} If \(v_{(w)}=v^{a_i}_0 \otimes
  v^{a_{i+1}}_0\) then all entries \(i\) and \(i+1\) of
  \(T_\bolda(w)\) are in the column. Then \(x\) is the longest
  element of
  \(({\bbS_{a_i+a_{i+1}}}/({\bbS_{a_i} \times
    \bbS_{a_{i+1}}}))^\short\) and hence 
  \begin{equation}
\trasT_\bolda^{\bolda'}
  \overline \Delta_{\bolda,k}(w) = q^{-a_i a_{i+1}}
  \overline \Delta_{\bolda',k}(w').\label{eq:o90}
  \end{equation}
\end{enumerate}

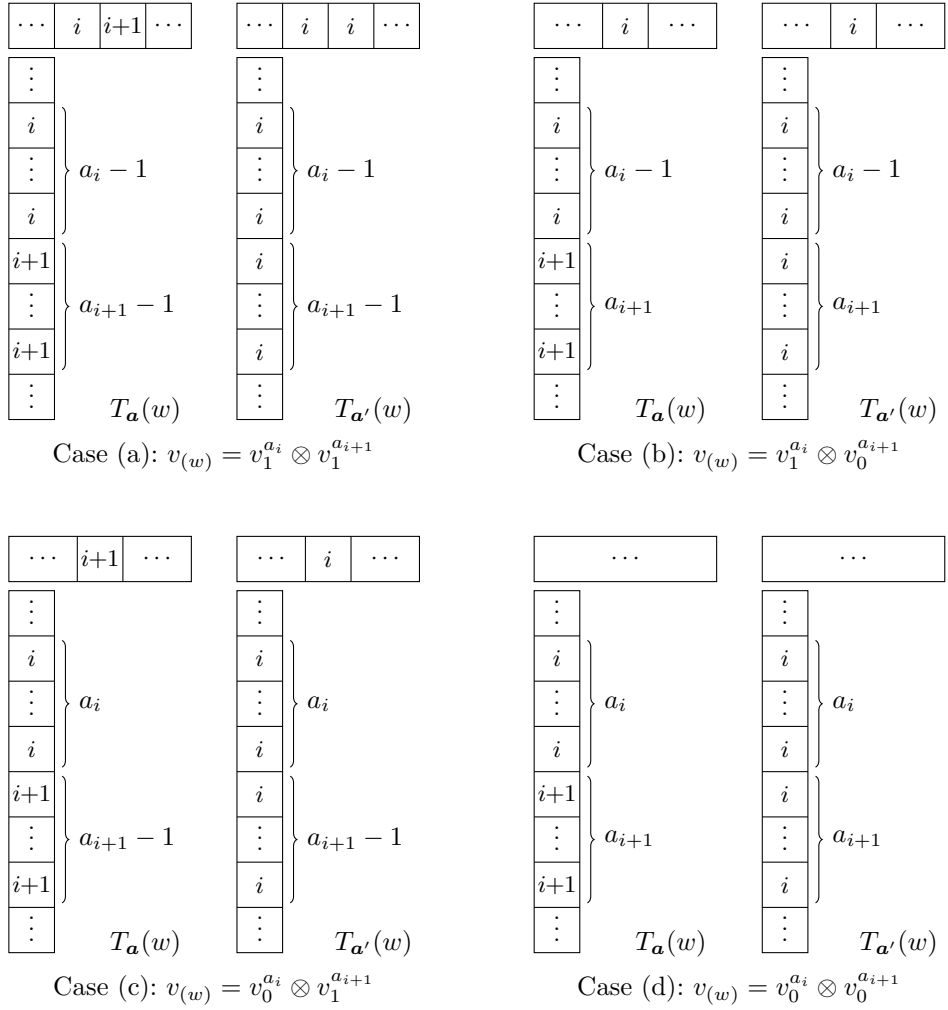
\begin{figure}
  \centering
  \begin{tikzpicture}[scale=0.6,entry/.style={font=\small}]
    \draw (0,0) rectangle (4,1);
    \draw (1,0) -- (1,1);
    \draw (2,0) -- (2,1);
    \draw (3,0) -- (3,1);
    \node[entry] at (0.5,0.5) {\(\cdots\)};
    \node[entry] at (1.5,0.5) {\(i\)};
    \node[entry] at (2.5,0.5) {\(i\hspace{-0.6ex}+\hspace{-0.6ex}1\)};
    \node[entry] at (3.5,0.5) {\(\cdots\)};
    \begin{scope}[yshift=-0.2cm]
      \draw (0,0) rectangle (1,-8);
      \draw (0,-1) -- ++(1,0);
      \draw (0,-2) -- ++(1,0);
      \draw (0,-3) -- ++(1,0);
      \draw (0,-4) -- ++(1,0);
      \draw (0,-5) -- ++(1,0);
      \draw (0,-6) -- ++(1,0);
      \draw (0,-7) -- ++(1,0);
      \node[entry] at (0.5,-0.3) {\(\vdots\)};
      \node[entry] at (0.5,-1.5) {\(i\)};
      \node[entry] at (0.5,-2.3) {\(\vdots\)};
      \node[entry] at (0.5,-3.5) {\(i\)};
      \node[entry] at (0.5,-4.5) {\(i\hspace{-0.6ex}+\hspace{-0.6ex}1\)}; 
      \node[entry] at (0.5,-5.3) {\(\vdots\)};
      \node[entry] at (0.5,-6.5) {\(i\hspace{-0.6ex}+\hspace{-0.6ex}1\)}; 
      \node[entry] at (0.5,-7.3) {\(\vdots\)};
      \draw[decoration={brace,raise=0.1cm},decorate] (1,-1.1) -- node[right=0.2cm] {\(a_i-1\)} (1,-3.9);
      \draw[decoration={brace,raise=0.1cm},decorate] (1,-4.1) -- node[right=0.2cm] {\(a_{i+1}-1\)} (1,-6.9);
    \end{scope}
    \node at (3,-8) {\(T_\bolda(w)\)};
  \begin{scope}[xshift=5cm] 
    \draw (0,0) rectangle (4,1);
    \draw (1,0) -- (1,1);
    \draw (2,0) -- (2,1);
    \draw (3,0) -- (3,1);
    \node[entry] at (0.5,0.5) {\(\cdots\)};
    \node[entry] at (1.5,0.5) {\(i\)};
    \node[entry] at (2.5,0.5) {\(i\)};
    \node[entry] at (3.5,0.5) {\(\cdots\)};
    \begin{scope}[yshift=-0.2cm]
      \draw (0,0) rectangle (1,-8);
      \draw (0,-1) -- ++(1,0);
      \draw (0,-2) -- ++(1,0);
      \draw (0,-3) -- ++(1,0);
      \draw (0,-4) -- ++(1,0);
      \draw (0,-5) -- ++(1,0);
      \draw (0,-6) -- ++(1,0);
      \draw (0,-7) -- ++(1,0);
      \node[entry] at (0.5,-0.3) {\(\vdots\)};
      \node[entry] at (0.5,-1.5) {\(i\)};
      \node[entry] at (0.5,-2.3) {\(\vdots\)};
      \node[entry] at (0.5,-3.5) {\(i\)};
      \node[entry] at (0.5,-4.5) {\(i\)};
      \node[entry] at (0.5,-5.3) {\(\vdots\)};
      \node[entry] at (0.5,-6.5) {\(i\)};
      \node[entry] at (0.5,-7.3) {\(\vdots\)};
      \draw[decoration={brace,raise=0.1cm},decorate] (1,-1.1) -- node[right=0.2cm] {\(a_i-1\)} (1,-3.9);
      \draw[decoration={brace,raise=0.1cm},decorate] (1,-4.1) -- node[right=0.2cm] {\(a_{i+1}-1\)} (1,-6.9);
    \end{scope}
    \node at (3,-8) {\(T_{\bolda'}(w)\)};
    \end{scope}
    \node at (4.5,-9) {Case \ref{item:4}: \(v_{(w)} = v_1^{a_i} \otimes v_{1}^{a_{i+1}}\)};
  \end{tikzpicture}\qquad\qquad
  \begin{tikzpicture}[scale=0.6,entry/.style={font=\small}]
    \draw (0,0) rectangle (4,1);
    \draw (1.5,0) -- (1.5,1);
    \draw (2.5,0) -- (2.5,1);
    \node[entry] at (0.75,0.5) {\(\cdots\)};
    \node[entry] at (2,0.5) {\(i\)};
    \node[entry] at (3.25,0.5) {\(\cdots\)};
    \begin{scope}[yshift=-0.2cm]
      \draw (0,0) rectangle (1,-8);
      \draw (0,-1) -- ++(1,0);
      \draw (0,-2) -- ++(1,0);
      \draw (0,-3) -- ++(1,0);
      \draw (0,-4) -- ++(1,0);
      \draw (0,-5) -- ++(1,0);
      \draw (0,-6) -- ++(1,0);
      \draw (0,-7) -- ++(1,0);
      \node[entry] at (0.5,-0.3) {\(\vdots\)};
      \node[entry] at (0.5,-1.5) {\(i\)};
      \node[entry] at (0.5,-2.3) {\(\vdots\)};
      \node[entry] at (0.5,-3.5) {\(i\)};
      \node[entry] at (0.5,-4.5) {\(i\hspace{-0.6ex}+\hspace{-0.6ex}1\)}; 
      \node[entry] at (0.5,-5.3) {\(\vdots\)};
      \node[entry] at (0.5,-6.5) {\(i\hspace{-0.6ex}+\hspace{-0.6ex}1\)}; 
      \node[entry] at (0.5,-7.3) {\(\vdots\)};
      \draw[decoration={brace,raise=0.1cm},decorate] (1,-1.1) -- node[right=0.2cm] {\(a_i-1\)} (1,-3.9);
      \draw[decoration={brace,raise=0.1cm},decorate] (1,-4.1) -- node[right=0.2cm] {\(a_{i+1}\)} (1,-6.9);
    \end{scope}
    \node at (3,-8) {\(T_\bolda(w)\)};
  \begin{scope}[xshift=5cm] 
    \draw (0,0) rectangle (4,1);
    \draw (1.5,0) -- (1.5,1);
    \draw (2.5,0) -- (2.5,1);
    \node[entry] at (0.75,0.5) {\(\cdots\)};
    \node[entry] at (2,0.5) {\(i\)};
    \node[entry] at (3.25,0.5) {\(\cdots\)};
    \begin{scope}[yshift=-0.2cm]
      \draw (0,0) rectangle (1,-8);
      \draw (0,-1) -- ++(1,0);
      \draw (0,-2) -- ++(1,0);
      \draw (0,-3) -- ++(1,0);
      \draw (0,-4) -- ++(1,0);
      \draw (0,-5) -- ++(1,0);
      \draw (0,-6) -- ++(1,0);
      \draw (0,-7) -- ++(1,0);
      \node[entry] at (0.5,-0.3) {\(\vdots\)};
      \node[entry] at (0.5,-1.5) {\(i\)};
      \node[entry] at (0.5,-2.3) {\(\vdots\)};
      \node[entry] at (0.5,-3.5) {\(i\)};
      \node[entry] at (0.5,-4.5) {\(i\)};
      \node[entry] at (0.5,-5.3) {\(\vdots\)};
      \node[entry] at (0.5,-6.5) {\(i\)};
      \node[entry] at (0.5,-7.3) {\(\vdots\)};
      \draw[decoration={brace,raise=0.1cm},decorate] (1,-1.1) -- node[right=0.2cm] {\(a_i-1\)} (1,-3.9);
      \draw[decoration={brace,raise=0.1cm},decorate] (1,-4.1) -- node[right=0.2cm] {\(a_{i+1}\)} (1,-6.9);
    \end{scope}
    \node at (3,-8) {\(T_{\bolda'}(w)\)};
    \end{scope}
    \node at (4.5,-9) {Case \ref{item:5}: \(v_{(w)} = v_1^{a_i} \otimes v_{0}^{a_{i+1}}\)};
  \end{tikzpicture}\\ \vspace{0.7cm}
  \begin{tikzpicture}[scale=0.6,entry/.style={font=\small}]
    \draw (0,0) rectangle (4,1);
    \draw (1.5,0) -- (1.5,1);
    \draw (2.5,0) -- (2.5,1);
    \node[entry] at (0.75,0.5) {\(\cdots\)};
    \node[entry] at (2,0.5) {\(i\hspace{-0.6ex}+\hspace{-0.6ex}1\)};
    \node[entry] at (3.25,0.5) {\(\cdots\)};
    \begin{scope}[yshift=-0.2cm]
      \draw (0,0) rectangle (1,-8);
      \draw (0,-1) -- ++(1,0);
      \draw (0,-2) -- ++(1,0);
      \draw (0,-3) -- ++(1,0);
      \draw (0,-4) -- ++(1,0);
      \draw (0,-5) -- ++(1,0);
      \draw (0,-6) -- ++(1,0);
      \draw (0,-7) -- ++(1,0);
      \node[entry] at (0.5,-0.3) {\(\vdots\)};
      \node[entry] at (0.5,-1.5) {\(i\)};
      \node[entry] at (0.5,-2.3) {\(\vdots\)};
      \node[entry] at (0.5,-3.5) {\(i\)};
      \node[entry] at (0.5,-4.5) {\(i\hspace{-0.6ex}+\hspace{-0.6ex}1\)}; 
      \node[entry] at (0.5,-5.3) {\(\vdots\)};
      \node[entry] at (0.5,-6.5) {\(i\hspace{-0.6ex}+\hspace{-0.6ex}1\)}; 
      \node[entry] at (0.5,-7.3) {\(\vdots\)};
      \draw[decoration={brace,raise=0.1cm},decorate] (1,-1.1) -- node[right=0.2cm] {\(a_i\)} (1,-3.9);
      \draw[decoration={brace,raise=0.1cm},decorate] (1,-4.1) -- node[right=0.2cm] {\(a_{i+1}-1\)} (1,-6.9);
    \end{scope}
    \node at (3,-8) {\(T_\bolda(w)\)};
  \begin{scope}[xshift=5cm]
    \draw (0,0) rectangle (4,1);
    \draw (1.5,0) -- (1.5,1);
    \draw (2.5,0) -- (2.5,1);
    \node[entry] at (0.75,0.5) {\(\cdots\)};
    \node[entry] at (2,0.5) {\(i\)};
    \node[entry] at (3.25,0.5) {\(\cdots\)};
    \begin{scope}[yshift=-0.2cm]
      \draw (0,0) rectangle (1,-8);
      \draw (0,-1) -- ++(1,0);
      \draw (0,-2) -- ++(1,0);
      \draw (0,-3) -- ++(1,0);
      \draw (0,-4) -- ++(1,0);
      \draw (0,-5) -- ++(1,0);
      \draw (0,-6) -- ++(1,0);
      \draw (0,-7) -- ++(1,0);
      \node[entry] at (0.5,-0.3) {\(\vdots\)};
      \node[entry] at (0.5,-1.5) {\(i\)};
      \node[entry] at (0.5,-2.3) {\(\vdots\)};
      \node[entry] at (0.5,-3.5) {\(i\)};
      \node[entry] at (0.5,-4.5) {\(i\)};
      \node[entry] at (0.5,-5.3) {\(\vdots\)};
      \node[entry] at (0.5,-6.5) {\(i\)};
      \node[entry] at (0.5,-7.3) {\(\vdots\)};
      \draw[decoration={brace,raise=0.1cm},decorate] (1,-1.1) -- node[right=0.2cm] {\(a_i\)} (1,-3.9);
      \draw[decoration={brace,raise=0.1cm},decorate] (1,-4.1) -- node[right=0.2cm] {\(a_{i+1}-1\)} (1,-6.9);
    \end{scope}
    \node at (3,-8) {\(T_{\bolda'}(w)\)};
    \end{scope}
    \node at (4.5,-9) {Case \ref{item:6}: \(v_{(w)} = v_0^{a_i} \otimes v_{1}^{a_{i+1}}\)};
  \end{tikzpicture}\qquad\qquad
  \begin{tikzpicture}[scale=0.6,entry/.style={font=\small}]
    \draw (0,0) rectangle (4,1);
    \node[entry] at (2,0.5) {\(\cdots\)};
    \begin{scope}[yshift=-0.2cm]
      \draw (0,0) rectangle (1,-8);
      \draw (0,-1) -- ++(1,0);
      \draw (0,-2) -- ++(1,0);
      \draw (0,-3) -- ++(1,0);
      \draw (0,-4) -- ++(1,0);
      \draw (0,-5) -- ++(1,0);
      \draw (0,-6) -- ++(1,0);
      \draw (0,-7) -- ++(1,0);
      \node[entry] at (0.5,-0.3) {\(\vdots\)};
      \node[entry] at (0.5,-1.5) {\(i\)};
      \node[entry] at (0.5,-2.3) {\(\vdots\)};
      \node[entry] at (0.5,-3.5) {\(i\)};
      \node[entry] at (0.5,-4.5) {\(i\hspace{-0.6ex}+\hspace{-0.6ex}1\)}; 
      \node[entry] at (0.5,-5.3) {\(\vdots\)};
      \node[entry] at (0.5,-6.5) {\(i\hspace{-0.6ex}+\hspace{-0.6ex}1\)}; 
      \node[entry] at (0.5,-7.3) {\(\vdots\)};
      \draw[decoration={brace,raise=0.1cm},decorate] (1,-1.1) -- node[right=0.2cm] {\(a_i\)} (1,-3.9);
      \draw[decoration={brace,raise=0.1cm},decorate] (1,-4.1) -- node[right=0.2cm] {\(a_{i+1}\)} (1,-6.9);
    \end{scope}
    \node at (3,-8) {\(T_\bolda(w)\)};
  \begin{scope}[xshift=5cm]
    \draw (0,0) rectangle (4,1);
    \node[entry] at (2,0.5) {\(\cdots\)};
    \begin{scope}[yshift=-0.2cm]
      \draw (0,0) rectangle (1,-8);
      \draw (0,-1) -- ++(1,0);
      \draw (0,-2) -- ++(1,0);
      \draw (0,-3) -- ++(1,0);
      \draw (0,-4) -- ++(1,0);
      \draw (0,-5) -- ++(1,0);
      \draw (0,-6) -- ++(1,0);
      \draw (0,-7) -- ++(1,0);
      \node[entry] at (0.5,-0.3) {\(\vdots\)};
      \node[entry] at (0.5,-1.5) {\(i\)};
      \node[entry] at (0.5,-2.3) {\(\vdots\)};
      \node[entry] at (0.5,-3.5) {\(i\)};
      \node[entry] at (0.5,-4.5) {\(i\)};
      \node[entry] at (0.5,-5.3) {\(\vdots\)};
      \node[entry] at (0.5,-6.5) {\(i\)};
      \node[entry] at (0.5,-7.3) {\(\vdots\)};
      \draw[decoration={brace,raise=0.1cm},decorate] (1,-1.1) -- node[right=0.2cm] {\(a_i\)} (1,-3.9);
      \draw[decoration={brace,raise=0.1cm},decorate] (1,-4.1) -- node[right=0.2cm] {\(a_{i+1}\)} (1,-6.9);
    \end{scope}
    \node at (3,-8) {\(T_{\bolda'}(w)\)};
    \end{scope}
    \node at (4.5,-9) {Case \ref{item:7}: \(v_{(w)} = v_0^{a_i} \otimes v_{0}^{a_{i+1}}\)};
  \end{tikzpicture}
  \caption{Here are depicted the tableaux \(T_\bolda(w)\) and \(T_{\bolda'}(w)\) in each of the four cases of the proof of Theorem \ref{thm:1}.}
\label{fig:four-cases}
\end{figure}
In cases \ref{item:5} and \ref{item:6} the tableau \(T_{\bolda'}(w')\) has one entry \(i\) in the row, hence \(v_{(w')}=v^{a_{i}+a_{i+1}}_1\), while in case \ref{item:7} the tableau \(T_{\bolda'}(w')\) has all entries \(i\) in the column and hence \(v_{(w')}= v^{a_i+a_{i+1}}_0\). Hence in all four cases we have that \eqref{eq:184} holds up to a multiple,
and we only need to verify that the coefficients fit. For example in case \ref{item:5} comparing with \eqref{eq:37} we must check that
\begin{equation}
  \label{eq:114}
  q^{-a_{i+1}} q^{-(a_i-1)a_{i+1}} \frac{(v_{(w)},v_{(w)})_\bolda}{(v_{(w')},v_{(w')})_{\bolda'}} = q^{-a_{i+1}} \qbin{a_{i}+a_{i+1}-1}{a_{i+1}}.
\end{equation}
Using the formula \eqref{eq:99} for the bilinear form and the notation as in \eqref{eq:55}, we compute the l.h.s.\ of \eqref{eq:114}:
\begin{equation}
  \label{eq:115}
  q^{-a_{i+1}} q^{-(a_i-1) a_{i+1}} \frac{[\beta_1 + \cdots + \beta_\ell]_0!}{[\beta_1]_0! \cdots [\beta_\ell]_0!} \frac{[\beta_1']_0!\cdots [\beta_{\ell-1}']_0!}{[\beta'_1+\cdots+ \beta'_{\ell-1}]_0!},
\end{equation}
where if \(v_{(w)}=v^\bolda_\boldeta\) and \(v_{(w')}=v^\bolda_\boldgamma\) we set \(\beta_j=\beta^\boldeta_j\) and \(\beta'_j=\beta^\boldgamma_j\).
Substituting \(\beta'_j=\beta_j\) for \(j<i\), \(\beta'_j=\beta_{j+1}\) for \(j>i\), \(\beta'_i=a_i+a_{i+1}-1\), \(\beta_i=a_i-1\), \(\beta_i=a_i\) we get exactly the r.h.s.\ of \eqref{eq:114}. Similarly we can handle cases \ref{item:6} and \ref{item:7}.

Now let us consider \eqref{eq:185}. Let \(w' \in
\Lambda_k(\bolda')\), and consider the corresponding
tableau \(T=T_{\bolda'}(w')\). Suppose first that
\(v_{(w')}=v_{(T)}=v^{a_i+a_{i+1}}_1\): then \(T\) has exactly one entry \(i\)
in the row, and we can apply Lemma~\ref{lem:2} below. Note that the
tableaux \(T''\) and \(T'\) of Lemma~\ref{lem:2} correspond to \(v^{a_i}_1
\otimes v^{a_{i+1}}_0\) and \(v^{a_i}_0 \otimes v^{a_{i+1}}_1\)
respectively. Hence we just need to check that the coefficients are
the right ones. Let us start with the first term of the r.h.s.\ of
\eqref{eq:104}:
comparing \eqref{eq:104} with \eqref{eq:38}, using the isomorphism defined by \eqref{eq:30}, we must show that
\begin{equation}
  \label{eq:100}
  \qbin{a_i+a_{i+1}-1}{a_{i+1}}_0 \frac{(v_{(T)},v_{(T)})_{\bolda'}}{(v_{(T'')},v_{(T'')})_\bolda} = 1
\end{equation}
or equivalently
\begin{equation}
  \label{eq:103}
  \qbin{a_i+a_{i+1}-1}{a_{i+1}}_0 (v_{(T)},v_{(T)})_{\bolda'} = (v_{(T'')},v_{(T'')})_\bolda.
\end{equation}
Using the formula \eqref{eq:99} for the bilinear form and the notation as in \eqref{eq:55}, we compute the r.h.s.\ of \eqref{eq:103}:
\begin{equation}
  \label{eq:101}
  \qbin{a_i+a_{i+1}-1}{a_{i+1}}_0 \qbin{\beta'_1+\cdots+\beta'_{\ell-1}}{\beta'_1,\ldots,\beta'_{\ell-1}}_0 = \frac{[a_i+a_{i+1}-1]_0!}{[a_{i+1}]_0! [a_i-1]_0!} \frac{[\beta'_1+\cdots+\beta'_{\ell-1}]_0!}{[\beta'_1]_0!\cdots [\beta'_{\ell-1}]_0!},
\end{equation}
where as before if \(v_{(T)}=v^\bolda_\boldeta\) and \(v_{(T')}=v^\bolda_\boldgamma\) we set \(\beta'_j=\beta^\boldeta_j\) and \(\beta_j=\beta^\boldgamma_j\).
Since \(\beta_i'=a_i+a_{i+1}-1\), \(a_{i+1}=\beta_{i+1}\), \(a_i -1 = \beta_i\), \(\beta'_j = \beta_j\) for \(j<i\) and \(\beta'_j=\beta_{j+1}\) for \(j>i\) we see that \eqref{eq:101} is equal to
\begin{equation}
  \label{eq:102}
  \frac{[\beta_1+\cdots+\beta_{\ell}]_0!}{[\beta_1]_0!\cdots [\beta_{\ell}]_0!}
\end{equation}
and we are done. Analogously for the second term of the r.h.s.\ of \eqref{eq:104} we have that
\begin{equation}
  \label{eq:112}
  \qbin{a_i+a_{i+1}-1}{a_{i}}_0 (v_{(T)},v_{(T)})_{\bolda'} = (v_{(T')},v_{(T')})_\bolda.
\end{equation}

Now suppose instead that \(v_{(w')}=v_{(T)} = v^{a_i+a_{i+1}}_0\): then \(T\) has all entries \(i\) in the column, and we can apply Lemma~\ref{lem:13} below. The tableau \(T'\) of Lemma~\ref{lem:13} corresponds to \(v^{a_i}_0 \otimes v^{a_{i+1}}_0\), and we just need to check that
\begin{equation}
  \label{eq:113}
  \qbin{a_i+a_{i+1}}{a_{i}}_0 \frac{(v_{(T)},v_{(T)})_{\bolda'}}{(v_{(T')},v_{(T')})_\bolda} = 1,
\end{equation}
that follows as before.
\end{proof}

\begin{lemma}
  \label{lem:2}
  Let \(\bolda,\bolda'\) as in the proof of Theorem~\ref{thm:1}. Let \(T\) be an admissible tableau of type \(\bolda'\) with exactly one entry \(i\) in the row. Construct admissible tableaux \(T'\), \(T''\) of type \(\bolda\) as follows: first increase by \(1\) all entries of \(T\) greater than \(i\); then substitute the first \(a_{i+1}\) entries \(i\) with \(i+1\) (here first means, as always for our hook diagrams, that we first go through the column from the bottom to the top and then through the row from the left to the right). Call the result \(T'\). Moreover, let \(T''=x_0 \cdot T'\) where \(x_0\) is the longest element of \((\bbS_{\bolda'}/\bbS_\bolda)^{\short}\). Then we have
  \begin{equation}
    \label{eq:104}
    [\trasT_{\bolda'}^\bolda \overline \Delta( T)] = \qbin{a_i+a_{i+1}-1}{a_{i+1}}_0 [\overline \Delta( T'')] + q^{a_i} \qbin{a_i+a_{i+1}-1}{a_i}_0 [\overline \Delta(T')],
  \end{equation}
  where for an admissible tableau \(T_\bolda(w)\) we wrote \(\oDelta(T_\bolda(w))\) for \(\oDelta(w)\).
\end{lemma}

\begin{proof}
  We just need to  translate Proposition~\ref{prop:11} into the combinatorics of tableaux.  Let \(w \in
  \Lambda_k(\bolda')\) be such that
  \(T= T_{\bolda'}(w)\). Consider the sum on the r.h.s.\ of
  \eqref{eq:139}. First consider the set \(\{ T_\bolda(wy) \suchthat y \in (\bbS_{\bolda'}/\bbS_\bolda)^{\short}\}\): this
  consists of all tableaux obtained by permuting the entries \(i\) and
  \(i+1\) of \(T'\). Notice now that for all \(y \in (\bbS_{\bolda'}/\bbS_\bolda)^{\short} \) the tableau \( T_\bolda( x_y wy) \) is obtained from \( T_\bolda ( wx)\)
  permuting the entries \(i\) and \(i+1\) in the column so that it becomes
  admissible; in particular
  \(\len(x_y)+\len(w)+\len(y)=\len(x_y wy)\) and the set \(\{
  T_\bolda( x_y w y) \suchthat y \in
  (\bbS_{\bolda'}/\bbS_\bolda)^{\short}\}\) consists of the two tableaux
  \(T'\) and \(T''\).  Notice also that for each \(y \in (\bbS_{\bolda'}/\bbS_\bolda)^{\short}\) we have \(  x_y w y =
  wx'_y y\) for a unique \(x'_y \in \bbS_{\bolda'}\)
  with \(\len(x'_y)=\len(x_y)\); in particular 
  \(\len(x'_y)+\len(y)=\len(x'_y y)\). Let
 \begin{equation}
   \label{eq:105}
   \begin{aligned}
     \boldb' &=(a_1,\ldots,a_i+a_{i+1}-1,1,a_{i+2},\ldots,a_\ell),\\
     \boldb &=(a_1,\ldots,a_i,a_{i+1}-1,1,a_{i+2},\ldots,a_\ell).
   \end{aligned}
 \end{equation}
Then we have \(T'=T_\bolda (w y_0')\)  and \(T''=T_\bolda(w y_0)\) where  \(y'_0\) is the longest element of \((\bbS_{\boldb'}/\bbS_\boldb)^{\short}\) and \(y_0\) is the longest element of \((\bbS_{\bolda'}/\bbS_\bolda)^{\short}\). Now we can compute the two coefficients of \eqref{eq:104}; the second coefficient is
  \begin{multline}
  \label{eq:106}
    \sum_{\substack{y \in
        (\bbS_{\bolda'}/\bbS_\bolda)^{\short}\\x'_y y=y'_0}}
    q^{\len(y_0)-\len(y)+\len(x'_y)} = \sum_{\substack{y \in
        (\bbS_{\bolda'}/\bbS_\bolda)^{\short}\\x'_y y=y'_0}}
    q^{\len(y_0)-2\len(y)+\len(y'_0)} \\ \qquad= q^{\len(y_0)-\len(y'_0)}
    \sum_{y \in
        (\bbS_{\boldb'}/\bbS_\boldb)^{\short}}
    q^{2\len(y'_0)-2\len(y)} = q^{a_i} \qbin{a_i+a_{i+1}-1}{a_i}_0,
  \end{multline}
while the first coefficient is 
\begin{multline}
  \label{eq:107}
    \sum_{\substack{y \in
        (\bbS_{\bolda'}/\bbS_\bolda)^{\short}\\x'_y y =y_0}}
    q^{\len(y_0)-\len(y)+\len(x'_y)} = \sum_{\substack{y \in
        (\bbS_{\bolda'}/\bbS_\bolda)^{\short}\\x'_y y= y_0}}
    q^{2\len(y_0)-\len(y)} \\ \qquad= \sum_{z \in
      (\bbS_{a_i+a_{i+1}}/(\bbS_{a_i-1} \times
      \bbS_{a_{i+1}}))^{\short}} q^{2\len(z_0)-2\len(z)} =
    \qbin{a_i+a_{i+1}-1}{a_{i+1}}_0
\end{multline}
where we restricted to \(\bbS_{a_i+a_{i+1}}\) (since the permutations act trivially elsewhere) and we substituted \(y=zz'\) for \(z'=s_{a_i+a_{i+1}-1} \cdots s_{a_{i}+1} s_{a_i}\); the element \(z_0\) is the longest element of \((\bbS_{a_i+a_{i+1}}/(\bbS_{a_i-1} \times
      \bbS_{a_{i+1}}))^{\short}\).
\end{proof}

\begin{lemma}
  \label{lem:13}
  Let \(\bolda,\bolda'\) as in the proof of Theorem~\ref{thm:1}. Let \(T\) be an admissible tableau of type \(\bolda'\) with all entries equal to \(i\) in the column. Construct an admissible tableaux \(T'\) of type \(\bolda\) as follows: first increase by \(1\) all entries of \(T\) greater than \(i\); then substitute the first \(a_{i+1}\) entries \(i\) with \(i+1\) (here first means, as always for our hook diagrams, that we first go through the column from the bottom to the top and then through the row from the left to the right). Then we have
  \begin{equation}
    \label{eq:109}
    [\trasT_{\bolda'}^\bolda \overline \Delta( T)] = \qbin{a_i+a_{i+1}}{a_{i}}_0 [\overline \Delta( T')],
  \end{equation}
  where for an admissible tableau \(T_\bolda(w)\) we wrote \(\oDelta(T_\bolda(w))\) for \(\oDelta(w)\).
\end{lemma}

\begin{proof}
  The proof is similar to the previous one, but easier. We just need to compute
  \begin{equation}
    \label{eq:110}
\hspace{-0.7cm}    \sum_{ y \in (\bbS_{\bolda'}/\bbS_{\bolda})^{\short}} q^{\ell(y_0) - \ell(y) + \ell(x_y)} =   \hspace{-0.2cm}  \sum_{ x \in (\bbS_{\bolda'}/\bbS_{\bolda})^{\short}} q^{2\ell(y_0) - 2\ell(y)} = \qbin{a_i+a_{i+1}}{a_i}_0. \qedhere
  \end{equation}
\end{proof}

Let us consider in particular the case of the regular composition
\(\bolda = \compn\) of \(n\). 
 For every \(i=1,\ldots,n-1\) consider \(\hat \bolda_i\) as defined in \eqref{catO:eq:196}. Define \(\theta_i
= \trasT_{\hat\bolda_i}^\bolda \circ \trasT_{\bolda}^{\hat\bolda_i}\) as a functor \(\theta_i \colon \calQ(\compn) \mapto \calQ(\compn)\). As a consequence of Theorem~\ref{thm:1} we have:

\begin{corollary}
  \label{cor:1}
  The endofunctors \(\theta_i\)  on \(\calQ(\compn)\) categorify (i.e.\ give, at the level of the Grothendieck group) the action of the Super Temperley-Lieb Algebra \(\STL_n\) (see Definition \ref{def:6}).
\end{corollary}

\begin{subequations}
It follows by Lemma~\ref{lem:23} that the functors \(\theta_i\) satisfy the relations
      \begin{align}
        \theta_i^2 & \cong \theta_i \langle 1 \rangle \oplus \theta_{i} \langle -1 \rangle, \label{catO:eq:3}\\
        \theta_i \theta_j& \cong \theta_j \theta_i,\qquad \text{for }\abs{i-j}>1\label{catO:eq:94}\\
        \theta_i \theta_{i+1} \theta_i \oplus\theta_{i+1} & \cong \theta_{i+1} \theta_i \theta_{i+1} \oplus\theta_{i}.\label{catO:eq:97}
      \end{align}
In
      fact, these relations are the categorical versions of
      the relations of the Hecke algebra
      and are satisfied by
      the endofunctors \(\theta_i\) of \(\catOZ\). By
      Corollary~\ref{cor:1}, the relations
      (\ref{eq:134}-\ref{eq:135}) are satisfied in
      the Grothendieck group. We conjecture that their
      categorical versions are satisfied by the functors
      \(\theta_i\):

\begin{conjecture}\label{conj:1}
The functors \(\theta_i\) on \(\calQ(\compn)\) satisfy the relations
      \begin{gather}
        \begin{multlined}
          \theta_{i-1} \theta_{i+1} \theta_i \theta_{i-1} \theta_{i+1}
          \oplus [2]^2 \theta_{i-1} \theta_{i+1} \theta_i \\\qquad\cong [2] (
          \theta_{i-1} \theta_{i+1} \theta_i \theta_{i-1}
          \oplus\theta_{i-1} \theta_{i+1} \theta_i \theta_{i+1}),
        \end{multlined}\label{catO:eq:108}\\
\begin{multlined}
  \theta_{i-1} \theta_{i+1} \theta_i \theta_{i-1} \theta_{i+1} \oplus [2]^2
  \theta_i \theta_{i-1} \theta_{i+1} \\\qquad \cong [2]( \theta_{i-1} \theta_i
  \theta_{i-1} \theta_{i+1} \oplus\theta_{i+1} \theta_i \theta_{i-1}
  \theta_{i+1})
\end{multlined}\label{catO:eq:111}
      \end{gather}
for all \(i=2,\ldots,n-2\), where we used the abbreviations \([2]\theta_i= \theta_i \langle 1 \rangle \oplus \theta_i \langle -1 \rangle\) and \([2]^2 \theta_i = \theta_i \langle 2 \rangle \oplus \theta_i \oplus \theta_i \oplus \theta_i \langle -2 \rangle\).
\end{conjecture}
\end{subequations}

Although apparently harmful, we believe Conjecture \ref{conj:1} to be
quite hard. The difficulty is due to the lack of a classification of
projective functors on the parabolic category \(\catO^\frakp\) if
\(\frakp\) is not the Borel subalgebra \(\frakb\).

\subsection{Canonical basis}
\label{sec:canonical-basis}

Now we give a categorical interpretation of the canonical basis of
\(V(\bolda)\). First we restrict to consider the regular composition
\(\compn\).  Recall that by Proposition~\ref{prop:19} the canonical
basis of \((V^{\otimes n})_k\) can be interpreted as a canonical basis
for the Hecke algebra action. In this section we will use the Hecke
module structure of the Grothendieck groups of our categories.

Let \(\frakp,\frakq\subset \gl_n\) be the parabolic subalgebras defined at the beginning of the section, such that \(\calQ_k(\compn)=\catOZ^{\frakp,\frakq\pres}_0\). Using the notation introduced in Section~\ref{sec:hecke-algebra-hecke}, we fix isomorphisms
\begin{equation}
\label{eq:35}
  \begin{aligned}
    K^{\C(q)}(\catOZ_{\lambda}) & \mapto \ucalH_n \qquad&     K^{\C(q)}(\catOZ^\frakp_\lambda) & \mapto \ucalM^\frakp\\
    [M(w \cdot \lambda)] & \mapsto H_w &     [M^{\frakp}(w \cdot \lambda)] & \mapsto N_w.
  \end{aligned}
\end{equation}
As is well-known, by the Kazhdan-Lusztig conjecture projective modules
are sent to the canonical basis elements of \(\ucalH_n\) and
\(\ucalM^\frakp\) by the two isomorphisms.

Composing the isomorphism \eqref{eq:30} with the isomorphism \eqref{eq:90} we get an isomorphism
\begin{equation}
  \begin{aligned}
    K_0(\catOZ^{\frakp,\frakq\pres}_\lambda) & \mapto \ucalM^\frakp_\frakq\\
    [\Delta(w_\frakq w \cdot \lambda)] & \mapsto N_w
  \end{aligned}\label{eq:27}
\end{equation}
for \(w \in W^{\frakp+\frakq}\), where \(w_\frakq \in \bbS_k\) is the longest element.

\begin{lemma}
  \label{lem:12}
  The coapproximation functor \(\mathfrak Q: \catOZ^\frakp_\lambda \mapto \catOZ^{\frakp,\frakq\pres}_\lambda\) categorifies the map \(\sfQ: \ucalM^\frakp \mapto \ucalM^\frakp_\frakq\) (defined in \textsection\ref{sec:induc-hecke-modul}).
\end{lemma}

\begin{proof}
  Let \(w \in \Lambda^\frakp(\compn) = W^\frakp\). By Proposition~\ref{prop:9} we have \(\frakQ M^\frakp(w \cdot 0) = q^{\len(x)} \oDelta(xw \cdot 0)\) where \(x \in W_\frakq\) is given by Lemma~\ref{lem:10}. Now \([M^\frakp(w \cdot 0)] = N_w \in \ucalM^\frakp\) and \([\oDelta(xw \cdot 0)] = \frac{1}{[k]_0!} N_{w_\frakq x w} \in \ucalM^\frakp_\frakq\). On the other side, by definition \(\sfQ N_w = c_\frakq^{-1} q^{- \len(w_\frakq)+\len(x)} N_{w_\frakq x w}\). The claim follows since
  \begin{equation}
c_\frakq^{-1} q^{-\len(w_\frakq)+\len(x)} = \frac{1}{[k]! q^{\len(w_\frakq)}} q^{\len(x)} = \frac{1}{[k]_0!} q^{\len(x)}.\qedhere\label{eq:170}
\end{equation}
\end{proof}

\begin{lemma}
  \label{lem:11}
  Under the isomorphism \eqref{eq:30} we have \([Q(w_\frakq w)] \mapsto \underline N_w\) for all \(w \in W^{\frakp+\frakq}\).
\end{lemma}

\begin{proof}
  By Lemma~\ref{lemma:FunctionBetweenM1} and the discussion after it,
  it follows that \(\sfQ\) sends the canonical basis element
  \(\underline N_{w_\frakq w} \in \ucalM^\frakp\) to \( \underline N_{w} \in \ucalM^\frakp_\frakq\).
  By Lemma~\ref{lem:12} we have
  \begin{equation}
[Q(w_\frakq w)] = [\frakQ P^{\frakp}(w_\frakq w \cdot 0)] = \sfQ [P^\frakp(w_\frakq w \cdot 0)] = \sfQ \underline N_{w_\frakq w} = \underline N_{w}.\qedhere\label{eq:171}
\end{equation}
\end{proof}

Now let us consider a general composition \(\bolda\).

\begin{prop}
  \label{prop:13}
  Under the isomorphism \eqref{eq:30} the class of the indecomposable projective module \(Q(w)\) maps to the canonical basis element \(v_{(w)}^\canon \in V(\bolda)\) corresponding to the standard basis element \(v_{(w)}\).
\end{prop}

\begin{proof}
  By Lemma~\ref{lem:11} we know the result for the regular composition
  \(\compn\).  Consider the standard inclusion 
  \(V(\bolda) \mapto V^{\otimes n}\) given by the web
  diagram \(\phi= \websplitmultiple_{a_1} \otimes \dotsb \otimes \websplitmultiple_{a_\ell}\), see \eqref{eq:142}. We know that the functor \(\funcF(\phi)\colon\calQ_k(\bolda)
  \mapto \calQ_k(\compn)\), which categorifies \(\phi\), sends
  indecomposable projective modules to indecomposable
  projective modules (Proposition~\ref{prop:7}). On the
  other side, it follows immediately from our diagrammatic
  calculus that what \(\phi\) sends to a canonical basis
  element is a canonical basis element (cf.\ Remark~\ref{rem:1}).
\end{proof}

\subsection{The bilinear form}
\label{sec:form}

We give now a categorical interpretation of the bilinear form \eqref{eq:53}. Given a \(\Z\)--graded complex vector space \(M= \bigoplus_{i \in \Z} M_i\), let \(h(M)= \sum_{i \in \Z} (\dim_\C M_i) q^i \in \Z[q,q^{-1}]\) be its graded dimension.
Now let \(M,N\) be objects of \(\calQ_k(\bolda)\). Set
\begin{equation}\label{eq:138}
h(\Ext (M,N)) = \sum_{j \in \Z} (-1)^j h(\Ext^j(M,N)).
\end{equation}
Let also \(\overline{\phantom{x}}\) be the involution of \(\Z[q,q^{-1}]\) given by \(\overline{q}=q^{-1}\).

Fix now a composition \(\bolda=(a_1,\dotsc,a_\ell)\) of \(n\) and an integer \(n - \ell \leq k \leq n\).
\begin{prop}
  \label{prop:22}
 For \(M,N \in \calQ_k(\bolda)\) we have
  \begin{equation}
    \label{eq:64}
    \overline {h(\Ext(M,N^*))} = ([M],[N])_\bolda,
  \end{equation}
  where \(N^*\) is the dual of \(N\) in \(\calQ_k(\bolda)\) (see Lemma~\ref{lem:25}).
\end{prop}

\begin{proof}
  First, note that the l.h.s.\ of \eqref{eq:64} defines a bilinear form on the Grothendieck group. Hence we only need to prove that the two sides coincide on a basis.

By the properties of properly stratified algebras (cf.\ \cite[Lemma~4]{MR2344576}) we have
  \begin{equation}
    \label{eq:65}
    \Ext^i\big(\Delta(z),(\oDelta (w))^*\big) =
    \begin{cases}
      \C & \text{if } z=w \text{ and } i=0,\\
      0 & \text{otherwise}.
    \end{cases}
  \end{equation}
Hence we are left to prove that
\begin{equation}
\label{eq:51}
  \frac{( [\Delta(z)], v_{(w)} )_\bolda}{(v_{(w)},v_{(w)})_\bolda} = \delta_{z,w} \qquad \text{for all } w,z \in \Lambda_k(\bolda)
\end{equation}
or equivalently that
\begin{equation}
  \label{eq:69}
  [\Delta(z)] =  v_{(z)} = (v_{(z)},v_{(z)})_\bolda [\overline \Delta(z)] \qquad \text{for all } z \in \Lambda_k(\bolda).
\end{equation}
By the properties of a properly stratified algebra, it suffices for
that to prove that the proper standard module \(\overline \Delta(z
)\) appears \((v_{(z)},v_{(z)})\)--times in some proper
standard filtration of the indecomposable projective \(P(z)\). Since by \eqref{eq:30} and by Proposition~\ref{prop:13} we know which basis the proper standard and the
indecomposable projective modules categorify, this follows.
\end{proof}

By Proposition~\ref{prop:22}, and since
  \begin{equation}
    \label{eq:204}
    \Ext^i(\Delta(z),(\oDelta (w))^*) =
    \Ext^i(Q(z),(S (w))^*) =
    \begin{cases}
      \C & \text{if } z=w \text{ and } i=0,\\
      0 & \text{otherwise},
    \end{cases}
  \end{equation}
we have:

\begin{theorem}
  \label{thm:5}
  Under the isomorphism \eqref{eq:30} we have the following correspondences:
  \begin{align*}
    \{\text{standard modules}\} & \longleftrightarrow  \text{standard basis,}\\
    \{\text{proper standard modules}\} & \longleftrightarrow  \text{dual standard basis,}\\
    \{\text{indecomposable projective modules}\} & \longleftrightarrow  \text{canonical basis,}\\
    \{\text{simple modules}\} & \longleftrightarrow  \text{dual canonical basis.}
  \end{align*}
\end{theorem}

Note that an analog of this theorem for tensor products of representations of \(U_q(\frakg)\) was established in \cite{MR2305608} for \(\frakg= \mathfrak{sl}_2\) and in \cite{2013arXiv1309.3796W} for  a general semisimple Lie algebra \(\frakg\).

We conclude with an example of how the bilinear form can be used to compute combinatorially dimensions of homomorphism spaces.

\begin{lemma}
  \label{lem:17}
  Let \(w,z \in \Lambda_k(\compn)\). Then the dimension of \(\Hom(Q_{\compn,k}(w),Q_{\compn,k}(z))\) is \(k!\) times the number of elements \(x \in \Lambda_k(\compn)\) such that both the canonical basis diagrams \(C(v^\canon_{(w)})\) and \(C(v^\canon_{(z)})\) have nonzero value when labeled at the top with the standard basis diagram \(v_{(x)}\) (the evaluation is computed according to the rules in Figure~\ref{reps:fig:evaluations}).
\end{lemma}

\begin{proof}
Since the modules \(Q(w)\), \(Q(z)\) are projective, we can compute the dimension of \(\Hom(Q(w) , Q(z))\) using Proposition~\ref{prop:22}:
\begin{equation}
  \label{eq:143}
  \dim_\C \Hom(Q(w),Q(z)) = \left([Q(w)],[Q( z)^*]\right)^{q=1}_\compn,
\end{equation}
where \((\cdot,\cdot)_\compn^{q=1}\) is the form \((\cdot,\cdot)_\compn\) evaluated at \(q=1\). By the orthogonality of the standard basis elements \(v_{(w)}\) for \(w \in \Lambda_k(\compn)\) we can write
\begin{equation}
  \label{eq:145}
  \left([Q(w)],[Q(z)^*]\right)_\compn = \frac{1}{[k]!} \sum_{x \in \Lambda_k(\compn)} \left([Q(w)],v_{(x)}\right)_\compn \left(v_{(x)} , [Q(z)^*]\right)_\compn.
\end{equation}
Since the duality \(*\) is simple-preserving, and since \((qM)^* \cong q^{-1} M^*\) for all \(M\) (cf.\ Remark~\ref{rem:14}) it follows that the expansions of \([Q(z)^*]\) and \([Q(z)]\) in the basis of the equivalence classes of simple modules can be obtained from one another by replacing  \(q\) with \(q^{-1}\). Hence specializing at \(q=1\) we can  write 
\begin{equation}  \label{eq:148}
\begin{aligned}
  \left([Q(w)],[Q(z)^*]\right)^{q=1}_\compn& = \frac{1}{k!} \sum_{x \in \Lambda_k(\compn)} \left([Q(w)],v_{(x)}\right)^{q=1}_\compn \left(v_{(x)} , [Q(z)]\right)^{q=1}_\compn\\
  &= \frac{1}{k!} \sum_{x \in \Lambda_k(\compn)} \left(v_{(w)}^\canon,v_{(x)}\right)^{q=1}_\compn \left(v_{(z)}^\canon ,v_{(x)}\right)^{q=1}_\compn.
\end{aligned}
\end{equation}

Let \(C(v_{(w)})\) be the canonical basis diagram corresponding to \(v_{(w)}\), and let \(\boldeta\) be the \(\up\down\)--sequence corresponding to the standard basis element \(v_{(x)}\) (or equivalently to the permutation \(x\)). Let also \(\ucalD_x\) be the diagram obtained by labeling \(C(v_{(w)})\) at the top by \(\boldeta\).
By the definition of the bilinear form, \(\big(v^\canon_{(w)},v_{(x)}\big)_\compn\) is equal to \([k]!\) times the evaluation of \(\ucalD_x\).
In the evaluation of \(\ucalD_x\), one only sees appearing (locally) the diagrams in the first row of Figure~\ref{reps:fig:evaluations} (or diagrams evaluating to zero). Therefore the evaluation of \(\ucalD_x\) is a monomial in \(q\) if \(\ucalD_x\) is oriented, and zero otherwise. Hence the claim follows.
\end{proof}

\subsection{Categorification of the action of \texorpdfstring{$\Uqgl$}{Uq(gl(1|1))}}
\label{sec:functors-scr-e}
We want now to define functors that categorify the action of \(\Uqgl\). As happens in the case of \(\mathfrak{sl}_2\), we are not able to categorify both the action of the intertwiners and the action of \(\Uqgl\) via exact functors; hence we will need to consider the derived categories.

\subsubsection{Functors \texorpdfstring{$\calE$}{E} and \texorpdfstring{$\calF$}{F}}
\label{sec:functors-cale-calf}
Fix an integer \(n\), a composition \(\bolda=(a_1,\ldots,a_\ell)\) of \(n\) and an integer \(n-\ell \leq k < n\). Let \(\lambda=\lambda_\bolda\), and let \(\frakp,\frakq,\frakp',\frakq'\) be the parabolic subalgebras of \(\gl_n\) such that \(\calQ_k(\bolda) = \catOZ^{\frakp,\frakq\pres}_\lambda\) and \(\calQ_{k+1}(\bolda) = \catOZ^{\frakp',\frakq'\pres}_\lambda\). Notice that \(\frakp'\subseteq \frakp\) and \(\frakq \subseteq \frakq'\). We have a diagram
\begin{equation}\label{eq:77}
\begin{tikzpicture}[baseline=(current bounding box.center)]
  \matrix (m) [matrix of math nodes, row sep=3em, column
  sep=2.5em, text height=1.5ex, text depth=0.25ex] {
    &  \catOZ^{\frakp',\frakq\pres}_\lambda &  \\
    \calQ_k(\bolda) &  & \calQ_{k+1}(\bolda)\\};
  \path[->] (m-2-1) edge[bend left=10] node[auto] {\( \frakj\)} (m-1-2);
  \path[->] (m-1-2) edge[bend left=10] node[auto] {\( \frakz\)} (m-2-1);
  \path[->] (m-2-3) edge[bend left=10] node[auto] {\( \fraki\)} (m-1-2);
  \path[->] (m-1-2) edge[bend left=10] node[auto] {\( \frakQ\)} (m-2-3);
\end{tikzpicture}
\end{equation}
Let us define \(\calE_k= \frakQ \circ \frakj\) and \(\calF_k = \frakz \circ \fraki\). We get then a pair of adjoint functors \(\calF_k \adjunction \calE_k\):
\begin{equation}\label{eq:76}
\begin{tikzpicture}[baseline=(current bounding box.center)]
   \node (A) at (0,0) {\(\calQ_k(\bolda)\)};
   \node (B) at (3,0) {\(\calQ_{k+1}(\bolda)\)};
  \path[->] (A) edge[bend left=10] node[auto] {\(\calE_k\)} (B);
  \path[->] (B) edge[bend left=10] node[auto] {\(\calF_k\)} (A);
\end{tikzpicture}
\end{equation}

We remark that by Lemma~\ref{lem:21} the functors \(\calE_k,\calF_k\) commute with translation functors.
We can compute explicitly the action of \(\calF_k\) on projective modules and of \(\calE_k\) on simple modules:

\begin{prop}
  \label{prop:3}
  For \(w \in \Lambda_{k+1}(\bolda)\) we have
  \begin{equation}
    \label{eq:6}
    \calF_k Q_{\bolda,k+1}(w) =
    \begin{cases}
      Q_{\bolda,k}(w) & \text{if } w \in \Lambda_k(\bolda),\\
      0 & \text{otherwise}.
    \end{cases}
  \end{equation}
\end{prop}

\begin{proof}
  Consider the diagram \eqref{eq:77}. Of course \(\Lambda_k(\bolda) =
  \Lambda^\frakp_\frakq(\lambda) \subseteq
  \Lambda^{\frakp'}_{\frakq}(\lambda)\), and we have \(\fraki Q_{\bolda,k+1}(w) = P^{\frakp'}(w \cdot \lambda) \in
  \catOZ^{\frakp',\frakq\pres}_\lambda\). By the definition of the Zuckermann's functor we have then \(\frakz P^{\frakp'}(w \cdot \lambda) = P^{\frakp}(w \cdot \lambda) = Q_{\bolda,k}(w) \in \calQ_k(\bolda)\) if \(w \in \Lambda^{\frakp}_{\frakq}(\lambda)\), or \(0\) otherwise.
\end{proof}

\begin{prop}
  \label{prop:16}
  For \(w \in \Lambda_k(\bolda)\) we have
  \begin{equation}
    \label{eq:82}
    \calE_k S_{\bolda,k}(w) =
    \begin{cases}
      S_{\bolda,k+1}(w) & \text{if } w \in \Lambda_{k+1}(\bolda),\\
      0 & \text{otherwise}.
    \end{cases}
  \end{equation}
\end{prop}

\begin{proof}
  Consider the diagram \eqref{eq:77}. By Lemma~\ref{lem:15},
  the simple objects of \(\calQ_k(\bolda)\) are the simple
  objects \(S(w \cdot \lambda)\) of
  \(\catOZ^{\frakp',\frakq\pres}_\lambda\) such that \(w \in
  \Lambda_k(\bolda)\). In particular, \(\frakj
  S_{\bolda,k}(w) = S(w \cdot \lambda)\) for each \(w \in
  \Lambda_k(\bolda)\). Let
  \(\frakQ_{\frakq'}:\catOZ^{\frakp'}_\lambda \mapto
  \catOZ^{\frakp',\frakq'\pres}_\lambda\) and
  \(\frakQ_\frakq: \catOZ^{\frakp'}_\lambda \mapto
  \catOZ^{\frakp',\frakq\pres}_\lambda\) be the
  corresponding coapproximation functors. As we already
  noticed, it follows from the definition that
  \(\frakQ_{\frakq'} = \frakQ \circ \frakQ_\frakq\). Since
  \(S(w \cdot \lambda) = \frakQ_{\frakq} L(w \cdot
  \lambda)\), we have \(\frakQ S(w \cdot \lambda) =
  \frakQ_{\frakq'} L(w \cdot \lambda)\). This is
  \(S_{\bolda,k+1}(w) \in \calQ_{k+1}(\bolda)\) if \(w \in
  \Lambda_{k+1}(\bolda)\), or \(0\) otherwise.
\end{proof}

At least in the regular case \(\bolda=\compn\), using the explicit description of \cite{2013arXiv1311.6968S} one can prove that the functors \(\calF_k\) and \(\calE_k\) are indecomposable:

\begin{theorem}[{\cite[Theorem~6.10]{2013arXiv1311.6968S}}]
  \label{thm:7}
  For the functors \(\calF_k : \calQ_{k+1}(\compn) \mapto \calQ_k(\compn)\) and \(\calE_k: \calQ_k(\compn) \mapto \calQ_{k+1}(\compn)\) we have 
\(\End (\calE_k)\cong \End(\calF_k) \cong \C[x_1,\ldots,x_n]/I_k\),
where \(I_k\) is the ideal generated by the complete symmetric functions
  \begin{equation}
    \label{eq:146}
    \begin{aligned}
      h_{k+1}(x_{i_1},\ldots,x_{i_m}) &\quad  \text{for all } 1  \leq m \leq n-k \text{ and } i_1 < \dotsb < i_m,\\
      h_{n-m+1}(x_{i_1},\ldots,x_{i_m}) &\quad  \text{for all } n-k+1  \leq m \leq n \text{ and } i_1 < \dotsb < i_m,
    \end{aligned}
  \end{equation}
  In particular, \(\calE_k\) and \(\calF_k\) are indecomposable functors.
\end{theorem}

\subsubsection{Unbounded derived categories}
\label{sec:unbo-deriv-categ}
Being the composition of exact functors, the functor \(\calE_k\) is
exact. On the other side, being the composition of right-exact
functors, \(\calF_k\) is right exact, but not exact in general. Therefore, \(\calF_k\) does not induce a map between the Grothendieck groups, unless we pass to the derived category.
Unfortunately, properly stratified algebras do not
have, in general, finite global dimension (this happens if and only if
they are quasi-hereditary). Hence, we shall consider unbounded derived
categories. The main problem with unbounded derived categories is that their Grothendieck group is trivial (see \cite{MR2223266}). A workaround to this problem has been developed by Achar and Stroppel in \cite{pre06137854}. We recall briefly their main definitions and results, adapted to our setting.

Consider a finite-dimensional positively graded \(\C\)--algebra \(A = \bigoplus_{i \leq 0} A_i\) with semisimple \(A_0\), and let \(\calA = \gmod{A}\). Each simple object of \(\calA\) is concentrated in one degree. Achar and Stroppel define a full subcategory \(\der^\nabla \mathcal A\) of the unbounded derived category \(\der^- \mathcal A\) by
\begin{equation}
  \label{eq:28}
  \der^\nabla \calA = \left\{ X \in \der^- \calA \, \left| \, \mbox{\parbox{21em}{\centering for each \(m \in \Z\) only finitely many of the \(H^i(X)\)\\ contain a composition factor of degree \(<m\)}} \right. \right\}.
\end{equation}

Recall that the Grothendieck group \(K(\calT)\) of a small triangulated
category \(\calT\) is defined to be the free abelian group on
isomorphism classes \([X]\) for \(X \in \calT\) modulo the relation
\([B]=[A]+[C]\) whenever there is a distinguished triangle of the form
\(A \mapto B \mapto C \mapto A[1]\). As for abelian categories, if \(\calT\) is
graded then \(K(\calT)\) is naturally a \(\Z[q,q^{-1}]\)--module. Let
\begin{equation}
  \label{eq:46}
  I=\{x \in \der^\nabla(\calA) \suchthat [\beta_{\leq m}]x=0 \text{ in } K(\der^\nabla(\calA)) \text{ for all } m \in \Z\},
\end{equation}
where \(\beta_{\leq m}: D^\nabla \calA \mapto D^\nabla \calA\) is induced by the exact functor \(\beta_{\leq m}: \calA \mapto \calA\) defined on the graded module \(M=\bigoplus_{i \in \Z} M_i\) by \(\beta_{\leq m} M = \bigoplus_{i \leq m} M_i\). Then \(\bK(\der^\nabla \calA) = K(\der^\nabla \calA)/I\) is the \emph{topological Grothendieck group} of \(\der^\nabla \calA\). The names is motivated by the fact that one can define on \(\bK(\der^\nabla \calA)\) a \((q)\)--adic topology with respect to which \(\bK(\der^\nabla \calA)\) is complete. It follows that \(\bK(\der^\nabla \calA)\) is a \(\Z[[q]][q^{-1}]\)--module.

On the other side, let \(\hat K(\calA)\) be the completion of the \(\Z[q,q^{-1}]\)--module \(K(\calA)\) with respect to the \((q)\)--adic topology (\cite[\textsection 2.3]{pre06137854}). Then the natural map \(K(\calA) \mapto \bK(\der^\nabla\calA)\) is injective and induces an isomorphism of \(\Z[[q]][q^{-1}]\)--modules
\begin{equation}
  \label{eq:80}
  \hat K(\mathcal A) \cong \bK (\der^\nabla \mathcal A).
\end{equation}
Moreover, if \(\{L_i \suchthat i \in I\}\), with \(I\) finite, is a full set of pairwise non-isomorphic simple objects of \(\calA\) concentrated in degree \(0\) and \(P_i\) is the projective cover of \(L_i\), then both \(\{L_i \suchthat i \in I\}\) and \(\{P_i \suchthat i \in I\}\) give a \(\Z[[q]][q^{-1}]\)--basis for \(\hat K(\calA)\). In particular \(\hat K(\calA) \cong \Z[[q]][q^{-1}] \otimes_{\Z[q,q^{-1}]} K(\calA)\).

In our setting, we have for each category \(\catOZ^{\frakp,\frakq\pres}_\lambda\) naturally
\begin{equation}
  \label{eq:81}
    K^{\C(q)}(\catOZ^{\frakp,\frakq\pres}_\lambda)
    \cong \C(q) \otimes_{\Z[[q]][q^{-1}]} \hat K
    (\catOZ^{\frakp,\frakq\pres}_\lambda).
\end{equation}
In particular, the same holds  for \(\calQ_k(\bolda)\). We define also
\begin{equation}
\bK^{\C(q)} (\der^\nabla \calA) = \C(q) \otimes_{\Z[[q]][q^{-1}]} \bK (\der^\nabla \calA).\label{eq:186}
\end{equation}

Let \(\calA_{\geq m}\) be the full subcategory of \(\calA\) consisting of objects \(M=\bigoplus_{i\geq m} M_i\). An additive functor \(G: \calA \mapto \calA'\) is said to be of \emph{finite degree amplitude} if there exists some \(\alpha \geq 0\) such that \(G(\calA_{\geq m}) \subset \calA'_{\geq m -\alpha}\) for all \(m \in \Z\). Let \(G : \calA \mapto \calA'\) be a right-exact functor that commutes with the degree shift. If \(G\) has finite degree amplitude, then the left-derived functor \(\derL G\) induces a continuous homomorphism of \(\Z[[q]][q^{-1}]\)--modules \([\derL G]:\hat K(\calA) \mapto \hat K (\calA')\).

\subsubsection{Derived functors \texorpdfstring{$\calE$}{E} and \texorpdfstring{$\calF$}{F}}
\label{sec:deriv-funct-cale}
Let us now go back to our functors \(\calE_k\) and \(\calF_k\). Being exact, \(\calE_k\) induces a functor \(\calE_k: \der^\nabla(\calQ_k(\bolda)) \mapto \der^\nabla (\calQ_{k+1}(\bolda))\). On the other side, it is immediate to check that the functors \(\fraki\) and \(\frakz\), and therefore also \(\calF_k\), have finite degree amplitude. Hence \(\derL \calF_k\) restricts to a functor \(\derL \calF_k: \der^\nabla(\calQ_{k+1}(\bolda)) \mapto \der^\nabla(\calQ_k(\bolda))\).  Since \(\calE_k\) is exact, it follows by standard arguments that we have a pair of adjoint functors \(\derL\calF_k \adjunction E_k\):
\begin{equation}
\begin{tikzpicture}[baseline=(current bounding box.center)]
    \node (A) at (0,0) {\(\der^\nabla \calQ_k(\bolda)\)};
    \node (B) at (5,0) {\(\der^\nabla \calQ_{k+1}(\bolda)\)};
   \path[->] (A) edge[bend left=10] node[auto] {\(\calE_k\)} (B);
   \path[->] (B) edge[bend left=10] node[auto] {\(\derL\calF_k\)} (A);
 \end{tikzpicture}\label{eq:93}
\end{equation}

\begin{remark}
  \label{rem:5}
  Since \(\fraki\) sends projective modules to projective modules, it follows from \cite[Corollary~10.8.3]{MR1269324} that \(\derL \calF_k = \derL \frakz \circ \derL \fraki\).
\end{remark}

\begin{theorem}
  \label{thm:2}
  The functors \(\derL\calF_k\) and \(\calE_k\) categorify \(F\) and \(E'\) respectively, that is, the following diagrams commute:
\begin{equation*}
  \begin{gathered}
    \begin{tikzpicture}[baseline=(current bounding box.center),yscale=0.9]
      \matrix (m) [matrix of math nodes, row sep=3em, column
      sep=3.5em, text height=1.5ex, text depth=0.25ex] {
        \der^\nabla \calQ_k(\bolda) & \der^\nabla\calQ_{k+1}(\bolda)\\
        V(\bolda)_k & V(\bolda)_{k+1}\\};
      \path[<-] (m-1-1) edge node[auto] {\( \derL\calF_k \)} (m-1-2);
      \path[->] (m-1-1) edge node[auto] {\( \bK^{\C(q)}\)} (m-2-1);
      \path[->] (m-1-2) edge node[auto] {\( \bK^{\C(q)}\)} (m-2-2);
      \path[<-] (m-2-1) edge node[auto] {\( F\)} (m-2-2);
    \end{tikzpicture} \qquad
    \begin{tikzpicture}[baseline=(current bounding box.center),yscale=0.9]
      \matrix (m) [matrix of math nodes, row sep=3em, column
      sep=3.5em, text height=1.5ex, text depth=0.25ex] {
        \der^\nabla \calQ_k(\bolda) & \der^\nabla\calQ_{k+1}(\bolda)\\
        V(\bolda)_k & V(\bolda)_{k+1}\\};
      \path[->] (m-1-1) edge node[auto] {\( \calE_k \)} (m-1-2);
      \path[->] (m-1-1) edge node[auto] {\( \bK^{\C(q)}\)} (m-2-1);
      \path[->] (m-1-2) edge node[auto] {\( \bK^{\C(q)}\)} (m-2-2);
      \path[->] (m-2-1) edge node[auto] {\( E'\)} (m-2-2);
    \end{tikzpicture}
  \end{gathered}
\end{equation*}
\end{theorem}

\begin{proof}
  We use Proposition~\ref{prop:3} to check that the first diagram commutes on the basis given by indecomposable projective modules.
  Let \(w \in \Lambda_{k+1}(\bolda)\) and
  write \(v_{(w)}=v^\bolda_\boldeta\). Then in
  \(\bK^{\C(q)}(\der^\nabla \calQ_{k+1}(\bolda))\) we have \([Q_{\bolda,k+1}(w \cdot
  \lambda)] = v^{\canon \bolda}_\boldeta\). Now \(w\) is in
  \(\Lambda_k(\bolda)\) if and only if it is a shortest coset
  representative for \(W_{\frakp} \backslash \bbS_n\).
  Let \(T^{k+1}_\bolda(w)\) (respectively, \(T^k_\bolda(w)\)) be the \((n-k-1,k+1)\)--tableau (respectively, \((n-k,k)\)--tableau) of type \(\bolda\) corresponding to \(w\). Obviously \(T^k_\bolda(w)\) can be obtained from \(T^{k+1}_\bolda(w)\) by removing the upper box \(\sfb\) of the column and adding it to the row in the leftmost position. Clearly \(T^k_\bolda(w)\) is admissible if and only if the entry of this box \(\sfb\) is \(1\). Hence \(w \in \Lambda_k(\bolda)\) if and only if \(\eta_1=0\), and in this case we have \([Q_{\bolda,k}(w)]= v^{a_1}_1 \canon v^{a_2}_{\eta_2} \canon \cdots \canon v^{a_\ell}_{\eta\ell}\) in \(\bK^{\C(q)}(\der^\nabla \calQ_k(\bolda))\). By Proposition
  \ref{prop:15}, this is the action of \(F\).

  Since \(\calE_k\) is the adjoint functor of \(\derL\calF_k\), the commutativity
  of the second diagram follows from the adjunction \eqref{eq:75} and Proposition~\ref{prop:22} (of course we could also argue as for \(\derL\calF_k\) and check directly the commutativity of the second diagram above using Proposition~\ref{prop:16}).
\end{proof}

We define \(\calE = \bigoplus_{k=n-\ell}^{n-1} \calE_k\) and \(\calF = \bigoplus_{k=n-\ell}^{n-1} \calF_k\) as endofunctors of \(\calQ(\bolda)\).
We have the following categorical version of the relation \(E^2=F^2=0\):

\begin{prop}
  \label{lem:20}
  The functors \(\calE\) and \(\calF\) satisfy \(\calE \circ \calE = \calF \circ \calF = 0\).
\end{prop}

\begin{proof}
  Let \(S \in \calQ(\bolda)\) be  simple. It follows from Proposition~\ref{prop:16} that \(\calE^2 S =0\). Since \(\calE\) is exact, this implies that \(\calE^2 = 0\).
  On the other side, it follows from Proposition~\ref{prop:3} that \(\calF^2\) is zero on projective modules. Since \(\calF\) is right exact and any object of \(\calQ(\bolda)\) has a projective presentation, it follows that \(\calF^2\) is the zero functor.
\end{proof}

Since \(\calL\) sends projective modules to projective modules, it follows (cf.\ \cite[Corollary 10.8.3]{MR1269324}) that \(\derL \calF \circ \derL \calF = \derL( \calF \circ \calF)=0\).

We summarize the results of this section in the following:
\begin{theorem}
  \label{thm:6}
Let \(\phi\) be a web defining a morphism \(V(\bolda) \mapto V(\bolda')\). Then the diagram
  \begin{equation}
    \label{eq:211}
    \begin{tikzpicture}[baseline=(current bounding box.center),yscale=0.9]
      \matrix (m) [matrix of math nodes, row sep=3em, column
      sep=3.5em, text height=1.5ex, text depth=0.25ex] {
        \der^\nabla \calQ(\bolda') & \der^\nabla\calQ(\bolda')\\
        \der^\nabla \calQ(\bolda) & \der^\nabla \calQ(\bolda)\\};
      \path[->] (m-1-1) edge node[auto] {\( \calE,\derL\calF \)} (m-1-2);
      \path[<-] (m-1-1) edge node[left] {\( \funcF(\phi)\)} (m-2-1);
      \path[<-] (m-1-2) edge node[auto] {\( \funcF(\phi)\)} (m-2-2);
      \path[->] (m-2-1) edge node[auto] {\( \calE,\derL\calF\)} (m-2-2);
    \end{tikzpicture}     
  \end{equation}
  commutes and categorifies (i.e.\ gives, after applying the completed Grothendieck group \(\bK^{\C(q)}\)) the diagram
  \begin{equation}
    \label{eq:212}
    \begin{tikzpicture}[baseline=(current bounding box.center),yscale=0.9]
      \matrix (m) [matrix of math nodes, row sep=3em, column
      sep=3.5em, text height=1.5ex, text depth=0.25ex] {
        V(\bolda') & V(\bolda')\\
        V(\bolda) & V(\bolda)\\};
      \path[->] (m-1-1) edge node[auto] {\( E',F \)} (m-1-2);
      \path[<-] (m-1-1) edge node[left] {\( \funcT(\phi)\)} (m-2-1);
      \path[<-] (m-1-2) edge node[auto] {\( \funcT(\phi)\)} (m-2-2);
      \path[->] (m-2-1) edge node[auto] {\( E',F\)} (m-2-2);
    \end{tikzpicture}
  \end{equation}
  In particular, for \(\bolda = \compn\) we have two families of endofunctors \(\{\calE, \derL\calF\}\) and \(\{\calC_i \suchthat i=1,\ldots,n-1\}\)  of \(\der^\nabla\calQ(\compn)\) which commute with each other and which on the Grothendieck group level give the actions of \(\Uqgl\) and of the Hecke algebra \(\ucalH_n\) on \(V^{\otimes n}\) respectively.
\end{theorem}

\nocite{MR573434,MR1029692}
\providecommand{\bysame}{\leavevmode\hbox to3em{\hrulefill}\thinspace}
\providecommand{\MR}{\relax\ifhmode\unskip\space\fi MR }
\providecommand{\MRhref}[2]{%
  \href{http://www.ams.org/mathscinet-getitem?mr=#1}{#2}
}
\providecommand{\href}[2]{#2}

\end{document}